\documentclass[11pt]{article}
\usepackage[usenames]{color} 
\usepackage[applemac]{inputenc}
\usepackage{amssymb} 
\usepackage{amsmath} 
\usepackage{mathabx}
\usepackage{amsthm}
\usepackage[all]{xy}
\usepackage{enumerate}
\usepackage[T1]{fontenc}
\usepackage[left=2cm,right=2cm,top=2cm,bottom=2cm]{geometry}
\usepackage{authblk}
\usepackage{cleveref}
\usepackage{filecontents}

\begin{document}
\theoremstyle{plain}
\newtheorem{deftn}{Definition}[section]
\newtheorem{lem}[deftn]{Lemma}
\newtheorem{prop}[deftn]{Proposition}
\newtheorem{thm}[deftn]{Theorem}
\newtheorem{cor}[deftn]{Corollary}
\newtheorem{conj}[deftn]{Conjecture}
\newtheorem{assump}{Assumption}[section]
\renewcommand{\theassump}{\Alph{assump}}

\theoremstyle{definition}
\newtheorem{ex}[deftn]{Example}
\newtheorem{rk}[deftn]{Remark}

\newcommand{\A}{\mathcal{A}}
\newcommand{\C}{\mathcal{C}}
\newcommand{\CQ}{\mathcal{C}_Q}
\newcommand{\Cw}{\mathcal{C}_w}
\newcommand{\yjh}{\hat{y_j}}
\newcommand{\mjh}{\hat{\mu_j}}
\newcommand{\p}{\mathbb{P}}
\newcommand{\ylh}{\hat{y_1}}
\newcommand{\ynh}{\hat{y_n}}
\newcommand{\W}{{\bf W}}
\newcommand{\Aqnw}{\mathcal{A}_q (\mathfrak{n}(w))}
\newcommand{\Yia}{Y_{i,a}}
\newcommand{\M}{{\bf M}}
\newcommand{\G}{{\bf G}}
\newcommand{\s}{\mathcal{S}}

\newcommand{\Address}{ \textsc{ \\ Universit\'e Paris-Diderot, \\ CNRS Institut de Math\'ematiques de Jussieu-Paris Rive Gauche, UMR 7586,}
 \\ B\^atiment Sophie Germain, Bo\^\i te Courrier 7012, \\ 8 Place Aur\'elie Nemours - 75205 PARIS Cedex 13, \\
 E-mail: \texttt{elie.casbi@imj-prg.fr}}

\title{\Large{ {\bf DOMINANCE ORDER AND MONOIDAL CATEGORIFICATION OF CLUSTER ALGEBRAS}}}
  \author{\large{ELIE CASBI}}

     \date{}

\maketitle
    
     \begin{abstract}
We study a compatibility relationship between Qin's dominance order on a cluster algebra $\A$ and partial orderings arising from classifications of simple objects in  a monoidal categorification $\C$ of $\A$.  Our motivating example is Hernandez-Leclerc's monoidal categorification using representations of quantum affine algebras. In the framework of Kang-Kashiwara-Kim-Oh's monoidal categorification via representations of quiver Hecke algebras, we focus on the case of the category $R-gmod$ for a symmetric finite type $A_n$ quiver Hecke algebra using Kleshchev-Ram's classification of irreducible finite-dimensional representations. 
      \end{abstract}
 
 \bigskip

\setcounter{tocdepth}{1}
\tableofcontents

\section{Introduction}

     Cluster algebras were introduced in \cite{FZ1} by Fomin and Zelevinsky in the 2000's to study  total positivity and  canonical bases of quantum groups. Cluster algebras are commutative $\mathbb{Z}$-subalgebras of fields of rational functions over $\mathbb{Q}$ generated by certain distinguished generators, called \textit{cluster variables}, satisfying relations called exchange relations. These cluster variables are grouped into overlapping finite sets of fixed cardinality (the rank of the cluster algebra)  called clusters. A monomial in cluster variables of the same cluster is called a \textit{cluster monomial}.  Applying exchange relations to one variable of a cluster leads to another cluster and this procedure is called \textit{mutation}. Any cluster can be reached from a given initial cluster by a finite sequence of mutations. Moreover, it is shown in \cite{FZ4} that once an initial seed  $((x_1, \ldots, x_n),B)$ has been fixed, any cluster variable of any other seed $((x_1^t, \ldots , x_n^t),B^t)$ can be related to the initial cluster variables in the following way:
   $$   x_l^t = \frac{F^{l,t}  \left( \ylh , \ldots , \ynh \right)}{{F^{l,t}}_{|{\p}} \left( y_1, \ldots, y_n \right)} x_1^{g_1^{l,t}} \cdots x_n^{g_n^{l,t}} $$
   where $ (g_1^{l,t}, \ldots , g_n^{l,t})$ is an $n$-tuple of integers called the $g$-vector of the cluster variable $x_l^t$ and $F^{l,t}$ is a polynomial called the $F$-polynomial of $x_l^t$. The variables $ y_j, \yjh$, $1 \leq j \leq n$ are Laurent monomials in the initial cluster variables $x_1, \ldots , x_n$ which do not depend on $l$ and $t$. Thus the combinatorics of a cluster algebra is entirely contained in the behaviour of $g$-vectors and $F$-polynomials.
     
      In \cite{BZ}, Berenstein and Zelevinsky defined quantum cluster algebras as quantizations of cluster algebras; these algebras  are not commutative: cluster variables belonging to the same cluster $q$-commute, i.e.  satisfy  relations of the form  $ x_i x_j = q^{\lambda_{ij}} x_j x_i $ for some integers $ \lambda_{ij} $.

 The notion of monoidal categorification of a cluster algebra has been introduced by Hernandez and Leclerc in \cite{HL}. The idea is to identify a given cluster algebra $\A$ with the  Grothendieck ring of a monoidal category; more specifically, the categories involved in this procedure will be categories of modules over algebras such as  quantum affine algebras. Monoidal categorification requires a correspondence between cluster monomials and \textit{real} simple objects in the category (i.e. simple objects $S$ such that $S \otimes_{\C} S$ is again simple) as well as between cluster variables and \textit{prime} real simple objects (real simple objects that cannot be decomposed as tensor products of  several non trivial objects). Given a cluster algebra $\A$, the existence of a monoidal categorification of $\A$ can give some fruitful information about the category itself, such as the existence of decompositions of simple objects into tensor products of prime simple objects. 
  In \cite{HL}, Hernandez and Leclerc define a sequence $ \{ {\C}_l , l \in \mathbb{N} \}$   of subcategories of modules over quantum affine algebras and conjecture that the category $\C_1$ is a monoidal categorification of a cluster algebra. They prove this conjecture in types $A_n$ and $D_4$ and exhibit a remarkable link between the $g$-vector (resp. $F$-polynomial) of each cluster variable and the dominant monomial (resp. the (truncated) $q$-character) of the corresponding simple module. In \cite{HL2}, Hernandez-Leclerc introduce other subcategories of finite dimensional representations of quantum affine algebras and prove several monoidal categorification statements for these categories in types $A_n$ and $D_n$.

   Other examples of monoidal categorification of cluster algebras appeared in various contexts. For instance, in the study of categories of sheaves on the Nakajima varieties: these varieties were constructed by Nakajima in \cite{Naka1} for the purpose of providing a geometric realization of quantum groups as well as their highest weight representations (\cite{Naka1,N}). In \cite{Naka}, certain categories of perverse sheaves on these varieties are shown to be monoidal categorifications of cluster algebras. In particular, Nakajima used these constructions to study the category $\C_1$ and proved Hernandez-Leclerc's conjecture in $ADE$-types. Recently, Cautis and Williams exhibited in \cite{CW} a new example of monoidal categorification of cluster algebras using the $\mathbb{G}_m$-equivariant coherent Satake category, i.e. the category of $G(\mathcal{O})$-equivariant perverse coherent sheaves on the affine Grassmannian $Gr_G$. In the case of  the general linear group $GL_n$, they show that this category is a monoidal categorification of a quantum cluster algebra and construct explicitly an initial seed.

 A large proportion of this paper will be devoted to another situation of monoidal categorification of (quantum) cluster algebras that came out of the works of Kang-Kashiwara-Kim-Oh \cite{KKK,KKKO3,KKKO}, involving categories of representations of quiver Hecke algebras. 
 Introduced by  Khovanov and Lauda in  \cite{KL} and independently by Rouquier in \cite{R}, quiver Hecke algebras (or KLR algebras) are  $\mathbb{Z}$-graded algebras which  categorify the negative part $U_q(\mathfrak{n})$ of the quantum group  $U_q(\mathfrak{g})$, where $\mathfrak{g}$ is a symmetric Kac-Moody algebra and $\mathfrak{n}$ the nilpotent subalgebra arising from a triangular decomposition.  
 Khovanov-Lauda conjectured that this categorification provides a bijection between  the canonical basis of $U_q(\mathfrak{n})$ and the set of indecomposable projective modules over the quiver Hecke algebra corresponding to $\mathfrak{g}$. 
   This was proved by Rouquier in \cite{R} and independently by  Varagnolo-Vasserot in \cite{VV} using a geometric realization of quiver Hecke algebras. 
     The category of finite-dimensional modules over quiver Hecke algebras can be given a monoidal structure using parabolic induction. In \cite{KR}, Kleshchev and Ram give a combinatorial  classification of simple  finite-dimensional modules over quiver Hecke algebras of finite types (i.e. associated with a finite type Lie algebra $\mathfrak{g}$) using Lyndon bases. More precisely, they introduce a certain class of simple modules, called \textit{cuspidal modules}, which are in bijection with the set of positive roots of $\mathfrak{g}$. Then the simple modules are realized as quotients of products of cuspidal modules, and are parametrized by \textit{dominant words} or \textit{root partitions} (see Definition~\ref{domiword}).

     Fix a total order on the set of vertices of the Dynkin diagram of $\mathfrak{g}$. The corresponding lexicographic order is a total ordering on the set of dominant words.  In Section~\ref{mutrule}, we use this framework for quiver Hecke algebras of finite type $A_n$ and exhibit an easy combinatorial way to compute the highest dominant word appearing in the decomposition of the product of classes of two simples into a sum of classes of simples. For $n \geq 1$, consider $\mathfrak{g}$ a Lie algebra of finite type $A_n$. Denote by $r_n=n(n+1)/2$ the number of positive roots for the root system associated to $\mathfrak{g}$. Let $R-gmod$ denote the category of finite-dimensional representations of the quiver Hecke algebra arising from $\mathfrak{g}$ and let $\M$  be the set of dominant words. For every $\mu \in \M$ we let $L(\mu)$ denote the unique (up to isomorphism) simple object in $R-gmod$ corresponding to $\mu$. For any $\mu,\mu' \in \M$, we define $\mu \odot \mu'$ as the highest word appearing in the decomposition of the product $[L(\mu)] \cdot [L(\mu')]$ as a sum of classes of simple objects in $R-gmod$. This is well-defined as $\M$ is totally ordered. 
 
 \begin{thm}[cf. Theorem~\ref{thmonoid}]
  The law $\odot$ provides $\M$ with a monoid structure and there is an isomorphism of monoids 
  $$(\M,\odot) \simeq (\mathbb{Z}_{\geq 0}^{r_n},+).$$
  Moreover this isomorphism is explicitly constructed.
 \end{thm}

     In \cite{KKKO}, Kang, Kashiwara,  Kim and Oh adapt the notion of monoidal categorification to the quantum setting  (in particular they define a notion of quantum monoidal seed) and prove that the category $R-gmod$ gives a  monoidal categorification of the quantum cluster algebra structure on $U_q(\mathfrak{n})$.  For this purpose, they introduce  in \cite{KKK} some R-matrices for categories of finite dimensional representations of quiver Hecke algebras, which give rise to exact sequences corresponding to cluster mutations in the Grothendieck ring. The notion of admissible pair is introduced in \cite{KKKO} as a sufficient condition for a quantum monoidal seed to admit mutations in every exchange direction (see Definition~\ref{KKKOdef6.1}). The main result of \cite{KKKO} consists in proving that given an initial seed coming from an admissible pair, the seeds obtained after any mutation again come from admissible pairs. Hence the existence of an admissible pair implies monoidal categorification statements. The main results of \cite{KKKO} (\cite[Theorems 11.2.2,11.2.3]{KKKO}) consist in constructing admissible pairs for certain subcategories $\Cw$ (for each $w$ in the Weyl group of $\mathfrak{g}$) of finite dimensional representations of quiver Hecke algebras. These categories thus provide monoidal categorifications of cluster algebra structures on the quantum coordinate rings $\Aqnw$ introduced by  Geiss, Leclerc and Schr\"oer in \cite{GLS}. We refer to Kashiwara's 2018 ICM talk \cite{KICM} for a survey on monoidal categorifications of cluster algebras and connections with the theory of crystal bases and in particular global bases of quantum coordinate rings.

    In this framework, it is natural to consider the dominant word of the simple module corresponding to a cluster variable $x_l^t$ and try to relate it to the dominant words of the simple modules belonging to an initial monoidal seed, for instance the seed arising from the construction of \cite[Theorem 11.2.2]{KKKO}. As mentioned above, the theory of cluster algebras encourages us to look at the variables $\yjh$ defined in \cite{FZ4}. Using the above result on dominant words, we can associate in a natural way analogs of  dominant words that we call \textit{generalized parameters} to each of these variables $\yjh$ (with respect to the initial seed constructed in \cite{KKKO}). It turns out that in the case of the category $R-gmod = {\C}_{w_0}$ (where $w_0$ stands for the longest element of the Weyl group of $\mathfrak{g}$), these generalized parameters share some remarkable properties. In particular, it implies some correspondence between the lexicographic ordering on dominant words and the following ordering on Laurent monomials in the initial cluster variables: for any two such Laurent monomials $ {\bf x}^{\boldsymbol{\alpha}} = \prod_i x_i^{\alpha_i}$ and $ {\bf x}^{\boldsymbol{\beta}} = \prod_i x_i^{\beta_i}$, one sets
    $$  {\bf x}^{ \boldsymbol{\alpha}} \preccurlyeq {\bf x}^{\boldsymbol{\beta}} 
      \Leftrightarrow \exists (\gamma_1 , \ldots , \gamma_n) \in \mathbb{Z}_{\geq 0}^n \quad \text{such that} \quad 
       {\bf x}^{\boldsymbol{\beta}} = {\bf x}^{\boldsymbol{\alpha}} . \prod_{j} {\yjh}^{\gamma_j}.$$ 
  This ordering can be easily seen to coincide with the \textit{dominance order} for the considered initial seed, introduced by Qin in \cite{Qin} as an ordering on multidegrees. As proven by Hernandez-Leclerc in \cite{HL}, this dominance order corresponds to the Nakajima order on monomials in the case of the monoidal categorification by the category $\C_1$ (see Proposition~\ref{HLcompat}). Geometrically, it is related to the partial ordering on subvarieties of Nakajima varieties, defined as inclusions of subvarieties into the closures of others (see \cite{Naka}). Qin uses this order to introduce the notions of pointed elements and pointed sets (see Definition~\ref{pointed}) and define triangular bases in a (quantum) cluster algebra. As an application, he proves that in the context of monoidal categorification of cluster algebras using representations of quantum affine algebras, cluster monomials  correspond to classes of simple modules, which partly proves Hernandez-Leclerc's conjecture \cite{HL}. 
 
    In this paper we study this relationship between orderings in a more general setting; considering the data of a cluster algebra $\A$ together with a monoidal categorification $\C$ of $\A$, we assume the simple objects in $\C$ are parametrized by elements of a partially ordered set $\M$. Given a seed $\s$  in $\A$, we define a notion of compatibility between the ordering on $\M$ and the  dominance order $\preccurlyeq$ associated to $\s$ and we say that $\s$ is a \textit{compatible seed} if this compatibility holds (see Definition~\ref{compat}). Not all seeds are compatible and it even seems that most seeds are not. We conjecture that under some technical assumptions on the category $\C$, there exists a compatible seed (Conjecture~\ref{conjecture}). The existence of such a seed $\s$ implies some combinatorial relationships between the $g$-vector with respect to $\s$ of any cluster variable on the one hand, and the parameter of the corresponding simple object in $\C$ on the other hand. The results of Hernandez-Leclerc in \cite{HL} provide a beautiful example of such a relationship.

   We then focus on the case of monoidal categorifications of cluster algebras via representations of quiver Hecke algebras of finite type $A_n$. More precisely, we consider the category $R-gmod$ of finite-dimensional representations of a type $A_n$ quiver Hecke algebra. One of the main results of this paper consists in the explicit computation of the dominant words of the simple modules of the initial (quantum) monoidal seed constructed in \cite{KKKO} for this category.

   \begin{thm}[cf. Theorem~\ref{initpara}]
 Let ${\s}_0^n$  be the initial seed constructed in \cite{KKKO}  for the category $R-gmod$ associated with a Lie algebra of type $A_n$.
 Then  the cluster variables of the seed ${\s}_0^n$  can be explicitly described in terms of dominant words as follows:

 $ \begin{array}{ccccc}
  [L(1)]  & {} & {} & {} & {} \\\relax
  [L(12)] & [L((2)(1))] & {} & {} & {} \\\relax
  [L(123)] & [L((23)(12))] & [L((3)(2)(1))] & {} & {} \\\relax
  \vdots & \vdots & {} & {} & {} \\\relax
  [L(1 \ldots k)] & [L \left( (2 \ldots k)(1 \ldots k-1) \right)] & \cdots & [L((k) \cdots (1))] & {} \\\relax
  \vdots  & \vdots & {} & {} & {} \\\relax
  [L(1 \ldots n)] & [L \left( (2 \ldots n)(1 \ldots n-1) \right)]  & \cdots & \cdots &  [L((n) \cdots (1))].  
\end{array}$

 The set of frozen variables corresponds to the last line and the set of unfrozen variables consists in the union of lines $1, \ldots , n-1$.
 \end{thm}

  Using this description, we can deduce the main result of this paper: 
  
   \begin{thm}[cf. Theorem~\ref{mainthm}]
   The seed ${\s}_0^n$ is compatible in the sense of Definition~\ref{compat}.
 \end{thm} 
 
 In particular, Conjecture~\ref{conjecture} holds for the category of finite dimensional representations of  a quiver Hecke algebra arising from a Lie algebra of type $A_n$.
 
 \bigskip
    
 This paper is organized as follows: in Section~\ref{reminders} we recall the definitions and main results of the theory of cluster algebras from \cite{FZ4}, as well as the notion of monoidal categorification of cluster algebras from \cite{HL}. Section~\ref{sectionqha} is devoted to the representation theory of quiver Hecke algebras. After some reminders of their definitions and main properties, we recall the constructions of Kang-Kashiwara-Kim-Oh (\cite{KKK,KKKO3,KKKO})  of renormalized R-matrices for modules over quiver Hecke algebras as well as their results on monoidal categorification of quantum cluster algebras. We also recall the construction of admissible pairs for the categories $\Cw$ from \cite{KKKO}. We end the section with the results of Kleshchev-Ram \cite{KR} about the classification of irreducible finite-dimensional representations  of finite type quiver Hecke algebras. In Section~\ref{dom-compat}, we consider the general situation of monoidal categorifications of cluster algebras. We introduce a partial ordering on Laurent monomials in the cluster variables of a given seed and show that it coincides with the dominance order introduced by Qin in \cite{Qin}. Then we define the notion of admissible seed and state our main conjecture (Conjecture~\ref{conjecture}). We show that the results of Hernandez-Leclerc \cite{HL} in the context of monoidal categorification of cluster algebras via quantum affine algebras provide an example where this conjecture holds. In Section~\ref{mutrule} we focus on the case of monoidal categorifications of cluster algebras via finite type $A_n$ quiver Hecke algebras. Section~\ref{KLRtechnique} is devoted to the proof of Theorem~\ref{thmonoid}. We use it in the framework of \cite{KKKO} to obtain a combinatorial rule of transformation of dominant words under cluster mutation. Then we study in detail the example of a quiver Hecke algebra of type $A_3$ and exhibit examples of compatible and non-compatible seeds in the corresponding category $R-gmod$. In Section~\ref{initseed} we state and prove the two main results of this paper (Theorems~\ref{initpara} and ~\ref{mainthm}). We conclude with some possible further developments. 

\bigskip

 \textit{The author is supported by the European Research Council under the European Union's Framework Programme H2020 with ERC Grant Agreement number 647353 Qaffine.}

\smallskip

\section*{Acknowledgements}

 I would like to warmly thank my advisor Professor David Hernandez for his constant support, encouragements, kindness, and very helpful discussions. I also thank Fan Qin for answering my questions about his work \cite{Qin} and Myungho Kim for pointing out a need of further explanations about the monoid structure on $\M$ in Section~\ref{genepara}. I am also very grateful to Hironori Oya for carefully reading this paper and making many valuable suggestions.

\section{Cluster algebras and their monoidal categorifications}
\label{reminders}
  
   In this section, we recall the main results of the theory of cluster algebras from \cite{FZ1}, \cite{FZ2}, \cite{FZ4}, as well as the notion of monoidal categorification from \cite{HL}. 
  
 \subsection{Cluster algebras}
     \label{algclust}
 
	 Cluster algebras were introduced in \cite{FZ1} by Fomin and Zelevinsky. They are commutative $\mathbb{Z}$-subalgebras of the field of rational functions over $\mathbb{Q}$ in a finite number of algebraically independent variables. They are defined as follows.
	   
	Let $1 \leq n < m $  be two nonnegative integers and let $\mathcal{F}$ be the field of rational functions over $\mathbb{Q}$ in $m$ independent variables. 
	The initial data is a couple  $((x_1, \ldots ,x_m),B)$ called an initial seed and made out of a \textit{cluster} i.e.  $m$ algebraically independent variables $x_1, \ldots ,x_m$ generating  $\mathcal{F}$ and a $ m \times n $ matrix $B := (b_{ij})$  called the \textit{exchange matrix} whose principal part  (i.e. the  square submatrix $(b_{ij})_{1 \leq i,j \leq n} $) is skew-symmetric.

  To the exchange matrix $B$ one can associate a quiver, whose index set is $\{1, \ldots , m \} $ with $b_{ij}$ arrows from $i$ to $j$ if $b_{ij} \geq 0 $, and $-b_{ij}$ arrows from $j$ to $i$ if $b_{ij} \leq 0 $. 
  
  By construction, one can recover the exchange matrix from the data of a quiver without loops and $2$-cycles in the following way:
    $$ b_{ij} = (\text{number of arrows from $i$ to $j$}) - (\text{number of arrows from $j$ to $i$}). $$
   
\bigskip

 For any  $k \in \{1, \ldots, n \}$ one defines new variables  :
 \begin{equation}   \label{mut} 
  x'_j := \begin{cases} 
   \frac{1}{x_k} \left( \prod\limits_{\substack{b_{lk}>0}}   x_l^{b_{lk}}  +  
   \prod\limits_{\substack{b_{lk}<0}}  x_l^{-b_{lk}} \right) 
    &\text{ if  $j=k$, } \\
        x_j & \text{ if  $j \neq k$,  } 
   \end{cases} 
\end{equation} 

  as well as a new matrix $B'$ :
 $$ \forall i,j  \quad (B')_{ij} := \begin{cases}  - b_{ij} & \text{ if $i=k$ or $j=k$, } \\ b_{ij} + \frac{ \lvert b_{ik} \rvert b_{kj} + b_{ik} \lvert b_{kj} \rvert } {2} & \text{ if $ i \neq k $ and $j \neq k$. }  \end{cases} $$
 Note that the principal part of the  matrix  $B'$ is again skew-symmetric.
 
  \bigskip
 
 The procedure producing the seed $((x'_1, \ldots ,x'_m),B')$ out of the initial seed $((x_1, \ldots ,x_m),B)$ is called the \textit{mutation} in the direction $k$ of the initial seed $((x_1, \ldots ,x_m),B)$. This procedure is involutive, i.e. the mutation of the seed  $(x'_1, \ldots, x'_m),B')$ in the same  direction $k$ gives back the initial seed $((x_1, \ldots , x_m),B)$.
  Any seed can give rise to $n$ new seeds, each of them obtained by a mutation in the direction  $k$ for  $1 \leq k \leq n$.  Let $\mathbb{T}$ be the tree whose vertices correspond to the seeds and edges to mutations. There are exactly $n$ edges adjacent to each vertex.  This tree can have a finite or infinite number of vertices depending on the initial seed.  Let $((x_1, \ldots ,x_m),B)$ be a fixed seed and $t_0$ the corresponding vertex in $\mathbb{T}$. For any vertex  $t \in \mathbb{T} $ one denotes by $((x_1^t, \ldots ,x_m^t),B^t)$ the seed corresponding to the vertex $t$. It is obtained from the initial seed $((x_1, \ldots ,x_m),B)$ by applying a sequence of mutations following a path starting at $t_0$ and ending at $t$.

\begin{deftn} 
The cluster algebra generated by the initial seed 
$((x_1, \ldots ,x_m), \allowbreak B)$ is the  $\mathbb{Z}[x_{n+1}^{\pm 1}, \ldots ,x_m^{\pm 1}]$-subalgebra of  $\mathcal{F}$ generated by  all the variables 
$ x_1^t, \allowbreak \ldots , \allowbreak x_n^t$ for all the vertices $t \in \mathbb{T}$.
\end{deftn}

 For any seed $((x_1, \ldots ,x_m),B)$, the variables $x_1, \ldots , x_m$  are called the cluster variables, $x_1, \ldots , x_n$ are  the unfrozen variables and $x_{n+1}, \ldots , x_m$ are the frozen variables. These last variables do not mutate and are present in every  cluster.

The first main result of the theory of cluster algebras is the \textit{Laurent phenomenon} :

   \begin{thm}[{{\cite[Theorem 3.1]{FZ1}}}]
  Let $((x_1, \ldots ,x_m),B)$ be a fixed seed in a cluster algebra $\A$. Then for any seed  
   $((x_1^t,  \allowbreak  \ldots  \allowbreak ,x_m^t),B^t)$ in $\A$ and any $1 \leq j \leq n$,  the cluster variable  $x_j^t$ is a Laurent polynomial in the variables $x_1, \ldots , x_m$. 
  \end{thm}

\bigskip
 
  Let $\p$ be the multiplicative group of all Laurent monomials in the frozen variables $x_{n+1}, \ldots , x_{m}$. One can endow it with an additional structure given by 
    $$ \prod_{i} x_i^{\alpha_i} \oplus \prod_{i} x_i^{\beta_i} := \prod_{i} x_i^{min(\alpha_i,\beta_i)}$$
  making $\p$ a semifield. 
Any substraction-free rational expression $F(u_1, \allowbreak \ldots , \allowbreak u_k) $ with integer coefficients in some variables $u_1, \ldots , u_k $ can be specialized on some elements $p_1, \ldots, p_k$ in $\p$. This will be denoted by $F_{|{\p}}(p_1, \ldots, p_k)$.

 The mutation relation~\eqref{mut} can be rewritten as 
$$ x_k x'_k = p_k^{+} \prod_{1 \leq i \leq n} x_i^{ [b_{ik}]_{+}}  +  p_k^{-} \prod_{1 \leq i \leq n} x_i^{[-b_{ik}]_{+}} $$
where $$p_k^{+} := \prod_{n+1 \leq i \leq m} x_i^{ [b_{ik}]_{+}} \quad  \text{and} \quad  p_k^{-} := \prod_{n+1 \leq i \leq m} x_i^{[-b_{ik}]_{+}}$$ belong to the semifield $\p$.
 
 Thus the frozen variables $x_{n+1} , \ldots , x_m$ play the role of coefficients and the cluster algebra $\A$ can be viewed as the $\mathbb{Z} \p$ algebra generated by the (exchange) variables $x_1^t , \ldots , x_n^t$ for all the vertices $t$ of the tree $\mathbb{T}$. Here $\mathbb{Z} \p$ denotes the group ring of the multiplicative group of the semifield $\p$. This group is always torsion-free and hence the ring $\mathbb{Z} \p$ is a domain. 
 
\bigskip

 The notion of isomorphism of cluster algebras is introduced in \cite{FZ2}: two cluster algebras $\A \subset \mathcal{F}$ and $\mathcal{A'} \subset \mathcal{F'}$ with the same coefficient part $\p$  are said to be isomorphic if there exists a $\mathbb{Z} \p$ algebras isomorphism $\mathcal{F} \rightarrow \mathcal{F'}$ sending a seed in $\A$ onto a seed in $\mathcal{A'}$. In particular the set of seeds of $\A$ is in bijection with the set of seeds of $\mathcal{A'}$ and $\A$ and $\mathcal{A'}$ are isomorphic as algebras. 

  The second important result is the classification of finite type cluster algebras, i.e. the ones with a finite number of seeds. 

\begin{thm}[{{\cite[Theorem 1.4]{FZ2}}}] 
There is a canonical bijection between isomorphism classes of cluster algebras of finite type and Cartan matrices of finite type.  
\end{thm}

\bigskip

Let $\A$ be a cluster algebra and let us fix $((x_1, \ldots , x_m), B)$ an initial seed. 
  In \cite{FZ4}, Fomin and Zelevinsky define, for any $1 \leq j  \leq n$:
  $$ 
   y_j   :=  \prod_{ n+1 \leq i \leq m} x_i^{b_{ij}}   \quad and \quad 
   \yjh := \prod_{ 1 \leq i \leq m} x_i^{b_{ij}}.
   $$
   
  \begin{thm}[{{\cite[Corollary 6.3]{FZ4}}}] \label{thmFpoly}
  Let $((x_1^t, \ldots , x_n^t , x_{n+1}, \ldots,  x_m) , B^t)$ be any seed in $\A$. Then for any $1 \leq l \leq n$, the cluster variable $x_l^t$ can be expressed in terms of the initial cluster variables $x_1 , \ldots , x_m$ in the following way:
\begin{equation}  \label{Fpoly}  x_l^t = \frac{F^{l,t}  \left( \ylh , \ldots , \ynh \right)}{{F^{l,t}}_{|{\p}} \left( y_1, \ldots, y_n \right)} x_1^{g_1^{l,t}} \cdots x_n^{g_n^{l,t}}.
\end{equation}

 \end{thm}

     In this formula $F^{l,t}$ is a polynomial called the $F$-polynomial associated to the variable $x_l^{t}$ and the $g_i^{l,t}$ are integers. We write for short 
 $ {\bf x}^{{\bf g}^{l,t}} $ for $x_1^{g_1^{l,t}} \cdots x_n^{g_n^{l,t}}$ and ${\bf g}^{l,t} = (g_1^{l,t} , \cdots , g_n^{l,t})$ is called the $g$-vector associated to the variable $x_l^t$. 
 
    The $F$-polynomial associated to any cluster variable satisfies several important and useful properties, which have been conjectured by Fomin-Zelevinsky  in \cite{FZ4} and proved by Derksen-Weyman-Zelevinsky in \cite{DWZ} using the theory of quivers with potentials. We recall here some of these results, which we will use in Section~\ref{dom-compat} in the study of compatible seeds. 
    
    \begin{thm}[{{\cite[Theorem 1.7]{DWZ}}}] \label{FpolyDWZ}
   
     Let $((x_1^t, \ldots , x_n^t , x_{n+1}, \ldots,  x_m) , B^t)$ be any seed in $\A$. Let  $1 \leq l \leq n$, and $F^{l,t}$  be the $F$-polynomial associated to the cluster variable $x_l^t$. Then 
     \begin{enumerate}[(i)]
      \item There is a unique monomial in $F^{l,t}$ that is strictly divisible by any other monomial in $F^{l,t}$. This monomial has coefficient $1$. 
      \item The polynomial $F^{l,t}$ has constant term $1$. 
      \end{enumerate}
 
  \end{thm}

 \subsection{Monoidal categorification of cluster algebras}
  \label{moncat}
 
     The notion of monoidal categorification of a cluster algebra was introduced by Hernandez and Leclerc in \cite{HL}. Recall that, if $\C$ is a monoidal category, a simple object $M$ in $\C$ is said to be \textit{real} if the tensor product $ M \otimes_{\C} M$ is simple. It is said to be \textit{prime} if it is not invertible in $\C$ cannot be decomposed as $ M = M_1 \otimes_{\C} M_2$ with  $M_1$ and $M_2$ two simple non invertible modules neither trivial nor equal to $M$ itself. We denote by $K_0(\C)$ the Grothendieck ring of the category $\C$.   Recall that for any objects $A,B,C$ in $\C$, the relation $[B]=[A]+[C]$ holds in $K_0(\C)$ if there is a short exact sequence $0 \rightarrow A \rightarrow B \rightarrow C \rightarrow 0$ in $\C$. The ring structure on $K_0(\C)$ is directly inherited from the monoidal structure of $\C$: $[M] \cdot [M'] = [M \otimes_{\C} M']$ for any objects $M,M'$ in $\C$. 
     
  \begin{rk}
      In the category $R-gmod$ that we will mostly be studying in this paper, all the  simple objects are non-invertible. However in other categories, simple objects may be invertible. This happens for instance in categories of modules over Borel subalgebras of quantum affine algerbras in the work of Hernandez-Leclerc  \cite{HL4}.
      \end{rk}

\begin{deftn}[Monoidal categorification of a cluster algebra]
A monoidal category $\C$ is a monoidal categorification of a cluster algebra $\A$ if the following conditions simultaneously hold:
     
       \begin{enumerate}[(i)]
        \item There is a ring isomorphism 
          $$ K_0 (\C) \simeq \A. $$     
         \item Under this isomorphism, classes of simple real objects in $\C$ correspond to cluster monomials in $\A$ and classes of simple real prime objects in $\C$ correspond to cluster variables in $\A$. 
        \end{enumerate}
\end{deftn}

 \smallskip
 
   Several examples of monoidal categorifications of cluster algebras appeared more recently in various contexts: on the one hand using categories of (finite dimensional) representations of quiver Hecke algebras through the works of Kang-Kashiwara-Kim-Oh \cite{KKK,KKKO3,KKKO}; on the other hand, via the coherent Satake category studied by Cautis-Williams in \cite{CW}. Let us point out that these examples use a slightly different notion of monoidal categorification:
   
    \begin{deftn}[Monoidal categorification of a cluster algebra in the sense of \cite{KKKO,CW}]
    A monoidal category $\C$ is a monoidal categorification of a cluster algebra $\A$ if:
    
   \begin{enumerate}[(i)]
        \item There is a ring isomorphism 
          $$ K_0 (\C) \simeq \A. $$     
         \item Under this isomorphism, any cluster monomial in $\A$ is the class of a simple real object in $\C$. 
        \end{enumerate}
\end{deftn}  

 See for instance Definition~\ref{defmoncatKKKO} for a precise definition in the context of quiver Hecke algebras.

\subsection{Example: representations of quantum affine algebras}
  \label{qaa}

 Let $\mathfrak{g}$ be a finite-dimensional semisimple Lie algebra of type $A_n, D_n,$ or $E_n$ and $\hat{\mathfrak{g}}$ the corresponding Kac-Moody algebra. The quantum affine algebra  $U_q (\hat{\mathfrak{g}})$ can be defined as a quantization of the universal enveloping algebra of $\hat{\mathfrak{g}}$ (see \cite{Drin} or \cite{CharPres} for precise definitions). Consider the category $\C$ of finite dimensional $U_q (\hat{\mathfrak{g}})$-modules. In \cite{CharPres}, Chari and Pressley proved that simple objects in this category are parametrized by their highest weights.  More precisely, let $I$ be the set vertices of the Dynkin diagram of $\mathfrak{g}$ and, for each $i \in I$ and $a \in \mathbb{C}^{*}$, let $Y_{i,a}$ be some indeterminate. The notion of $q$-character of a  finite dimensional $U_q ( \hat{\mathfrak{g}})$-module was introduced  by Frenkel and Reshetikhin in \cite{FR1} as a  an injective ring homomorphism 
$$ \chi_q : K_0 (\C) \rightarrow \mathbb{Z}[Y_{i,a}^{\pm 1} , i \in I , a \in \mathbb{C}^{*} ].$$
Let $\mathcal{M}$ be the set of Laurent monomials in the variables $\Yia$.  
 For any $i \in I$ and $a \in \mathbb{C}^{*}$, set
$$ A_{i,a} := Y_{i,aq} Y_{i,aq^{-1}} \prod_{j \neq i} Y_{j,a}^{a_{ji}}  \in \mathcal{M} . $$
One defines a  partial ordering (the Nakajima order) on $\mathcal{M}$ in the following way:
$$  \mathfrak{m} \leq \mathfrak{m}' \Leftrightarrow \frac{\mathfrak{m}'}{\mathfrak{m}} \text{ is a monomial in the $A_{i,a}$ } $$
for any monomials  $\mathfrak{m}, \mathfrak{m}' \in  \mathcal{M}$. 

A monomial $\mathfrak{m} \in \mathcal{M}$ is called dominant if it does not contain negative powers of the variables $\Yia$. Let ${\mathcal{M}}^{+}$ denote the subset of $\mathcal{M}$ of all dominant monomials. For any simple object $V$ of $\C$,  the set of monomials occurring in the $q$-character of $V$ has a unique maximal element $\mu_V$ for the above order, and this monomial is always dominant. Conversely, it is possible to associate a simple finite dimensional  $U_q ( \hat{\mathfrak{g}})$-module to any dominant monomial in the variables $\Yia$, providing a bijection between the set of simple objects in $\C$ and $\mathcal{M}^{+}$. 
For any dominant monomial $\mathfrak{m}$, we let $L(\mathfrak{m})$ denote the unique (up to isomorphism) simple object in $\C$  corresponding to $\mathfrak{m}$ via this bijection. In the case where $\mathfrak{m}$ is reduced to a single variable $Y_{i,a}$ for some $i \in I$ and $a \in \mathbb{C}^{*}$, the simple module $L(\mathfrak{m})=L(Y_{i,a})$ is called a \textit{fundamental representation}.

 \bigskip

 The Dynkin diagram of $\mathfrak{g}$ is a bipartite graph hence its vertex set $I$ can be decomposed as $I = I_0 \sqcup I_1$ such that every edge connects a vertex of $I_0$ with one of $I_1$. Then for any $i \in I$ set:
$$ \xi_i :=  \begin{cases}  0 &\text{if $i \in I_0$,} \\ 1 &\text{ if $i \in I_1$.}  \end{cases} $$
Hernandez and Leclerc introduce a subcategory $\mathcal{C}_1$ of $\C$ whose Grothendieck ring is generated (as a ring) by the classes of the  fundamental representations $L(Y_{i,q^{\xi_i}}) , L(Y_{i,q^{ \xi_i +2}}) (i \in I)$. 
One of the main results of \cite{HL} can be stated as follows: 

  \begin{thm}[{{\cite[Conjecture 4.6]{HL}}}]
The category $\C_1$ is a monoidal categorification of a (finite type) cluster algebra of the same Lie type as the Lie algebra $\mathfrak{g}$.
 \end{thm}

They prove this conjecture for $\mathfrak{g}$ of types $A_n$ ($n \geq 1$) and $D_4$ (\cite[Sections 10,11]{HL}). In \cite{Naka}, Nakajima proved this conjecture in types $ADE$ using geometric  methods involving graded quiver varieties. Note that this geometric construction is valid for any orientation of the Dynkin graph of $\mathfrak{g}$. In \cite{HL2}, Hernandez-Leclerc exhibit other examples of monoidal categorifications of cluster algebras via categories of representations of quantum affine algebras in types $A_n$ and $D_n$ (\cite[Theorem 4.2, Theorem 5.6]{HL2}).

 \section{Quiver Hecke algebras}
  \label{sectionqha}

 The works of Kang-Kashiwara-Kim-Oh \cite{KKK,KKKO3,KKKO} provide many examples of monoidal categorifications of cluster algebras arising from certain categories of modules over quiver Hecke algebras. In this section, we recall the main definitions and properties of quiver Hecke algebras; then we recall the constructions of renormalized $R$-matrices from \cite{KKK} as well as the statements of monoidal categorification from \cite{KKKO}. We also recall the classification of simple finite dimensional representations of quiver Hecke algebras  of finite type using combinatorics of Lyndon words from \cite{KR}. 

 \subsection{Definition and main properties} 
   \label{qha}

In this subsection we recall the definitions and main properties of quiver Hecke algebras, as defined in \cite{KL} and \cite{R}.
 
 Let $\mathfrak{g}$ be a Kac-Moody algebra, $P$ the  associated weight lattice and $\Pi = \{ \alpha_i , i \in  I \} $ the set of simple roots.  We also define the coweight lattice as  $P^{\vee} = Hom(P, \mathbb{Z}) $ and we let $\Pi^{\vee}$ denote the set of simple coroots. We also denote by $A$ the generalized Cartan matrix, $W$ the Weyl group of $\mathfrak{g}$, and $(\cdot , \cdot)$ a $W$-invariant symmetric bilinear form on $P$. Let ${\bf k}$ be a base field.

 The root lattice is defined as  $ Q := \bigoplus_i \mathbb{Z}\alpha_i $. We also set $ Q_{+} := \bigoplus_i \mathbb{Z}_{\geq 0}\alpha_i $ and $ Q_{-} := \bigoplus_i \mathbb{Z}_{\leq 0}\alpha_i $.
For any $\beta \in Q $ that we write  $ \sum_i m_i \alpha_i $, its length is defined as  $ \sum_i |m_i| $. When $\beta \in Q_{+}$, we denote as in \cite{KL} by $ Seq(\beta) $ the set of all finite sequences (called words) of the form  $ i_1, \ldots, i_n $ (where $n$ is the length of $\beta$) with $m_i$ occurrences of the integer $i$ for all $i$. For the sake of simplicity, we identify letters with simple roots. In particular, for any $i,j \in \{1, \ldots, n \}$, $(i,j)$ stands for $(\alpha_i,\alpha_j)$ and if $\mu=i_1, \ldots, i_n$ and $\nu=j_1, \ldots,j_m$ are two words in $Seq(\beta)$, $(\mu,\nu)$ stands for $\sum_{p,q} (i_p,j_q)$.
 
 \bigskip
 
  To define quiver Hecke algebras, we fix a nonnegative integer $n$ and a family $\{Q_{i,j} , 1 \leq i,j \leq n \}$ of two-variables polynomials with coefficients in ${\bf k}$. These polynomials are required to satisfy certain properties, in particular $Q_{i,j}=0$ if $i=j$ and $Q_{i,j}(u,v)=Q_{j,i}(v,u)$ for any $i,j$ (see for instance \cite[Section 2.1]{KKKO} for more details). 
  In the case of finite type $A_n$ symmetric quiver Hecke algebras (which we will focus on in Sections~\ref{mutrule} and ~\ref{initseed}), the polynomials $Q_{i,j}$ are the following (see \cite{KR}):
$$ Q_{i,j}(u,v) = \begin{cases}  (u-v) &\text{if $j=i+1$,} \\  (v-u) &\text{ if $j=i-1$, } \\ 0 & \text{ if $i=j$, } \\ 1 &\text{otherwise.}   \end{cases} $$

  \begin{deftn}
  
 For any $\beta$ in $ Q_{+}$ of length $n$, the quiver Hecke algebra  $R(\beta)$ at $\beta$ associated to the Kac-Moody algebra $\mathfrak{g}$ and the family $\{Q_{i,j} , 1 \leq i,j \leq n \}$ is the ${\bf k}$-algebra generated by operators $ \{e(\nu)\}_{\nu \in Seq(\beta)} , \{x_i\}_{i \in \{1, \ldots ,n\} } $ , and $ \{\tau_k\}_{k \in \{1, \ldots ,n-1 \} } $ satisfying the following relations :
 
  \begin{equation*}
\begin{gathered}
 e(\nu)e(\nu') = \delta_{\nu,\nu'} e(\nu),   \\  \sum_{\nu \in Seq(\beta)} e(\nu) = 1, \\  x_i x_j = x_j x_i , \\  x_i e(\nu) = e(\nu) x_i,  \\  \tau_k e(\nu) = e(s_k(\nu)) \tau_k , \\ 
\tau_k \tau_l = \tau_l \tau_k  \text{ if $ \lvert k-l \rvert > 1$} , \\  \tau_k^2 e(\nu) =  Q_{\nu_k,\nu_{k+1}}(x_k,x_{k+1}) e(\nu) , \\ 
(\tau_k x_i - x_{s_k(i)} \tau_k ) e(\nu) = { \begin{cases} -e(\nu) & \text{ if $ i=k , \nu_k=\nu_{k+1} $, } \\  e(\nu) & \text{ if $ i=k+1 , \nu_k=\nu_{k+1} $ } \\ 0 & \text{ otherwise} \end{cases} },  \\ 
(\tau_{k+1}\tau_{k}\tau_{k+1} - \tau_{k}\tau_{k+1}\tau_{k}) e(\nu) = 
{ \begin{cases} \frac{Q_{\nu_k,\nu_{k+1}}(x_k,x_{k+1}) - Q_{\nu_k,\nu_{k+1}}(x_{k+2},x_{k+1})}{x_k - x_{k+2}} e(\nu) & \text{ if $ \nu_k = \nu_{k+2} $, } \\ 0 & \text{otherwise,} \end{cases} } 
\end{gathered}
\end{equation*}
where for any $\nu \in Seq(\beta)$, $\nu_k$ stands for the $k$th letter of the word $\nu$. 
  \end{deftn}
 
  The quiver Hecke algebra  $R(\beta)$ is called symmetric when the polynomials $Q_{i,j}$ are polynomials in $u-v$. 
 
The first main property of quiver Hecke algebras is that they naturally come with a  $\mathbb{Z}$-grading by setting 
$$ \deg e(\nu) = 0, \quad \deg x_k e(\nu) = 2 , \quad \deg \tau_i e(\nu) = -(\nu_i , \nu_{i+1}). $$
 
\bigskip
 
 For any   $\beta$ and $\gamma$ in $Q_{+}$ of respective lengths $m$ and $n$, let $M$ be a $R(\beta)$ module  and $N$ a  $R(\gamma)$ module. One defines the  convolution product of $M$ and $N$  via parabolic induction (see \cite{KL,KKK}).
 
 Set  
$$e(\beta,\gamma):= \sum_{\nu \in Seq(\beta) \atop \lambda \in Seq(\gamma)} e(\nu \lambda)  \in R(\beta+\gamma) .$$ 
It is an  idempotent in $R(\beta+\gamma)$.
Consider the  homomorphism of  $\mathbf{k}$-algebras 
 $$ R(\beta) \otimes R(\gamma) \rightarrow e(\beta,\gamma)R(\beta+\gamma)e(\beta,\gamma)$$ 
 given by   
 $$e(\nu) \otimes e(\lambda) \mapsto e(\nu \lambda), \nu \in Seq(\beta), \lambda \in Seq(\gamma) $$
 $$ x_k \otimes 1 \mapsto x_k e(\beta,\gamma), 1 \leq k \leq m , 1 \otimes x_l \mapsto x_{m+l} e(\beta,\gamma), 1 \leq l \leq n $$
 $$ \tau_k \otimes 1 \mapsto \tau_k e(\beta,\gamma), 1 \leq k < m , 1 \otimes \tau_l \mapsto \tau_{m+l} e(\beta,\gamma), 1 \leq l < n .$$ 
 Then one defines 
 $$ M \circ N := R(\beta + \gamma) \otimes_{ R(\beta) \otimes R(\gamma) } {M \otimes N }. $$
For any $\beta \in Q_{+}$, let $R(\beta)-pmod$ be the category of (left) graded finite type projective  $R(\beta)$ modules,  $R(\beta)-gmod$ the category of left finite dimensional graded $R(\beta)$-modules, and also    
$$ R-pmod := \bigoplus_{\beta \in Q_{+}} R(\beta)-pmod , \quad  R-gmod := \bigoplus_{\beta \in Q_{+}} R(\beta)-gmod .$$
Convolution product induces a monoidal structure on the categories $R-gmod $ and $R-pmod$.
The grading on quiver Hecke algebras also yields  shift functors  for these categories: decompose any object  $M$  as 
$$ M = \bigoplus_{n \in \mathbb{Z}} M_n $$ 
and define $qM$ as
 $$ qM = \bigoplus_{n \in \mathbb{Z}} M_{n-1}. $$ 
 The natural  $\mathbb{Z}[q^{\pm 1}]$  action 
 $$ q \cdot [M] := [qM]  $$ 
gives rise to   $\mathbb{Z}[q^{\pm 1}]$-algebras structures on the Grothendieck rings of the categories $R-pmod$ and $R-gmod$.

\bigskip

 The following definition introduces a notion of graded character for representations of quiver Hecke algebras. 

     \begin{deftn}[\cite{KL,KR}] \label{caract}
Let $M$ be a finite dimensional graded $R(\beta)$-module. For any $\nu \in Seq(\beta)$, set $M_{\nu} := e(\nu) \cdot M$. The module $M$ can be decomposed as
$$M = \bigoplus_{\nu} M_{\nu} .$$ 
Define 
$$ ch_q (M) := \sum_{\nu} (dim_q M_{\nu}) . \nu $$
where for any graded vector space $ V = \bigoplus_{n \in \mathbb{Z}}  V_n $, $ dim_q (V) := \sum_{n \in \mathbb{Z}} q^n dim V_n .$
 This is a formal series in words belonging to $Seq(\beta)$ with coefficients in $\mathbb{Z}[q,q^{-1}]$. 

 \end{deftn}

One can put a ring structure on the image set of $ch_q$ by defining a "product" of two words called  \textit{quantum shuffle product}.
  For any nonnegative integer $n$, let $\mathfrak{S}_n$ denote the symmetric group of rank $n$. 

     \begin{deftn}[Quantum shuffle product] \label{shuffle}
 Let $ {\bf i} = i_1, \ldots, i_r $ and $ {\bf j} = j_1, \ldots, j_s $ be two words. 
 We set $ i_{r+1} := j_1 $, $ i_{r+s} := j_s $ so that we can consider the concatenation  $ {\bf ij} = i_1 \cdots i_{r+s}$. 

Define the quantum shuffle product of $ {\bf i}$ and ${\bf j}$ :
$$ {\bf i} \circ {\bf j} := \sum_{\sigma \in \mathfrak{S}_{r,s} } q^{-e(\sigma)} ( i_{ \sigma^{-1}(1)}, \ldots, i_{\sigma^{-1}(r+s)} ) $$
 where $ \mathfrak{S}_{r,s} $ denotes the subset of  $\mathfrak{S}_{r+s}$ defined as:
 $$\mathfrak{S}_{r,s} := \{ \sigma \in \mathfrak{S}_{r+s} \mid \sigma(1) < \cdots < \sigma(r)  \quad \text{and} \quad   \sigma(r+1) < \cdots < \sigma(r+s) \}$$
  and, for any element $\sigma \in \mathfrak{S}_{r,s}$, 
$$ e(\sigma) :=  \sum_{ 1 \leq k\leq r< l \leq r+s \atop \sigma(k) > \sigma(l) }  (i_k , i_l) . $$ 

 \end{deftn}

By linearity one can also define quantum shuffle products of two formal series in elements of $Seq(\beta)$ with coefficients in $\mathbb{Z}[q,q^{-1}]$ (for any $\beta \in Q_{+}$). 

     \begin{prop}[{{\cite[Lemma 2.20]{KL}}}]
 For any $ \beta,  \gamma  \in Q_{+} $, and any $M \in R(\beta)-gmod$ and $N \in R( \gamma)-gmod$, we have:
$$ ch_q (M \circ N) = ch_q (M) \circ ch_q (N). $$
  \end{prop} 
  
One can now state  the main property of quiver Hecke algebras, which is to categorify the negative part of the quantum group $U_q(\mathfrak{g})$ in a way that makes correspond the basis of indecomposable objects in $R-pmod$ with the canonical basis of $U_q(\mathfrak{n})$. In the following we will mostly consider the category $R-gmod$ hence we give here the dual statements, involving the category $R-gmod$ and the quantum coordinate ring ${\A}_q(\mathfrak{n})$ (the precise definition of which can be found in \cite{GLS} or \cite{KKKO}). The first theorem was proved by Khovanov-Lauda \cite{KL} and Rouquier \cite{R}. The second was conjectured by Khovanov-Lauda, and proved by Rouquier \cite{R} and Varagnolo-Vasserot \cite{VV} using geometric methods.

 \begin{thm}[Khovanov-Lauda, Rouquier] \label{KL-R}
The map $ch_q$ induces a $\mathbb{Z}[q,q^{-1}]$-algebra isomorphism 
$$  K_{0}(R-gmod) \simeq {\A}_q(\mathfrak{n}) . $$
 \end{thm}
 
   \begin{thm}[Rouquier, Varagnolo-Vasserot] \label{canobasis}
  The map $ch_q$ (see Definition~\ref{caract}) induces a bijection between the canonical basis of the quantum coordinate ring ${\A}_q(\mathfrak{n})$ and the set of isomorphism classes of self-dual simple modules in the category $R-gmod$. 
  \end{thm}
 
  \subsection{Renormalized $R$-matrices for quiver Hecke algebras} 
 
 Recall from Section~\ref{qha} that the weight lattice associated to the Kac-Moody algebra $\mathfrak{g}$ is given with a symmetric bilinear form $(\cdot , \cdot)$. Denoting by $A$ the symmetrizable generalized Cartan matrix of $\mathfrak{g}$, this bilinear form is entirely determined by its values on simple roots, namely:
 $$  \forall i,j, \quad (\alpha_i,\alpha_j) = {\bf s}_i a_{ij} $$
 where the ${\bf s}_i$ are the entries of a diagonal matrix $D$ such that $DA$ is symmetric.

  One also defines another symmetric bilinear form $(\cdot , \cdot)_n$ on the root lattice $Q$ as in \cite{KKK}:
  $$ \forall i,j, \quad  (\alpha_i,\alpha_j)_n := 
   \begin{cases}
      1 &\text{ if $i=j$, } \\
      0 &\text{otherwise.}
   \end{cases} $$
  Let  $\beta \in Q_{+} $ of length  $m$ and $ 1 \leq k < m $; the following operators $\varphi_k$ are introduced in \cite{KKK}:
 $$ \forall \nu \in Seq(\beta), \varphi_k e(\nu) :=  \begin{cases} (\tau_k (x_k - x_{k+1}) + 1)e(\nu) & \text{if $\nu_k = \nu_{k+1}$,} \\ \tau_k e(\nu) & \text{otherwise.} \end{cases}  $$
These operators satisfy the braid relation, hence for any permutation  $\sigma$, $\varphi_{\sigma}:=\varphi_{i_1} \cdots \varphi_{i_l}$ does not  depend  on the choice of a reduced expression  $\sigma=s_{i_1} \cdots s_{i_l}$.

 For any $ m,n \in \mathbb{Z}_{\geq 0} $, let $w[m,n]$ be the element of $\mathfrak{S}_{m+n}$ sending  $k$ on $k+n$ if $1 \leq k \leq m$ and on  $k-m$ if $m< k \leq m+n$.

\bigskip

Consider a non-zero $R(\beta)$-module $M$ and a non-zero $R(\gamma)$-module $N$. The following map is  defined in \cite{KKK}:
$$\begin{array}{ccccc}
   & M \otimes N & \longrightarrow & N \circ M \\
   & u \otimes v & \longmapsto & \varphi_{w[n,m]}(v \otimes u).
    \end{array} $$
  It is  $R(\beta) \otimes R(\gamma)$ linear and hence induces a   homomorphism of $R(\beta + \gamma)$-modules 
   $$ R_{M,N} : M \circ N \longrightarrow N \circ M. $$
The map $R_{M,N}$ satisfies the Yang-Baxter equation (see \cite{KKK}).

 Let $z$ be an  indeterminate, homogeneous of degree $2$.  For any $\beta \in Q_{+}$ and any non-zero module $M$ in $R(\beta)-gmod$, one defines $M_z := {\bf k}[z] \otimes M$ with the following  ${\bf k}[z] \otimes R(\beta)$-module structure:
 
\begin{equation*}
\begin{gathered}
 $$ e(\nu) . (P \otimes m) := P \otimes (e(\nu)m) \\
 x_k . (P \otimes m) := (zP) \otimes m  +  P \otimes ({x_k}m) \\
 \tau_k . (P \otimes m) := P \otimes ({\tau_k}m)
  $$
 \end{gathered}
\end{equation*}
for any  $\nu \in Seq(\beta)$, $P \in {\bf k}[z]$ and $m \in M$.

\smallskip
\smallskip

It is shown in \cite{KKK} that for any $\beta, \gamma \in Q_{+}$ and any non-zero $R(\beta)$-module $M$ and non-zero $R(\gamma)$-module $N$, the map  $R_{M_z,N}$ is polynomial in  $z$ and does not vanish. Let $s$ be the largest non-negative integer such that the image of $R_{M_z,N}$ is contained in $z^s N \circ M_z$. One defines R-matrices in the category $R-gmod$ in the following way:

\begin{deftn} \label{renormRmat}

 Let $\beta, \gamma \in Q_{+}$. For any non zero $R(\beta)$-module $M$ and non zero $R(\gamma)$-module $N$, define a homomorphism of $R(\beta + \gamma)$-modules
$$ r_{M,N} : M \circ N \longrightarrow N \circ M $$
by setting 
 $$ r_{M,N} := (z^{-s}R_{{M_z},N})_{\mid z=0} $$
 where $s$ is the integer defined above.
 
  \end{deftn}
  
  \begin{prop}[\cite{KKK}]
  The homomorphism $r_{M,N}$ does not vanish and satisfies the Yang-Baxter equation. 
  \end{prop}

Thus the maps $r_{M,N}$ are  R-matrices for the category $R-gmod$.
They are called renormalized R-matrices. As in the case of categories of representations of quantum affine algebras, these R-matrices are in general not invertible and thus yield (graded) short exact sequences in the category $R-gmod$. Consequently this produces some relations  in the Grothendieck ring $K_0(R-gmod) \simeq \mathcal{A}_q(\mathfrak{n})$. In the context of monoidal categorifications of (quantum) cluster algebras (see Section~\ref{moncatKLR} below), the exchange relations in $\mathcal{A}_q(\mathfrak{n})$ will be identified with some of these relations. 

The corresponding relations in the Grothendieck ring $K_0(R-gmod)$ will be identified with exchange relations   
 For any non-zero modules $M$ and $N$, we denote by $ \Lambda(M,N) $ the homogeneous  degree of the  morphism  $r_{M,N}$. It is given by 
 $$  \Lambda(M,N)  = -(\beta, \gamma) + 2(\beta,\gamma)_n -2s . $$
 The next statement gives a criterion for the renormalized R-matrix $r_{M,N}$ to be an isomorphism. It will be particularly useful for the proof of Theorem~\ref{initpara} (see for instance Corollary~\ref{Lambda}).

\begin{lem}[{{\cite[Lemma 3.2.3]{KKKO}}}] \label{comutlambda}
Let $M$ and $N$ be two simples in the category $R-gmod$ and assume one of them is real.
 Then the following are  equivalent :

 \begin{enumerate}[(i)]
\item $\Lambda(M,N) + \Lambda(N,M) = 0 $.
\item $r_{M,N}$ and $r_{N,M}$ are inverse to each other up to a constant multiple.
\item $M \circ N$ and $N \circ M$ are isomorphic up to grading shift.
\item $M \circ N$ is simple in the category $R-gmod$. 
 \end{enumerate}
\end{lem}

One says that  $M$ and $N$   commute  if they satisfy these  properties.

 \subsection{Monoidal categorification via representations of quiver Hecke algebras}
  \label{moncatKLR}

 In this subsection we focus on the case where $\C$ is a full subcategory of $R-gmod$ stable under convolution products, subquotients, extensions, and grading shifts. $\C$ can be decomposed as  
$$ \C = \bigoplus_{\beta \in Q_{+}} \mathcal{C}_{\beta} $$
with $\mathcal{C}_{\beta}  := \C \cap R(\beta)-gmod$ for every $\beta \in Q_{+}$,
 so that the tensor product in $\mathcal{C}$ sends $\mathcal{C}_{\beta} \times \mathcal{C}_{\gamma} $ onto $\mathcal{C}_{\beta + \gamma}$ for any $\beta , \gamma \in Q_{+}$.

  Kang-Kashiwara-Kim-Oh \cite{KKKO} adapt the notion of monoidal categorification to the setting of quantum cluster algebras. 
  In the classical setting, a monoidal seed in $\C$ is defined as a triple $(\{M_i\}_{1 \leq i \leq n} , B , D)$ where  $\{M_i\}_{1 \leq i \leq n}$ is a collection of simple objects in $\C$ such that for any  $i_1, \ldots ,i_t$ in $\{1, \ldots , n\}$ , the object  $M_{i_1} \circ \cdots  \circ M_{i_t}$ is simple in $\mathcal{C}$, $B$ is an integer-valued matrix with skew-symmetric principal part and $D$ is a diagonal matrix encoding the weights of the modules $M_i$ (i.e. the elements $\beta_i \in Q_{+}$ such that $M_i \in \C_{\beta_i}$). Cluster mutations correspond to some (ungraded) short exact sequences in the category $\C$. These exact sequences come from the failure of the renormalized R-matrices (see Definition~\ref{renormRmat}) to be isomorphisms. The cluster mutations being involutive imposes some relations between the entries of the matrices $B$ and $D$.
  
 \smallskip
 
  In  the framework of \cite{KKKO}, one takes into account the natural grading of quiver Hecke algebras defined in Section~\ref{qha}: objects in $\C$ are graded as well. A quantum monoidal seed is the data of such a triple  $(\{M_i\} , B, D)$ with the further assumption that there exist integers $\lambda_{ij}$ and isomorphisms of graded modules $M_i \otimes M_j \simeq q^{\lambda_{ij}} M_j \otimes M_i$ for any  $i,j \in \{1, \ldots ,n\}$. The matrix $L=(\lambda_{ij})_{1 \leq i,j \leq n}$ is a skew-symmetric matrix and is assumed to satisfy some compatibility relations with the matrix $B$ as in \cite{BZ}. See \cite[Section 6.2.1]{KKKO} for a precise definition.  

In the quantum setting, cluster mutations correspond to some \textit{graded} short exact sequences.

\begin{deftn}[{{\cite[Definition 6.2.3]{KKKO}}}]
 Let $k \in \{1, \ldots ,r \}$ fixed. A quantum monoidal seed  $\mathcal{S} = (\{M_i\}_{1 \leq i \leq n} , \allowbreak  L \allowbreak , B, D)$ admits a mutation in the direction $k$ if there exists a simple object  $M'_k$ of $\mathcal{C}$ such that:
\begin{enumerate}[(a)]
\item $M'_k \in \mathcal{C}_{d'_k}$ with $d'_k := -d_k + \sum_{b_{ik}>0} b_{ik}d_i$. 
\item One has the following short exact sequences in  $\mathcal{C}$ : 
 $$ 0 \longrightarrow q M^{\mathbf{b'}} \longrightarrow q^{m_k} M_k \otimes M'_k \longrightarrow M^{\mathbf{b''}} \longrightarrow 0 $$
$$ 0 \longrightarrow q M^{\mathbf{b''}} \longrightarrow q^{m'_k} M'_k \otimes M_k \longrightarrow M^{\mathbf{b'}} \longrightarrow 0 $$
where $m_k$ and $m'_k$ are some integers. 
\item $ \mathcal{S}'^{(k)} := (\{M_i\}_{i \neq k} \cup \{M'_k\} ,  L'^{(k)}, B'^{(k)}, D'^{(k)})$ is again a quantum monoidal seed in $\mathcal{C}$,
where $L'^{(k)}$ and $B'^{(k)}$ are defined as in \cite[Definition 3.5]{BZ} and $D'^{(k)}$ is the diagonal matrix whose entries are the $d_i$ for $i \neq k$ and $d'_k$ for $i=k$.
\end{enumerate}
\end{deftn}

\begin{deftn} \label{defmoncatKKKO}
 The category $\mathcal{C}$ is a monoidal categorification of a quantum cluster algebra  $\A$ if:
\begin{enumerate}[(a)]
\item There is an isomorphism of graded rings $\mathbb{Z}[q^{\pm \frac{1}{2}}] \otimes_{\mathbb{Z}[q^{\pm 1}]} {K_0}(\mathcal{C}) \simeq \A $. 
\item There exists a quantum monoidal seed  $\mathcal{S} := ( \{M_i\}, L,B,D)$ in $\mathcal{C}$ such that $[\mathcal{S}] := (\{q^{- \frac{(d_i,d_i)}{4} }[M_i]\},L,B)$ is a quantum seed in  $\A$. 
\item The quantum monoidal seed  $\mathcal{S}$ admits arbitrary sequences of mutations in all directions.
\end{enumerate}
\end{deftn}

\bigskip

In this setting, the existence of a monoidal categorification implies that any  (quantum) cluster monomial is the class of some real simple object in $\C$.
Recall from Section~\ref{moncat} that this notion is slightly different from the notion of monoidal categorification initially defined by Hernandez-Leclerc \cite{HL}.
 
  The following definition provides a sufficient condition for producing quantum monoidal seeds. 
  
  \begin{deftn} \label{KKKOdef6.1}
  
  A pair $(\{M_i\},B)$ is admissible if:
  \begin{enumerate}[(i)]
  \item $ \{M_i\}_{1 \leq i \leq n}$ is a family of self-dual real simple modules  commuting with each other.
  \item The matrix $B$ is defined as above.
  \item For each $ 1 \leq k \leq r$ there exists a self-dual simple module  $M'_k$  such that  $M'_k$ commutes with the  $M_i$ for $i \neq k$ and there is a short exact sequence of graded objects in $\C$
$$ 0 \longrightarrow q M^{\mathbf{b'}} \longrightarrow q^{ \tilde{\Lambda} ( M_k,M'_k)} M_k \circ M'_k \longrightarrow M^{\mathbf{b''}} \longrightarrow 0. $$
  where $\tilde{\Lambda}(M,N)$ is defined as $ \frac{1}{2} ( \Lambda(M,N) + (\beta,\gamma) ) $ for $M \in R(\beta)-gmod$ and $N \in R(\gamma)-gmod$.
  \end{enumerate}
  \end{deftn}

 The data of an admissible pair naturally gives rise to a quantum monoidal seed in  $\mathcal{C}$. More precisely, if $(\{M_i\}_{1 \leq i \leq n}, B)$ is an admissible pair in $\mathcal{C}$, $M'_k$ as in the previous definition, then one defines a $r \times r$ skew-symmetric matrix $L$ and a diagonal matrix $D$ of size $n$ by setting
$$L_{ij} := \Lambda(M_i,M_j) \quad \text{and} \quad  D=Diag(d_1, \ldots ,d_n)$$
where $d_i$ stands for the weight of the module $M_i$. Then (\cite[Proposition 7.1.2]{KKKO}) the quadruple $\mathcal{S} := ( \{M_i\}_{1 \leq i \leq n}, -L,B,D)$ is a quantum monoidal seed in $\mathcal{C}$ which admits mutations in every direction $k$ (for $1 \leq k \leq r$). 

\smallskip
   
  The main result of  \cite{KKKO} can now be stated as follows: 
  
 \smallskip
    
   Let $(\{M_i\}_{1 \leq i \leq n}, B)$ be an admissible pair in  $\mathcal{C}$ and $$ \mathcal{S} := ( \{M_i\}_{1 \leq i \leq n}, -L,B,D) $$
the corresponding quantum monoidal seed. Set $ [\mathcal{S}] := ( \{q^{- \frac{(wt(M_i),wt(M_i))}{4} } [M_i]\}_{1 \leq i \leq n} , -L,B,D)$.

    \begin{thm}[{{\cite[Theorem 7.1.3]{KKKO}}}] \label{6.3KKKO}
   
   Assume there is a $\mathbb{Q}(q^{\frac{1}{2}})$-algebras isomorphism 
   $$ \mathbb{Q}(q^{\frac{1}{2}}) \otimes_{\mathbb{Z}[q^{\pm 1}]} {K_0}(\mathcal{C}) \simeq  \mathbb{Q}(q^{\frac{1}{2}}) \otimes_{\mathbb{Z}[q^{\pm 1}]} {\mathcal{A}_{q^{\frac{1}{2}}}}([\mathcal{S}]). $$
     Then for each  $1 \leq k \leq r$ the pair $(\{M_i\}_{i \neq k} \cup \{M'_k\}, B'^{(k)})$ is again an  admissible pair in the category $\mathcal{C}$.
   
    \end{thm}

\subsection{Quantum monoidal seeds for $\Cw$} 
 \label{qmseedcw}

  In this subsection we recall  from \cite{KKKO} the definition of the subcategories $\Cw$ of $R-gmod$  as well as the construction of  admissible pairs for these categories. 
  
   For any element $w$ of the Weyl group $W$ associated to $\mathfrak{g}$, Geiss, Leclerc and Schr\"oer defined  algebras $\Aqnw$ as subalgebras of the quantum coordinate rings ${\A}_q(\mathfrak{n})$ (\cite[Section 7.2]{GLS}). They show  (\cite[Theorem 12.3]{GLS}) that it is possible to put a quantum cluster algebra structure on $\Aqnw$ for every $w \in W$. In \cite{KKKO}, Kang-Kashiwara-Kim-Oh introduce, for each $w \in W$, a subcategory $\Cw$ of $R-gmod$ such that  the Grothendieck ring $K_0(\Cw)$ is the preimage of $\Aqnw$ under the isomorphism given by Theorem~\ref{KL-R}: $M \in \Cw$ if and only if $ch_q(M) \in \Aqnw$. In \cite{KKKO}, Kang-Kashiwara-Kim-Oh prove the following:
 
  \begin{thm}[{{\cite[Theorem 11.2.3]{KKKO}}}] \label{categow}
 For each element $w$ of the Weyl group $W$, the category $\Cw$ is a monoidal categorification of the quantum cluster algebra $\mathcal{A}_{q^{1/2}}(\mathfrak{n}(w))$.
 \end{thm}
Thus the categories $\Cw$ provide many examples of monoidal categorifications of (quantum) cluster algebras. 

\begin{rk}
 The category $R-gmod$ corresponds ${\C}_{w_0}$ where $w_0$ stands for the longest element of the Weyl group of $\mathfrak{g}$.  When $w$ is the square of a well chosen Coxeter element $c$ in $W$, the quantum cell $\mathcal{A}_{q^{1/2}}(\mathfrak{n}(w))$ is also categorified by the category $\mathcal{C}_1$ defined in \cite{HL}. The category ${\C}_{w_0}$ (resp. ${\C}_{c^2}$) is related to the category $\CQ$ (resp. $\mathcal{C}_1$) introduced by Hernandez-Leclerc in \cite{HL2} (resp. \cite{HL}) via a functor called \textit{generalized quantum affine Schur-Weyl duality} defined in \cite{KKK}. In the case of ${\C}_{w_0}$ Fujita \cite{Fujita}  proved that this functor is an equivalence of categories. 
\end{rk} 

\smallskip

 Note that Geiss-Leclerc-Schr\"oer defined categories $\tilde{\mathcal{C}}_w$ which provide additive categorifications of the quantum coordinate rings $\Aqnw$ for each $w \in W$ (\cite[Theorem 12.3]{GLS}). The categories $\tilde{\mathcal{C}}_w$ are defined as subcategories of the preprojective algebra of certain quivers. The categories $\Cw$ as defined in \cite{KKKO} can be seen as monoidal analogs of the categories $\tilde{\mathcal{C}}_w$ of \cite{GLS} in terms of representations of quiver Hecke algebras. However, Theorem~\ref{categow} provides a monoidal categorification statement and is thus of different nature than the results of \cite{GLS}. 
 
 \smallskip
 \smallskip
 
 In order to prove Theorem~\ref{categow}, Kang-Kashiwara-Kim-Oh construct an admissible pair (see Definition~\ref{KKKOdef6.1}) in the category $\Cw$ for each $w \in W$. We now recall this construction. By the results of \cite{KKKO}, the existence of such a pair implies Theorem~\ref{categow}.
 
 First one defines \textit{unipotent quantum minors} as some distinguished elements of ${\A}_q(\mathfrak{n})$: for any dominant weight $\lambda$ in the weight lattice $P$ and any couple $(\mu,\zeta)$ of elements of $W \lambda$, the unipotent quantum minor $D(\mu, \zeta)$ is an element of ${\A}_q(\mathfrak{n})$ which is either a member of the canonical basis of ${\A}_q(\mathfrak{n})$ or zero (\cite[Lemma 9.1.1]{KKKO}). The following statement gives a necessary and sufficient condition so that $D(\mu,\zeta)$ is non zero. First recall some notation from \cite{KKKO}:
 
  \begin{deftn}
  Let $\lambda \in P^{+} , \mu, \zeta \in W \lambda$. We write $\mu \lesssim \zeta$ if there exists a finite sequence $(\beta_1 , \ldots , \beta_l)$ such that, setting $\lambda_0 := \zeta , \lambda_k = s_{\beta_k} \lambda_{k-1} (1 \leq k \leq l) $ one has $\lambda_l = \mu$ and $ \forall 1 \leq k \leq l , (\beta_k , \lambda_{k-1}) \geq 0$.
  \end{deftn}

  \begin{lem}[{{\cite[Lemma 9.1.4]{KKKO}}}]
  Let $\lambda \in P^{+} , \mu, \zeta \in W \lambda$. Then $D(\mu , \zeta) \neq 0$ if and only if $\mu \lesssim \zeta$.
  \end{lem}
  
  The following statement is a direct consequence of Theorem~\ref{canobasis} and of the previous lemma:

\begin{cor} \label{determimod}
  Let $\lambda \in P^{+} , \mu, \zeta \in W \lambda$ such that $\mu \lesssim \zeta$.
  There exists a unique self-dual simple module $M(\mu, \zeta) \in R-gmod$ whose image under the character map $ch_q$ is $D(\mu,\zeta)$. Moreover, $M(\mu,\zeta)$ is real. 
  \end{cor}
  
  This module is called \textit{determinantial module} (\cite[Definition 10.2.1]{KKKO}). Its weight is equal to $\zeta - \mu$, i.e. $M(\mu,\zeta) \in R(\zeta - \mu)-gmod$. 
 
 \begin{rk}  
  This is one of the key points that we will use to compute the dominant words of the modules corresponding to the frozen variables in $R-gmod$ in Section~\ref{initseed}.
  \end{rk}
  
  \smallskip
  
  One can now construct an admissible seed for the category $ \mathcal{C}_w$. Fix some element $w$ in the Weyl group $W$ and a reduced expression $w=s_{i_1} \cdots s_{i_r}$. For $s \in \{1, \ldots, r \}$, set 

$ \begin{array}{rcl} 
   s_{+} & := & \min (\{ k \mid s < k \leq r , i_k = i_s \} \cup \{r+1 \} ) \\
   s_{-} & := & \max (\{ k \mid 1 \leq k < s  , i_k = i_s \} \cup \{ 0 \} ) 
 \end{array}$

 For $1 \leq k \leq r $, set 
 $$ \lambda_k := s_{i_1} \cdots s_{i_k} \omega_{i_k} .$$
For $0 \leq t \leq s \leq r$, set 
$$ D(s,t) := \begin{cases}  D(\lambda_s , \lambda_t)  &\text{if $0 <t$,} \\ D(\lambda_s , \omega_{i_s}) &\text{if $0=t<s \leq r$,} \\ 1 &\text{if $t=s=0$.}  \end{cases} $$
 
  \begin{deftn}[\cite{KKKO}] \label{determimodbis}
 As in Corollary~\ref{determimod}, consider $M(s,t)$  the unique simple real module (up to shift and isomorphism) such that $ch(M(s,t)) = D(s,t)$ for any $0 \leq s \leq t \leq r$.
\end{deftn} 
 
Set $J= \{1, \ldots , r\} , J_{fr} = \{ k \in J \mid k_{+}= r+1 \} $ and $J_{ex} = J \setminus J_{fr}$. The initial quiver is set to have $J = \{1 , \ldots , r\}$ as set of vertices with the following arrows:
 
$ \begin{array}{ll}
    s \longrightarrow t  & \text{if $1 \leq s < t < s_{+} < t_{+} \leq r+1$ }  \\
   s \longrightarrow s_{-} & \text{ if $1 \leq s_{-} < s \leq r$ }
 \end{array} $

 Denoting by $B$ the corresponding exchange matrix, the main result of \cite{KKKO} can be stated in the following way:

   \begin{thm}[{{\cite[Theorem 11.2.2]{KKKO}}}] \label{initseedCw}
 The pair $(\{ M(k,0)\}_{1 \leq k \leq r} , B )$ is admissible in the category ${\C}_w$. 
 \end{thm}

\subsection{Irreducible representations of quiver Hecke algebras}
 \label{irredKLR}

In this subsection we recall from \cite{KR} the classification of simple finite-dimensional modules over finite type quiver Hecke algebras. The main result (Theorem~\ref{KR} below) is that simple objects in the category $R-gmod$ are parametrized in a combinatorial way by \textit{dominant words}, which are analogs of Zelevinsky's multisegments in the classification of simple representations of affine Hecke algebras of type $A$. As for Lie  algebras, simple modules over quiver Hecke algebras are constructed as quotients of tensor products of some distinguished irreducible representations, called \textit{cuspidal modules} in \cite{KR}.

Choose a labeling  of the vertices of the Dynkin diagram of $\mathfrak{g}$ by $I=\{1,\ldots,n\}$.
A word is a finite set of  elements of $I$.  
We fix a total order  on $I$ by setting $1< \cdots <n$. The set of all words is a totally ordered set with respect to the lexicographic order induced by $<$. 

  For $ {\bf i } := (i_1,\ldots,i_d) $, set $ |{\bf i }| := \alpha_1 + \cdots + \alpha_d  \in Q_{+} $. Recall from Section~\ref{qha} that for any $ \beta \in Q_{+} $, $Seq(\beta) = \{ {\bf i},  |{\bf i}| = \beta \}$.

     \begin{deftn}
 A word is called Lyndon if it is smaller than all its proper right factors. 
  \end{deftn} 

     \begin{ex}
 The words $123$, $24$, $13$ are Lyndon. The word $231$ is not. 
 \end{ex}

\smallskip
\smallskip

The following statement is a well-known fact (see \cite[Theorem 5.1.5]{Lo}):

     \begin{prop}[Canonical factorization] \label{canofacto}
 Any word $\mu$ can be written in a unique way as a concatenation of Lyndon words in the decreasing order :
 $$ \mu = ({\bf i}^{(1)})^{n_1} \cdots ({\bf i}^{(r)})^{n_r} $$
with ${\bf i}^{(1)}, \ldots ,{\bf i}^{(r)}$ Lyndon words satisfying ${\bf i}^{(1)} > \cdots > {\bf i}^{(r)}$ and $n_1, \ldots ,n_r$ nonnegative integers. 
  \end{prop} 

This is called the \textit{canonical factorization} of the word $\mu$. 
 Recall from Section~\ref{qha} (Definition~\ref{caract}) that for $\beta \in Q_{+} $, any $ R(\beta) $-module $M$ decomposes as a direct sum of vector spaces $ M = \bigoplus_{\nu \in Seq(\beta)} M_\nu $ with $M_{\nu} := e(\nu) M $. 
 
    \begin{deftn} \label{domiword}
 A word $\mu$ is dominant if there is an $R(\beta)$-module $M$ such that $\mu$ is the highest word among the words $\nu$ such that $ M_{\nu}$ is not zero: $ M = M_{\mu} \oplus \bigoplus_{\nu < \mu} M_{\nu} $ and $M_\mu \neq 0$. 
   \end{deftn} 

Dominant words play the same role as highest weights in the representation theory of finite dimensional semisimple Lie algebras (see \cite{CharPres}).
The next  statement provides a very useful combinatorial criterion to determine whether a word is dominant or not. In particular it shows that a dominant word can be seen as a collection (or a sum with positive coefficients) of positive roots, which is why the terminology \textit{root partitions} is sometimes used (see \cite{McN}).

     \begin{thm}[\cite{KR}]
     \label{domfacto}

 \begin{enumerate}[(i)]
   \item  There is a bijection between the set of dominant Lyndon words and the set $\Delta_{+}$ of positive roots of $\mathfrak{g}$, given by ${\bf i} \mapsto |{\bf i}|$. 
   \item A word $\mu$ is dominant if and only if all the Lyndon words appearing in the canonical factorization of $\mu$ are dominant. 
\end{enumerate}

\end{thm}

     \begin{ex}
In type $A_4$, $24$ is Lyndon but not dominant Lyndon, and $123$ is dominant Lyndon. 
The word $ 12312 $ is dominant but the word $ 3213 $ is not. 
\end{ex}

 \begin{rk}
 Dominant Lyndon words already appear in the work of Leclerc \cite{L} as \textit{good Lyndon words} in the study of dual canonical basis for quantum groups and quantum coordinate rings. 
 \end{rk} 
 
\smallskip
\smallskip

 The next Proposition introduces the notion of \textit{cuspidal modules}. The existence of cuspidal modules follows from results of Varagnolo-Vasserot \cite{VV} in simply-laced cases and Rouquier \cite{R} in the general case. Cuspidal modules can be seen as analogs of fundamental representations for Lie algebras.

     \begin{prop}[{{\cite[Proposition 8.4]{KR}}}]
For any dominant Lyndon word ${\bf i}$ in $Seq(\beta)$, there is a unique (up to isomorphism and shift) irreducible $R(\beta)$-module of highest weight ${\bf i}$. We denote it by $L({\bf i})$.  
 \end{prop}  

In \cite{KR}, Kleshchev-Ram give explicit constructions of cuspidal modules for each finite type. For instance, in type $A_n$, the set of positive roots is $\Delta_{+} = \{ \alpha_i + \cdots + \alpha_j  , 1 \leq i \leq j \leq n \} $  and the cuspidal module $L(k \ldots l)$ corresponding to the positive root $\alpha_k + \cdots + \alpha_l$  is the one dimensional vector space spanned by a vector $v$ with action of $R(\alpha_k + \cdots + \alpha_l)$ given by:
$$  x_i \cdot v = 0 , \quad  \tau_j \cdot v = 0 , \quad 
 e(\nu) \cdot v = \begin{cases}
     v &\text{if $ \nu =k \ldots l$} \\
     0 &\text{otherwise.} 
     \end{cases} $$

\smallskip
\smallskip
  
  Recall from Section~\ref{qha} that the graded character of a finite dimensional $R(\beta)$-module $M$ is a (finite) formal sum of elements of $Seq(\beta)$ with coefficients in $\mathbb{Z}[q,q^{-1}]$.  For any such formal sum $S := \sum_{\lambda} P_{\lambda}(q)  \lambda $, we  let $ max(S)$ denote the greatest word appearing in this sum (for the lexicographic order). In particular, for any finite dimensional $R(\alpha)$-module $M$, we set  $max(M) := max(ch_q (M))$. The word $max(M)$ is called the \textit{highest weight} of the module $M$ in \cite{KR}.

 The next statement shows that any word (not necessarily dominant) always appears as the highest word in the quantum shuffle product of the Lyndon words appearing in its canonical factorization. 
 
 \begin{prop}[{{\cite[Lemma 5.3]{KR}}}] \label{KRprod}
 Let $\mu$ be a  word, and $\mu = {\bf i}^{(1)} \cdots {\bf i}^{(r)}$ its canonical factorization.
  Then we have 
  $ max({\bf i}^{(1)} \circ \cdots \circ {\bf i}^{(r)})  = \mu. $
   \end{prop}

One can now state the main result of \cite{KR}. It shows that finite dimensional irreducible representations of finite type quiver Hecke algebras are parametrized by dominant words.

     \begin{thm}[{{\cite[Theorem 7.2]{KR}}}]
      \label{KR} 
Let $\mu$ be a dominant word, and $ \mu = ({\bf i}^{(1)})^{n_1} \cdots ({\bf i}^{(r)})^{n_r} $ its canonical factorization. 
Set $$ \Delta(\mu) := L({\bf i}^{(1)})^{ \circ n_1} \circ \cdots \circ L({\bf i}^{(r)})^{ \circ n_r} \langle s(\mu) \rangle  $$
where $ s(\mu) := \sum_{k=1}^{r} ( {\bf i}_k . {\bf i}_k ) n_k (n_k - 1)/4. $

 Then :
 \begin{enumerate}[(i)]
  \item $\Delta(\mu)$ has an irreducible head, denoted $L(\mu)$.
  \item  The highest weight of $L(\mu)$ is $\mu$: $max(L(\mu)) = \mu$.
  \item  The set $ \{ L(\mu) \} $ for $\mu$ dominant words in $Seq(\beta)$ is a complete and irredundant set of irreducible graded $R(\beta)$-modules up to isomorphism and shift. 
 \end{enumerate}
Moreover for $\mu$ of the form ${\bf j}^n$ with  ${\bf j}$ dominant Lyndon, one has $L(\mu) = L({\bf j})^{\circ n}. $
  \end{thm}
  
    \begin{ex}
  Here are some examples of characters of some simple modules:

$ \begin{array}{lll}
 ch_q (L(1)) &= &(1) \\
ch_q (L(12)) &= &(12) \\
ch_q (L(21)) &= & (21) \\
ch_q(L(312) &= & (312) + (132) \\
ch_q (L(11)) &= & (q + q^{-1}) (11).
 \end{array}$

 \end{ex}

 \section{Dominance order and compatible seeds}
  \label{dom-compat}
 
  In this section we define a partial ordering on the set of Laurent monomials in the cluster variables of a cluster algebra $\A$. This ordering coincides with the \textit{dominance order} introduced by Qin in \cite{Qin}. In the context of monoidal categorification of a cluster algebra, we use this dominance order to introduce the notion of \textit{compatible seed} and state the main conjecture of this work (Conjecture~\ref{conjecture}). We end this section with a discussion of monoidal categorifications of cluster algebras via representations of quantum affine algebras following the work of Hernandez-Leclerc \cite{HL}, which provides a first example where Conjecture~\ref{conjecture} holds. 
 
 \subsection{Partial ordering on monomials}
  \label{dominance}

Consider a cluster algebra $\A$ and choose a seed $((x_1 , \ldots , x_n, x_{n+1} , \ldots , x_m),B)$, where $x_1 , \ldots , x_n$ are the unfrozen variables and $x_{n+1} , \ldots , x_m$ are the frozen variables. Let $\mathcal{M}_{{\bf x}}$ be the monoid of all monomials in the  $x_i$ and $\mathcal{G}_{{\bf x}}$ the abelian group of all Laurent monomials in the  $x_i$. Recall from Section~\ref{algclust} the variables $\yjh$ defined as
$$ \yjh := \prod_{1 \leq i \leq m} x_i^{b_{ij}} $$
for any $1 \leq j \leq n$.
In what follows we write ${\bf x}^{\bf \alpha}$ for  $ \prod_i x_i^{\alpha_i}$ for any integers $\alpha_i$.

  One now defines a partial preorder on $\mathcal{G}_{{\bf x}}$ in the following way: given two Laurent monomials $ {\bf x}^{\boldsymbol{\alpha}} = \prod_i x_i^{\alpha_i}$ and $ {\bf x}^{\boldsymbol{\beta}} = \prod_i x_i^{\beta_i}$ in $\mathcal{G}_{{\bf x}}$, we set 
  $$  {\bf x}^{\boldsymbol{\alpha}} \preccurlyeq {\bf x}^{\boldsymbol{\beta}} $$
  if and only if there exist non-negative integers $ \gamma_j , 1 \leq j \leq n$ such that 
  $$ {\bf x}^{\boldsymbol{\beta}} = {\bf x}^{\boldsymbol{\alpha}} . \prod_{j} {\yjh}^{\gamma_j}.$$ 
 We denote by $\succcurlyeq$ the opposite preorder. 
    
  Assume that the initial exchange matrix $B$ has full rank $n$. Then the preorder $\preccurlyeq$ becomes an order on $\mathcal{G}_{{\bf x}}$. This order is the same as the \textit{dominance order} as defined by F. Qin \cite[Definition 3.1.1]{Qin}.  Indeed, by definition, the relation $\prod_i x_i^{\alpha_i}  \preccurlyeq \prod_i x_i^{\beta_i}$ is equivalent to the existence of nonnegative integers $ \gamma_j , 1 \leq j \leq n$ such that 
$$ \forall i , \beta_i = \alpha_i + \sum_j b_{ij} \gamma_j .$$
 In vector notation this can be rewritten as 
 $$\beta = \alpha + B \gamma $$
 and thus the order $\preccurlyeq$ coincides with Qin's dominance order on multi-indices. 

 Following \cite{Qin}, one can use this dominance order to introduce the notions of \textit{pointed elements} and \textit{pointed sets}.
 
  \begin{deftn}[{{\cite[Definitions 3.1.4, 3.1.5]{Qin}}}] \label{pointed}
  Fix a seed $((x_1 , \ldots , x_n, x_{n+1} , \ldots , x_m),B)$ in $\A$ and assume $B$ has full rank $n$. 
  \begin{enumerate}[(i)]
   \item Let $P$ be a Laurent polynomial in the cluster variables $x_1 , \ldots , x_m$. One says that $P$ is pointed with respect to the seed $((x_1, \ldots , x_m),B)$ if among the monomials of $P$, there is a unique monomial which is a maximal element (for the dominance order $\preccurlyeq$)  and has coefficient $1$. This monomial is called the leading term of $P$ in \cite{Qin}.
    \item Let $L$ be any set of Laurent polynomials in the cluster variables $x_1 , \ldots , x_m$. One says that $L$ is pointed with respect to the seed $((x_1, \ldots , x_m),B)$ if all the elements of $L$ are pointed and two distinct elements of $L$ have different leading terms. 
   \end{enumerate}

\end{deftn}
     
     One can associate a degree to each of the cluster variables $x_i$ (see \cite{Qin}). If $P$ is a pointed element with respect to the seed $((x_1, \ldots , x_m),B)$, then the degree of its leading term can be seen as a generalization of the notion of $g$-vector. 
 
 \subsection{Generalized parameters} 
 \label{genepara}
 
 Let us consider an Artinian monoidal category $\C$ and assume we are given a classification of simple objects in $\C$.  That is, suppose we are given a  poset $(\M, \leq)$ together with  a  bijection $\psi$ between $\M$ and the set  $ {\bf S} :=  \{ [V] , V  \text{simple in $\C$} \} $. Let $L(\mu)$ denote a representative of the isomorphism class in $ K_0(\C) $ corresponding to $ \mu \in \M $ via this bijection. 
 $$ \begin{array}{crcc}
       \psi :      &{\bf S} &\longrightarrow &\M   \\
                &[L(\mu)]   &\longmapsto      &\mu. 
     \end{array}$$
   In what follows, $\mu$ will be referred to as the \textit{parameter} of the simple object $L(\mu)$ in $\C$. 
     We will also assume that the identity object is simple in $\C$.
     
     \bigskip
     
    From now on we assume that the category $\C$ satisfies the following property:
  
  \begin{assump}[Decomposition property]  \label{decomp} 
    let $\mu,\mu' \in \M$ and $L(\mu),L(\mu')$ the corresponding simple objects in $\C$; then the following equality holds in the Grothendieck ring $K_0(\C)$   
    \begin{equation*}
       [L(\mu)] \cdot [L(\mu')] = \sum_{\nu \in {\bf N}_{\mu,\mu'} \subset \M } a_{\nu} [L(\nu)] 
     \end{equation*}
   where ${\bf N}_{\mu,\mu'}$ is a finite subset of $\M$  such that there exists a unique maximal element in ${\bf N}_{\mu,\mu'}$ and $\{a_{\nu}, \nu \in {\bf N}_{\mu,\mu'} \}$ is a family of nonzero integers. The maximal element of ${\bf N}_{\mu,\mu'}$ is denoted by $ \mu \odot \mu' .$
   \end{assump}
 
 \begin{rk}
 In various examples of categories satisfying this property (for instance categories of modules over quiver Hecke algebras or quantum affine algebras), the integer $a_{\mu \odot \mu'}$ happens to be equal to $1$ for any $\mu,\mu' \in \M$, but we will not need this assumption here.
 \end{rk}
 
  In what follows we will need the additional assumption that the law $\odot$ is compatible with the partial ordering on $\M$ in the  following sense:

 \begin{assump}  \label{assump}
\begin{equation*} 
 \forall \mu , \nu \in \M , \mu \leq \nu \Rightarrow \forall \lambda \in \M , (\lambda \odot \mu \leq \lambda \odot \nu \quad \text{and} \quad \mu \odot \lambda \leq \nu \odot \lambda) . 
 \end{equation*}
 \end{assump}
 
 Combining Assumptions~\ref{decomp} and ~\ref{assump} leads to the following:
 
\begin{lem} \label{lemKim} 
  The law $\odot$ on $\M$ is associative. 
  \end{lem}
 
  \begin{proof}
  
 First note that for any $\mu,\mu' \in \M$, the set ${\bf N}_{\mu,\mu'}$ is finite and has a unique maximal element by Assumption~\ref{decomp}, hence this element (namely $\mu \odot \mu'$) is a greatest element in ${\bf N}_{\mu,\mu'}$. 
 Now let $\mu$,$\mu'$ and $\mu''$ in $\M$ and decompose in two different ways the product $[L(\mu)][L(\mu')][L(\mu'')]$. On the one hand, Assumption~\ref{decomp} gives
 $$[L(\mu)][L(\mu')][L(\mu'')] = [L(\mu)] \cdot  \sum_{\nu \in {\bf N}_{\mu',\mu''}} a_{\nu} [L(\nu)] $$
 with $\nu \leq \mu' \odot \mu''$ for every $\nu \in {\bf N}_{\mu',\mu''}$. For any $\nu \in {\bf N}_{\mu',\mu''}$, the parameters appearing in the decomposition of  $[L(\mu)] \cdot  [L(\nu)]$ into classes of simples are all smaller than $\mu \odot \nu$; as $\nu \leq \mu' \odot \mu''$,  Assumption~\ref{assump} implies $ \mu \odot \nu \leq \mu \odot (\mu' \odot \mu'')$. Hence all the parameters involved in the decomposition of $[L(\mu)][L(\mu')][L(\mu'')]$ into classes of simples are smaller than $\mu \odot (\mu' \odot \mu'')$. 
  
   On the other hand, one can write 
 $$[L(\mu)][L(\mu')][L(\mu'')] =   \sum_{\nu \in {\bf N}_{\mu,\mu'}} a_{\nu} [L(\nu)] . [L(\mu'')]$$ and the same arguments show that all the parameters involved in the decomposition of $[L(\mu)][L(\mu')][L(\mu'')]$ into classes of simples are smaller than $(\mu \odot \mu') \odot \mu''$. 
     In particular we get $\mu \odot (\mu' \odot \mu'') \leq (\mu \odot \mu') \odot \mu''$ and 
     $\mu \odot (\mu' \odot \mu'') \geq (\mu \odot \mu') \odot \mu''$ and hence 
    $ \mu \odot (\mu' \odot \mu'') =  (\mu \odot \mu') \odot \mu''$.
    
     \end{proof}
 
 Thus the operation 
 $$  
     \begin{array}{lcr}
       \M \times \M &\rightarrow &\M \\
        (\mu,\mu') &\mapsto &\mu \odot \mu'
      \end{array}
 $$       
   provides $\M$ with a monoid structure. The neutral element $1_M$ is the image via $\psi$ of the class of the identity object of $\C$. 
    By Assumption~\ref{assump}, the monoid $(\M,\odot)$ is an ordered monoid with respect to $\leq$.
     
      \bigskip

 We now assume that $\C$ is a monoidal categorification of a cluster algebra $\A$. Let $\phi$ be a ring isomorphism
  $$  \phi : K_0(\C) \xrightarrow[]{\simeq}  \A .$$
  As $K_0(\C)$ is isomorphic to $\A$, it is in particular commutative which implies that the monoid $(\M , \odot)$ is commutative as well. Hence it can be canonically embedded into its Grothendieck group $G(\M)$, which is defined as follows (see \cite{Bou}):

\begin{deftn}[Grothendieck group of $\M$]
 Elements of $G(\M)$ are equivalence classes of couples $(\mu,\nu)$ of elements of $\M$ with respect to the equivalence relation 
 $$ (\mu,\nu) \sim (\mu',\nu') \Leftrightarrow \exists \lambda \in \M , \mu \odot \nu' \odot \lambda = \nu \odot \mu' \odot \lambda. $$
The group $G(\M)$ is an abelian group. We denote it by $\G$ in what follows. 
\end{deftn}

The inverse in $\G$ of an element $\mu \in \M$ will be denoted by $ {\mu}^{\odot -1}$. Similarly $g^{\odot -1}$ stands for the inverse in $\G$ of any element $g$ of $\G$. We will refer to elements of $\M$ as \textit{parameters} and elements of $\G$ as \textit{generalized parameters}.

\begin{prop}
The ordering $\leq$ on $\M$  naturally extends to a partial ordering on $\G$ that we also denote by $\leq$.
\end{prop}

\begin{proof}

One defines a partial ordering on $\M \times \M$  by setting 
$$ (\mu,\nu) \leq  (\mu',\nu') \Leftrightarrow  \exists \lambda \in \M , \lambda  \odot \mu \odot \nu' \leq \lambda \odot \mu' \odot \nu .$$
Using the assumption~\ref{assump}, one can check that if $(\mu,\nu) \sim (\mu',\nu')$ then for any $(\mu'',\nu'') \in \M \times \M$, one has $ (\mu,\nu) \leq  (\mu'',\nu'') \Leftrightarrow (\mu',\nu') \leq  (\mu'',\nu'') $. 
Thus $\leq$ naturally gives rise to a well-defined  partial ordering on $\G$. 
\end{proof}

\smallskip
\smallskip
 
Let us now fix a seed  $((x_1, \ldots , x_m), B)$ in $\A$ and choose for each $1 \leq i \leq m$ a representative $L(\mu_i)$ of the isomorphism class $\phi^{-1}(x_i) \in {\bf S}$ corresponding to the cluster variable $x_i$ . Recall that $\mathcal{M}_{{\bf x}}$ stands for the monoid of all monomials in the  $x_i$ and $\mathcal{G}_{{\bf x}}$ for the abelian group of all Laurent monomials in the  $x_i$. For any $m$-tuple of integers $(\alpha_1, \ldots , \alpha_m)$, we let ${\boldsymbol{\mu}}^{\boldsymbol{\alpha}}$ denote the element $ \bigodot_{1 \leq i \leq m} {\mu_i}^{\alpha_i}$ of $\G$. Of course ${\boldsymbol{\mu}}^{\boldsymbol{\alpha}}$ belongs to $\M$ if all the $\alpha_i$ are nonnegative.  

  Let us consider the subset $\mathcal{P}$ of $K_0(\C)$ consisting of nonzero classes $[M]$ such that there is a unique maximal element among the parameters of the simples appearing in the Jordan-H\"older series of $M$. 
  Let $\Psi$ be the map 
  $$ \begin{array}{cccc}
 \Psi  : & \mathcal{P} & \longrightarrow & \G \\
    {}   & [M] =  a_1[M_1] + \cdots + a_r[M_r] & \longmapsto &  max_{\leq}\left( \psi([M_k]) , 1 \leq k \leq r \right).
    \end{array}$$
    Here the $a_k$ are integers and the $M_i$ are the simples of the Jordan-H\"older series of $M$. Note that $\mathcal{P}$ contains $1$, whose image by $\Psi$ is the neutral element of $\M$.
  The set ${\bf S}$ of classes of simples in $\C$ is a basis of $K_0(\C)$ hence the writing $a_1[M_1] + \cdots + a_r[M_r]$ of an element of $\mathcal{P}$ is unique up to reordering the terms and thus the map $\Psi$ is well-defined.

  Let $\tilde{\mathcal{P}}$  be the subset of $Frac(\A)$ defined in the following way:  
  $$ \tilde{\mathcal{P}} := \{ {\bf x}^{\boldsymbol{\alpha}} \phi (p) ,  \boldsymbol{\alpha} \in \mathbb{Z}^m , p \in \mathcal{P} \}.$$ 
  In particular $\tilde{\mathcal{P}}$ contains $\mathcal{G}_{{\bf x}}$.
 Let  $\tilde{\Psi}$ be the map
 $$ \begin{array}{cccc}
 \tilde{\Psi}  : & \tilde{\mathcal{P}} & \longrightarrow & \G \\
    {}   & {\bf x}^{\boldsymbol{\alpha}} \phi (p) & \longmapsto & {\boldsymbol{\mu}}^{\boldsymbol{\alpha}} \odot \Psi (p).
    \end{array}$$
 
  \begin{prop} \label{psi}
 The map $ \tilde{\Psi}$ is well defined and satisfies the following properties:
 
   \begin{enumerate}[(i)]
   \item $\tilde{\Psi} \circ \phi $ coincides with $\psi$ on ${\bf S}$.
   \item $\tilde{\Psi}$ defines an abelian group morphism from $\mathcal{G}_{{\bf x}}$ (for the natural multiplication) to $(\G,\odot)$.
   \end{enumerate}
   
   \end{prop}
    
    \begin{proof} 
    
 In order to show that the map $\tilde{\Psi}$ is well-defined, we need to check that if $\boldsymbol{\alpha} , \boldsymbol{\beta}$ are two $m$-tuples of integers and $p,q$ two elements of $\mathcal{P}$ such that ${\bf x}^{\boldsymbol{\alpha}} \phi (p) = {\bf x}^{\boldsymbol{\beta}} \phi (q)$, then the equality ${\boldsymbol{\mu}}^{\boldsymbol{\alpha}} \odot \Psi (p) = {\boldsymbol{\mu}}^{\boldsymbol{\beta}} \odot \Psi (q)$ holds in $\G$. Let us write $ p = a_1[M_1] + \cdots + a_r[M_r] $ and $q = b_1[N_1] + \cdots + b_s[N_s]$, with $r,s \geq 0$, $a_1, \ldots a_r, b_1 , \ldots , b_s \in \mathbb{Z}$ and $[M_1], \ldots [M_r],[N_1], \ldots , [N_s] \in {\bf S}$. Let $\boldsymbol{\gamma}$ be an $m$-tuple of nonnegative integers such that ${\bf x}^{\boldsymbol{\gamma}}{\bf x}^{\boldsymbol{\alpha}}$ and ${\bf x}^{\boldsymbol{\gamma}}{\bf x}^{\boldsymbol{\beta}}$ are monomials in the $x_i$. One can write
 \begin{align*}
 {\bf x}^{\boldsymbol{\alpha}} \phi (p) = {\bf x}^{\boldsymbol{\beta}} \phi (q) &\Leftrightarrow {\bf x}^{\boldsymbol{\gamma}}{\bf x}^{\boldsymbol{\alpha}} \phi (p) = {\bf x}^{\boldsymbol{\gamma}}{\bf x}^{\boldsymbol{\beta}} \phi (q)   \\
   & \Leftrightarrow \phi \left( \prod_i [L(\mu_i)]^{\gamma_i + \alpha_i} p \right) = \phi \left( \prod_i [L(\mu_i)]^{\gamma_i + \beta_i} q \right) \\
   & \Leftrightarrow [L({\boldsymbol{\mu}}^{\boldsymbol{\gamma}+\boldsymbol{\alpha}})].(a_1[M_1] + \cdots + a_r[M_r]) =  [L({\boldsymbol{\mu}}^{\boldsymbol{\gamma}+\boldsymbol{\beta}})].(b_1[N_1] + \cdots + b_s[N_s]) 
 \end{align*}
 as $\phi$ is an isomorphism. Let us set $\mu := {\boldsymbol{\mu}}^{\boldsymbol{\gamma}+\boldsymbol{\alpha}} $ and $\nu := {\boldsymbol{\mu}}^{\boldsymbol{\gamma}+\boldsymbol{\beta}}$; they are elements of $\M$. By  Assumption~\ref{decomp}, for every $1 \leq i \leq r$ the product $[L(\mu)].[M_i]$  decomposes as a sum of classes of simples and $\mu \odot \psi([M_i])$ is maximal among the corresponding parameters. As $p \in \mathcal{P}$, the finite set of parameters $\psi([M_1]), \ldots \psi([M_r])$ has a unique maximal element. For simplicity let us assume it is $\psi([M_1])$. Then Assumption~\ref{assump} implies that $\mu \odot \psi([M_1])$ is maximal among the parameters appearing in the decompositions of $[L(\mu)].[M_i], 1 \leq i \leq r$ into classes of simples. The same arguments of course hold for the right hand side $[L(\nu)].q$.  As ${\bf S}$ is a basis of $K_0(\C)$, one gets
  $$ \mu \odot \psi([M_1]) = \nu \odot \psi([N_1]).$$
  Now by definition of $\Psi$, one has $\psi([M_1]) = \Psi(p)$ and $\psi([N_1]) = \Psi(q)$. Hence we can write in $\G$:
 $$ {\boldsymbol{\mu}}^{\boldsymbol{\alpha}} \odot \Psi (p) = {\boldsymbol{\mu}}^{- \boldsymbol{\gamma}} \odot \mu \odot \Psi (p) = {\boldsymbol{\mu}}^{- \boldsymbol{\gamma}} \odot \mu \odot \psi([M_1]) 
   = {\boldsymbol{\mu}}^{- \boldsymbol{\gamma}} \odot \nu \odot \psi([N_1]) =  {\boldsymbol{\mu}}^{- \boldsymbol{\gamma}} \odot \nu \odot \Psi (q) =  {\boldsymbol{\mu}}^{\boldsymbol{\beta}} \odot \Psi (q) $$
  which is the desired equality. Thus $\tilde{\Psi}$ is well-defined.

  For any $\mu \in \M$,  $\phi ([L(\mu)])$ belongs to $\tilde{\mathcal{P}}$ with $ \boldsymbol{\alpha}$ being zero and $p=[L(\mu)]$. Hence by definition one has $$\tilde{\Psi} \left( \phi ([L(\mu)]) \right) = \Psi ([L(\mu)]) = \mu = \psi ([L(\mu)])$$ which proves (i). 
  
   Taking elements of $\tilde{\mathcal{P}}$ for which $p$ is the empty sum, one gets 
\begin{align*}  
   \tilde{\Psi} ({\bf x}^{\boldsymbol{\alpha}}) 
    & = {\boldsymbol{\mu}}^{\boldsymbol{\alpha}} 
     =  \bigodot_{1 \leq i \leq m} {\mu_i}^{\alpha_i} \\
     &=  \bigodot_{1 \leq i \leq m}   { \tilde{\Psi} \left( \phi ([L(\mu_i)]) \right) }^{\alpha_i}    \text{ by (i) } \\
    & = \bigodot_{1 \leq i \leq m}   { \tilde{\Psi}(x_i) }^{\alpha_i} 
    \end{align*}
      which proves (ii). 
    
      \end{proof}

 \subsection{Compatible seeds} 
  \label{compatseed}

 In this subsection, we introduce the notion of compatible seed, state our main conjecture (Conjecture~\ref{conjecture}) and explain some consequences. 

 \smallskip
 
      \begin{deftn}[Compatible seed]
        \label{compat}
     Let $ \s := ((x_1, \ldots , x_m), B)$ be a seed in $\A$,  $\mathcal{G}_{{\bf x}}$ the group of Laurent monomials in the cluster variables $x_1, \ldots, x_m$, and $\preccurlyeq $ the corresponding dominance order on $\mathcal{G}_{{\bf x}}$. Let $\tilde{\Psi}$ be the map given by Proposition~\ref{psi}. For any $1 \leq j \leq n$, set 
$$ \mjh := \tilde{\Psi}(\yjh) =  \bigodot_{1 \leq i \leq m} \mu_i^{\odot b_{ij} }.$$
    We say that the seed $\s$  is compatible if the restriction $\tilde{\Psi}_{\mid \mathcal{G}_{{\bf x}}} : (\mathcal{G}_{{\bf x}} , \preccurlyeq) \longrightarrow (\G , \leq) $ is either increasing or decreasing.
      \end{deftn}

  \begin{rk} \label{rkpsi}
By construction, the restriction of  $\tilde{\Psi}$  to $\mathcal{G}_{{\bf x}}$ is  increasing if and only if  for any  Laurent monomials  $\prod_i x_i^{\alpha_i}$ and $ \prod_i x_i^{\beta_i}$ one has 
   $$ \prod_i x_i^{\alpha_i} \preccurlyeq \prod_i x_i^{\beta_i} \Rightarrow \bigodot_i \mu_i^{\odot \alpha_i } \leq \bigodot_i \mu_i^{\odot \beta_i }.$$ 
   This is equivalent to the following:
     $$ \forall 1 \leq j \leq n , \forall \mu \in \M ,  \mu  \leq \hat{\mu_j} \odot \mu. $$
  In the same way, $\tilde{\Psi}$ is decreasing on $\mathcal{G}_{{\bf x}}$ if and only if
   $$ \forall 1 \leq j \leq n , \forall \mu \in \M ,  \mu  \geq \hat{\mu_j} \odot \mu .$$ 
  In many examples of monoidal categorifications of cluster algebras, for instance via quantum affine algebras or quiver Hecke algebras, we will check this condition to prove that a seed is compatible. 
 \end{rk}    
 
  \begin{rk}
  Note that if the seed $ \s = ((x_1, \ldots , x_m), B)$ is compatible with $\tilde{\Psi}$ increasing on $\mathcal{G}_{{\bf x}}$, then $\tilde{\mathcal{P}}$ contains all the Laurent polynomials in the $x_i$ that are  pointed with respect to $\s$, i.e. $ \tilde{\mathcal{P}} \supseteq \mathcal{P} \mathcal{T}(0)$ with the notations of \cite{Qin}.
  \end{rk}
  
  \smallskip
 
 One can now state the main conjecture of this paper:
 
      \begin{conj} \label{conjecture}
         Let $\A$ be a cluster algebra  and $\C$ an Artinian  monoidal categorification of $\A$. Assume there exists a poset $(\M,\leq)$ as above such that Assumptions~\ref{decomp} and ~\ref{assump} hold. 
      Then there exists a compatible seed in $\A$.
       \end{conj}
 
\bigskip

 The next statements provide some useful consequences of the existence of a compatible seed. In particular we combine it with the results of \cite{FZ4} and \cite{DWZ} recalled in Section~\ref{algclust} to relate parameters of simple objects in $\C$ to some cluster algebra invariants, such as $g$-vectors and $F$-polynomials.
 
  Let $\C$ be a Artinian monoidal categorification  of a cluster algebra $\A$ and assume Conjecture~\ref{conjecture} holds. Let $((x_1, \ldots , x_n, x_{n+1} , \ldots , x_m) , B)$ be a compatible seed. Consider $x_l^t$ a cluster variable in $\A$ belonging to another cluster and let $F^{l,t}$ be its $F$-polynomial. Let $X_1^{a_1^{l,t}} \cdots X_n^{a_n^{l,t}}$ be the monomial given by Theorem~\ref{FpolyDWZ}(i).

 \begin{cor}
 One has $F^{l,t}  \left( \ylh , \ldots , \ynh \right) \in \tilde{\mathcal{P}}$ and  
  $$\tilde{\Psi} \left( F^{l,t}(\ylh , \ldots , \ynh) \right) =
   \begin{cases}
       \bigodot_j {\mjh}^{ \odot a_j^{l,t}} &\text{ if $\tilde{\Psi}$ is increasing on $\mathcal{G}_{{\bf x}}$,} \\
        1_{\G} &\text{in the other case.}
        \end{cases}$$
   Here $1_{\G}$ denotes the neutral element of $\G$. 
\end{cor}

 \begin{proof}
 
 By Theorem~\ref{thmFpoly}, $F^{l,t}(\ylh , \ldots , \ynh)$ is the product of a Laurent monomial in the $x_i$ with the cluster variable $x^{l,t}$. As $\C$ is a monoidal categorification of $\A$, $x^{l,t}$ is the image by $\phi$ of the class of a simple object in $\C$. Hence $F^{l,t}(\ylh , \ldots , \ynh) \in \tilde{\mathcal{P}}$.

 By Theorem~\ref{FpolyDWZ} (i) and (ii),  any monomial of $F^{l,t}$ can be written as $X_1^{b_1} \cdots X_n^{b_n}$ with $ 0 \leq b_j \leq a_j^{l,t}$ for all $1 \leq j \leq n$. As the seed $((x_1, \ldots , x_m),B)$ is compatible, evaluating on the $\yjh$ and considering the corresponding generalized parameters $\mjh$ yields
$$ \bigodot_j {\mjh}^{a_j^{l,t}} \geq \bigodot_j {\mjh}^{b_j} \quad  \text{if $\tilde{\Psi}$ is increasing on $\mathcal{G}_{{\bf x}}$}$$
and 
$$ \bigodot_j {\mjh}^{a_j^{l,t}} \leq \bigodot_j {\mjh}^{b_j} \quad  \text{if $\tilde{\Psi}$ is decreasing $\mathcal{G}_{{\bf x}}$}.$$
Hence among all the generalized parameters  appearing in the term $F^{l,t}(\ylh , \ldots , \ynh)$, there is one which is greater than the others, namely $\bigodot_j {\mjh}^{\odot a_j^{l,t}}$ if $\tilde{\Psi}$ is increasing $\mathcal{G}_{{\bf x}}$, $1_{\G}$ in the other case. 

  \end{proof} 
 
  The following Corollary shows that the existence of a compatible seed in $\A$ implies relations between the $g$-vector (with respect to this initial (compatible) seed) of any cluster variable in $\A$ and  the parameter of the corresponding simple object in $\C$.
  
  Let $x_{n+1}^{c_1^{l,t}} \cdots x_m^{c_{m-n}^{l,t}}$ be the monomial in the frozen variables equal to the denominator ${F^{l,t}}_{|{\p}}(y_1, \ldots, y_n)$ in the right hand side of equation~\eqref{Fpoly}. Note that, in view of the definition of the semifield $\p$, the $c_i^{l,t}$ are negative integers, as the $F$-polynomial $F^{l,t}$  has  constant term equal to $1$ by Theorem~\ref{FpolyDWZ}(ii). 
  
 \begin{cor}
 Let $\mu^{l,t}$ be  the parameter of the simple module corresponding to the cluster variable $x^{l,t}$, i.e. $x^{l,t} = \phi([L(\mu^{l,t})])$. Then 
 $$ \mu^{l,t} = 
  \begin{cases}
     \bigodot_{1 \leq j \leq n} {\mjh}^{ \odot a_j^{l,t}} \odot \bigodot_{1 \leq i \leq m-n} \mu_{n+i}^{ \odot (-c_i^{l,t})}  \odot  \bigodot_{1 \leq i \leq n} \mu_i^{ \odot g_i^{l,t} } &\text{if $\tilde{\Psi}$ is increasing on $\mathcal{G}_{{\bf x}}$, } \\
     \bigodot_{1 \leq i \leq m-n} \mu_{n+i}^{ \odot (-c_i^{l,t})}  \odot  \bigodot_{1 \leq i \leq n} \mu_i^{ \odot g_i^{l,t} } &\text{if $\tilde{\Psi}$ is decreasing on $\mathcal{G}_{{\bf x}}$.}
      \end{cases}$$
 \end{cor}
 
  \begin{proof}
 
  We give the proof only in the case where the restriction of $\tilde{\Psi}$ to $\mathcal{G}_{{\bf x}}$ is increasing, the other case being analogous.

 By Proposition~\ref{psi} (i), $\tilde{\Psi}(x_l^t)=\psi([L(\mu^{l,t})])=\mu^{l,t}$.
 On the other hand, one can use the previous Corollary and apply $\tilde{\Psi}$ to both sides of equation~\eqref{Fpoly}:
 \begin{align*}
    \bigodot_j {\mjh}^{ \odot a_j^{l,t}}
  &= \tilde{\Psi} \left( F^{l,t}(\ylh , \ldots , \ynh) \right)
  = \tilde{\Psi} \left( \frac{x_{n+1}^{c_1^{l,t}} \cdots x_m^{c_{m-n}^{l,t}}}{x_1^{g_1^{l,t}} \cdots x_n^{g_n^{l,t}}} . x^{l,t} \right) \\
  &=  \mu^{l,t} \odot \bigodot_{1 \leq i \leq m-n} \mu_{n+i}^{ \odot c_i^{l,t}} \odot \bigodot_{1 \leq i \leq n} \mu_i^{\odot (-g_i^{l,t})}  \text{by Proposition~\ref{psi}(ii).}  
  \end{align*}
  Finally one can conclude:
$$ \mu^{l,t} = \bigodot_{1 \leq j \leq n} {\mjh}^{ \odot a_j^{l,t}} \odot \bigodot_{1 \leq i \leq m-n} \mu_{n+i}^{ \odot (-c_i^{l,t})} \odot \bigodot_{1 \leq i \leq n} \mu_i^{\odot g_i^{l,t}} .$$

\end{proof}

\subsection{First example: the category ${\C}_1$}

    The first example of compatible seed appears in the work of Hernandez-Leclerc \cite{HL} and is one of the main motivations for this work. 
    
     \smallskip
     \smallskip
     
     Recall from Section~\ref{qaa} the definition of the category ${\C}_1$. This category was introduced by Hernandez-Leclerc in \cite{HL} as a (monoidal) subcategory of the category of finite dimensional representations of the quantum affine algebra $U_q (\hat{\mathfrak{g}})$. For $\mathfrak{g}$ of types $ADE$, the category ${\C}_1$ is a monoidal categorification of a cluster algebra of the same cluster type (in the classification of \cite{FZ2}) than the Lie type of $\mathfrak{g}$ (\cite{HL,Naka}). As explained in Section~\ref{qaa}, the simple finite dimensional $U_q (\hat{\mathfrak{g}})$-modules are parametrized by dominant monomials. The monoid $\M$ parametrizing simple objects in the category ${\C}_1$ is a submonoid of the set of dominant monomials involving only the variables $Y_{i,a}, i \in I , a \in q^{\mathbb{Z}}$. The monoid law $\odot$ is simply the natural multiplication of monomials. The ordering $\leq$ on $\M$ is the restriction of the Nakajima order on dominant monomials (see Section~\ref{qaa}).  Assumptions~\ref{decomp} and ~\ref{assump} are obviously satisfied for the category ${\C}_1$.

     In \cite{HL}, Hernandez-Leclerc give explicitly an initial  seed in the category ${\C}_1$:
   
   \begin{thm}[\cite{HL}] \label{HLseed}

 Each seed has $n=|I|$ unfrozen variables and $n$ frozen variables. These frozen variables are given by the classes of the modules $L(Y_{i,q^{\xi_i}}Y_{i,q^{\xi_i +2}})$, $i \in I$.

Moreover, an initial seed in this cluster algebra  is given by the following classes:
 $$  [L(Y_{i,q^{ \xi_i +2}})] \qquad  \qquad [L(Y_{i, q^{\xi_i}}Y_{i, q^{\xi_i +2}})] \qquad  i \in I $$
together with the exchange matrix $B=(b_{ij})$  whose columns are indexed by $I$ and rows by $I \sqcup I' = [1,n] \sqcup [n+1,2n]$, and whose entries are given by
$$ b_{ij} := \begin{cases}  (-1)^{\xi_j} a_{ij} &\text{if $i,j \in I$ and $i \neq j$,}  \\ -1 &\text{if $j \in I$ and $i=j+n \in I'$,} \\ 
                                    -a_{kj} &\text{if $j \in J_0$ and $i=k+n \in I'$ with $k \neq j$,} \\ 0 &\text{otherwise.}   \end{cases}  $$
Here the $a_{ij}$ are the entries of the Cartan matrix associated to $\mathfrak{g}$. 

 \end{thm}

The cluster algebra $\A$ is the cluster algebra (in the classification of \cite{FZ2}) generated by the initial seed $((x_1, \ldots , x_n , x_{n+1} , \ldots , x_{2n}), B)$ where $B$  is the matrix above. Following \cite{HL}, we denote by $\iota$ the unique ring isomorphism 
\[ \xymatrix{ \iota : \A \ar[r]^{\simeq} & K_0(\C_1) }
\]
such that 
$$ \iota(x_i) = [L(Y_{i,q^{ \xi_i +2}})] \qquad  \qquad  \iota(x_{n+i}) = [L(Y_{i,q^{\xi_i}}Y_{i,q^{\xi_i +2}})] \qquad 1 \leq i \leq n. $$
The isomorphism $\iota$ is the inverse of the isomorphism $\phi$ in Section~\ref{genepara}.  Using the map $\tilde{\Psi}$ associated to the seed $((x_1, \ldots , x_n , x_{n+1} , \ldots , x_{2n}), B)$ above, one can compute the generalized parameters $\mjh$  corresponding to this seed. This is done by the following statement, as a direct consequence of Theorem~\ref{HLseed}:

 \begin{prop}[{{\cite[Lemma 7.2]{HL}}}] \label{HLcompat}
 With the notations of Section~\ref{genepara}, one has:
 $$ \forall 1 \leq j \leq n, \quad \mjh = A_{j, q^{\xi_j +1}}^{-1} .$$
 \end{prop}

 \begin{cor}
 Conjecture~\ref{conjecture} holds for the category ${\C}_1$.
 \end{cor}
 
 \begin{proof}
By definition of the Nakajima ordering on monomials (see Section~\ref{qaa}), Proposition~\ref{HLcompat} implies that for any dominant monomial $\mathfrak{m}$, one has
$$ \forall 1 \leq j \leq n ,  \mathfrak{m}  \geq \mjh  \mathfrak{m}. $$
By Remark~\ref{rkpsi}, this implies that the map $\tilde{\Psi}$ associated to the seed $((x_1, \ldots , x_n , x_{n+1} , \ldots , x_{2n}), B)$ is decreasing. Hence this seed is compatible in the sense of Definition~\ref{compat} and Conjecture~\ref{conjecture} holds.
\end{proof}

We conclude this section with an illustration of the relationships between $g$-vectors and highest weights for quantum affine algebras that occur as a consequence of the initial seed of \cite{HL} being compatible. The cluster structure on $K_0(\C_1)$ is a finite type cluster algebra; thus one can use the results of \cite{FZ2} and label the cluster variables by almost positive roots, i.e. positive roots together with the opposite of the simple roots. Let $x[\alpha]$ denote the cluster variable associated to the almost positive root $\alpha$ with respect to the above initial seed. 

 Following \cite{FZ2}, one defines piece-wise linear involutions $\tau_{\epsilon}$ ($\epsilon \in \{-1,1\}$) of the root lattice $Q$ of $\mathfrak{g}$: for any $\gamma \in Q$, 
 $$[\tau_{\epsilon}(\gamma) : \alpha_i] = 
    \begin{cases} 
      - [\gamma : \alpha_i] - \sum_{j \neq i} a_{ij} \max(0,[\gamma, \alpha_j]) &\text{if $\epsilon_i = \epsilon$} \\
      [\gamma : \alpha_i] &\text{if $\epsilon_i \neq \epsilon$}
    \end{cases}
 $$
 where $[\gamma : \alpha_i]$ stands for the coefficient of $\alpha_i$ in the expansion of $\gamma$ on simple roots. 
 
 \begin{cor}[{{\cite[Corollary 7.4]{HL}}}]
 Let $\alpha$ be an almost positive root. Set $ \beta := \tau_{-}(\alpha)$. Write $\beta = \sum_i b_i \alpha_i$.  The highest weight of the simple module corresponding to the cluster variable $x[\alpha]$ is given by
   $$ \prod_{i \in I_0} Y_{i,1}^{b_i} . \prod_{i \in I_1} Y_{i,q^3}^{b_i}. $$
 \end{cor}
 
  It is known from \cite{FZ4} that in the case of a cluster algebra of finite type, the $g$-vector of the variable $x[\alpha]$ is given by 
   $$ {\bf g}(\alpha) = E \tau_{-}(\alpha)$$
      where $E$ is the automorphism of the root lattice of $\mathfrak{g}$ which sends the simple root $\alpha_i$ onto $ (-1)^{\xi_i +1} \alpha_i$. 
      
   Thus the previous corollary can be reformulated in the following way:
   
 \begin{cor}
Let $\alpha$ be an almost positive root and let ${\bf g}(\alpha)$ be the $g$-vector of the cluster variable $x[\alpha]$ with respect to the above initial seed. The highest weight of the simple module corresponding to $x[\alpha]$ is given by
$$ \prod_{i \in I_0} Y_{i,1}^{-g_i} . \prod_{i \in I_1} Y_{i,q^3}^{g_i}. $$
\end{cor}

 \section{A mutation rule for parameters of simple representations of quiver Hecke algebras}
 \label{mutrule}

 In this section, we consider the category $\C= R-gmod$ of finite dimensional representations of symmetric quiver Hecke algebras of finite type $A_n$. The set $\M$ is the set of \textit{dominant words} (see Section~\ref{reminders}) and the order $\leq$ is the natural lexicographic order; it is a total ordering hence Assumption~\ref{decomp} obviously holds. Moreover, with the notations of Section~\ref{irredKLR}, one has $\mu \odot \nu = \max (L(\mu) \circ L(\nu))$ for any dominant words $\mu$ and $\nu$. We begin by describing explicitly the monoid operation $\odot$ for dominant words. In particular it can be easily computed using  \textit{canonical factorizations} of dominant words (see Proposition~\ref{canofacto}). We apply this in the context of monoidal categorifications of cluster algebras via quiver Hecke algebras following the works of Kang-Kashiwara-Kim-Oh (\cite{KKK,KKKO}). We obtain a combinatorial rule for the transformation of dominant words under cluster mutation.

\smallskip
 
\subsection{Convolution product of simple modules} \label{KLRtechnique}

 This subsection is devoted to the description of the monoid structure $\odot$ on the monoid $\M$ of dominant words in the case of a symmetric quiver Hecke algebra of type $A_n$. First we restrict to the case where the canonical factorizations of two words $\mu , \mu'$ are ordered with respect to each other. We show that in this case, Proposition~\ref{KRprod} implies that the monoidal product $\mu \odot \mu'$ is simply the concatenation of $\mu$ and $\mu'$ (Corollary~\ref{littlecor}). Then we state the main result of this section (Proposition~\ref{prod}) which gives a combinatorial expression for $\mu \odot \mu'$ for any $\mu, \mu'$. Our proof involves ideas similar to the ones used in \cite{KR} for the proof of Proposition~\ref{KRprod}, but here we use the specific form of dominant Lyndon words (in bijection with positive roots) in type $A_n$. 

Recall (see Proposition~\ref{canofacto}) that any word $\mu$  can be written in a unique way as a concatenation of Lyndon words in the decreasing order. This is called the canonical factorization of $\mu$. Moreover, if the word $\mu$ is dominant, then all the Lyndon words involved in the canonical factorization of $\mu$ are dominant as well (Theorem~\ref{domfacto}(ii)). By Theorem~\ref{domfacto}(i), the canonical factorization of a dominant word $\mu$ can be seen as a sum of positive roots in the decreasing order. In particular, in type $A_n$ these positive roots correspond to words of the form $k(k+1) \ldots l$ with $k \leq l$. 

  We begin by recalling a technical result from \cite{KR}.
  
     \begin{lem}[{{\cite[Lemma 5.1]{KR}}}]
     \label{max}
 Let $ {\bf i}^{(1)}, \ldots ,{\bf i}^{(r)},{\bf j}^{(1)}, \ldots ,{\bf j}^{(r)}$ be words such that for each $k \in \{1, \ldots, r\}$, ${\bf i}^{(k)}$ and ${\bf j}^{(k)}$ have same length. 
Assume that  ${\bf i}^{(1)} \geq {\bf j}^{(1)}, \ldots , {\bf i}^{(r)} \geq {\bf j}^{(r)}.$
 Then $ \max ({\bf i}^{(1)} \circ \ldots \circ {\bf i}^{(r)}) \geq \max ({\bf j}^{(1)} \circ \ldots \circ  {\bf j}^{(r)}).$
Moreover, this inequality is an equality if and only if all the inequalities $ {\bf i}^{(k)} \geq {\bf j}^{(k)}$ are equalities. 
 \end{lem} 

The next Proposition states that the product $\mu \odot \mu'$ of two dominant words $\mu$ and $\mu'$ coincides with the  highest word in the quantum shuffle product of $\mu$ and $\mu'$. 
 
   \begin{prop} \label{maxshuffle}
 Let $\mu,\nu$ be dominant words. 
Then $$ \mu \odot \nu = \max (\mu \circ \nu).$$
  \end{prop}

\begin{proof}

 By  Theorem~\ref{KR}, we can write 
$$ ch_q ( L(\mu) ) = P(q) . \mu  + \sum_{ \mu'< \mu} a_{\mu'} (q) . \mu'  \qquad and \qquad ch_q ( L(\nu) ) = Q(q) . \nu  + \sum_{\nu' < \nu} b_v (q) . \nu' $$
where $P,Q,a,b$ are Laurent polynomials in $q$ ($P$ and $Q$ non zero).
  By definition 
 $$ \mu \odot \nu = \max (L(\mu)  \circ L(\nu)) = \max (ch_q (L(\mu)  \circ L(\nu)))  = \max ( ch_q(L(\mu))  \circ ch_q(L(\nu)) ).$$
By Lemma~\ref{max}, for any $ \mu'< \mu$, $ \max (\mu' \circ \nu) < \max  (\mu \circ \nu) $.  Similarly, $ \max (\mu \circ \nu') < \max (\mu \circ \nu) $ and $ \max (\nu' \circ \nu) < \max (\mu \circ \nu) $ for any $\mu' < \mu$ and $\nu' < \nu$. 
 So the highest word of $ ch_q(L(\mu))  \circ ch_q(L(\nu)) $ can only come from the shuffle product $\mu \circ \nu$. 
Hence 
$$  \mu \odot \nu = \max (\mu \circ \nu).$$

\end{proof}

  More generally, the same proof shows that for any finite formal series $R$ and $S$ on $W$ with coefficients in $\mathbb{Z}[q,q^{-1}]$, one has $ \max (R \circ S) =  \max \left( max(R) \circ max(S) \right)$. The following Corollary is then a direct consequence of Proposition~\ref{KRprod}.
  
  \begin{cor} \label{littlecor}
   Let $\mu, \mu'$ two dominant words. Write their canonical factorizations
   $$  \mu = ({\bf i}^{(1)})^{n_1} \cdots ({\bf i}^{(r)})^{n_r} \qquad and \qquad 
 \mu' = ({\bf i'}^{(1)})^{n'_1} \cdots ({\bf i'}^{(r')})^{n'_{r'}} $$
  with  ${\bf i}^{(1)} > \cdots > {\bf i}^{(r)}$ and   ${\bf i'}^{(1)} > \cdots > {\bf i'}^{(r')}$.
 Assume ${\bf i}^{(r)} \geq {\bf i'}^{(1)}$.
  
  Then 
  $$ \mu \odot \mu'  = ({\bf i}^{(1)})^{n_1} \cdots ({\bf i}^{(r)})^{n_r} ({\bf i'}^{(1)})^{n'_1} \cdots ({\bf i'}^{(r')})^{n'_{r'}}. $$
   \end{cor}
   
    \begin{proof}
    Setting $R := ({\bf i}^{(1)})^{n_1} \circ \cdots \circ ({\bf i}^{(r)})^{n_r}$ and $R' := ({\bf i'}^{(1)})^{n'_1} \circ \cdots \circ ({\bf i'}^{(r)})^{n'_{r'}}$, Proposition~\ref{KRprod} implies  $\mu = \max (R)$ and $\mu' = \max (R')$. Then by Proposition~\ref{maxshuffle}, one has:
  \begin{align*}
   \mu \odot \mu'  &=   \max( \mu \circ \mu' )  &\text{ by  Proposition~\ref{maxshuffle} }  \\
          &= \max (max(R) \circ max(R'))\\
          &= \max( R \circ R' ) &\text{ by  Proposition~\ref{maxshuffle} }  \\
          &= \max \left( ({\bf i}^{(1)})^{n_1} \circ \cdots \circ ({\bf i}^{(r)})^{n_r} \circ ({\bf i'}^{(1)})^{n'_1} \circ \cdots \circ ({\bf i'}^{(r)})^{n'_{r'}} \right). 
  \end{align*}
The assumption ${\bf i}^{(r)} \geq {\bf i'}^{(1)}$ implies that $({\bf i}^{(1)})^{n_1} \cdots ({\bf i}^{(r)})^{n_r}  ({\bf i'}^{(1)})^{n'_1}  \cdots  ({\bf i'}^{(r)})^{n'_{r'}} $ is the canonical factorization of the concatenation $\mu \mu'$. Hence by Proposition~\ref{KRprod}, we get 
$$ \mu \odot \mu' = \mu \mu' = ({\bf i}^{(1)})^{n_1} \cdots ({\bf i}^{(r)})^{n_r} ({\bf i'}^{(1)})^{n'_1} \cdots ({\bf i'}^{(r')})^{n'_{r'}}. $$
    \end{proof}
  
  One can now state the main result of this section. It can be seen as a generalization of Corollary~\ref{littlecor}, as we now drop the hypothesis ${\bf i}^{(r)} \geq {\bf i'}^{(1)}$.

   \begin{prop} \label{prod}
  Let $\mu, \mu'$ two dominant words. Write their canonical factorizations
   $$  \mu = ({\bf i}^{(1)})^{n_1} \cdots ({\bf i}^{(r)})^{n_r} \qquad and \qquad 
 \mu' = ({\bf i'}^{(1)})^{n'_1} \cdots ({\bf i'}^{(r')})^{n'_{r'}} . $$
Let $ \{ {\bf j}^{(1)}, \ldots , {\bf j}^{(s)} \}$ be the set of all the words ${\bf i}^{(1)},  \ldots , {\bf i}^{(r)}, {\bf i'}^{(1)}, \ldots , {\bf i'}^{(r')}$ ranged in the decreasing  order. Let $m_1, \ldots , m_s ,m'_1 \ldots m'_s$ the positive integers uniquely determined by 
  $$ \mu = ({\bf j}^{(1)})^{m_1} \cdots ({\bf j}^{(s)})^{m_s} \qquad , \qquad  \mu' = ({\bf j}^{(1)})^{m'_1} \cdots ({\bf j}^{(s)})^{m'_s}. $$
 Then 
$$ \mu \odot \mu' = ({\bf j}^{(1)})^{m_1 + m'_1} \cdots ({\bf j}^{(s)})^{m_s + m'_s}. $$
 \end{prop} 
 
  Using Theorem~\ref{KR}, one can reformulate this statement  in the following way: write the positive roots in the decreasing order
 \begin{align} \label{decrorder}
     & \alpha_n > \alpha_{n-1} + \alpha_n  >  \alpha_{n-1} > \cdots > \alpha_i +  \alpha_{i+1} +\cdots+ \alpha_n > \cdots> \alpha_i > \cdots \nonumber \\
      &  \qquad   \qquad  \qquad  \qquad  \qquad  \qquad  \qquad  \qquad  \qquad  \qquad  \qquad  \cdots > \alpha_1 +\cdots+ \alpha_n > \alpha_1
     \end{align} 
  and let $r_n=n(n+1)/2$ denote the number of these positive roots.  
 Define a  map 
 \begin{equation} \label{isomonoids}
   \begin{array}{ccc}
        (\M , \odot) & \longrightarrow & (\mathbb{Z}_{\geq 0}^{r_n} , +) \\
            \mu        & \longmapsto      &  \overrightarrow{\mu}
   \end{array} 
  \end{equation}
 such that the $i$th coordinate of the vector $\overrightarrow{\mu}$ is equal to the multiplicity of the Lyndon word corresponding to the $i$th positive root (in the above decreasing order) in the canonical factorization of $\mu$.
 
  \begin{thm}
   \label{thmonoid}
  The map~\eqref{isomonoids} is an isomorphism of abelian monoids. 
   \end{thm}
  
 \begin{proof}[Proof of Proposition~\ref{prod}]

 For simplicity we use a slight change of notation for the proof: we write
 $$ \mu = {\bf i}^{(1)} \cdots {\bf i}^{(r)} \qquad  and \qquad  \mu' = {\bf i'}^{(1)} \cdots {\bf i'}^{(r')} $$
  with  ${\bf i}^{(1)} \geq \cdots \geq  {\bf i}^{(r)}$ and   ${\bf i'}^{(1)} \geq \cdots \geq {\bf i'}^{(r')}$ dominant Lyndon words not necessarily distinct. 
 Let  $n$ (resp. $n'$) be the length of $\mu$ (resp. $\mu'$). The starting point is the word $\mu . \mu'$ which is the concatenation of the words $\mu$ and $\mu'$ and we consider permutations $\sigma \in \mathfrak{S}_{r,s} $, i.e. whose restrictions to  $ [| 1;n|]$ and $[|n+1;n+n'|]$ are increasing (see Definition~\ref{shuffle}).
 
 First note that the word $\mu \odot \mu'$ indeed appears in the quantum shuffle product of $\mu$ with $\mu'$: consider the permutation $\sigma$ simply defined by rearranging the blocks  $({\bf i}^{(1)}),\ldots, ({\bf i}^{(r)}), ({\bf i'}^{(1)}),\ldots, ({\bf i'}^{(r')})$ of the concatenation $\mu . \mu'$ and put them in the decreasing order. 
 
  \bigskip
  
   We write $\mu= h_1, \ldots, h_n$ and $\mu' = h'_1, \ldots , h'_{n'}$. The concatenation of $\mu$ and $\mu'$ is then 
   $\mu . \mu' = h_1, \ldots, h_n, h'_1, \ldots , h'_{n'}$. As in Definition~\ref{shuffle}, we set $h_{n+1} := h'_1 , \ldots , h_{n+n'} := h'_{n'}$ and thus $\mu . \mu' = h_1 , \ldots , h_{n+n'}$. 
   We set $\sigma (\mu . \mu') := h_{\sigma^{-1}(1)} , \ldots , h_{\sigma^{-1}(n+n')}$ for any permutation $\sigma \in  \mathfrak{S}_{r,s}$.
   From now on we fix  $\sigma \in  \mathfrak{S}_{r,s} $ and we assume that the word  $\sigma(\mu . \mu')$ is greater than or equal to $ \mu \odot \mu'$ (for the lexicographic order). We show that under this assumption, one necessarily has $\sigma(\mu . \mu') = \mu \odot \mu'$.

    The proof is based on an induction on $r + r'$ or equivalently on the sum of the lengths of $\mu$ and $\mu'$. 
   
   We first look at the action of  $\sigma$ on the Lyndon words  ${\bf i}^{(1)}$ and  ${\bf i'}^{(1)}$  and  show  that $\sigma$ necessarily rearranges these two blocks so that in the word $\sigma(\mu . \mu')$  they will appear in the decreasing order. Then, considering the restriction of the action of $\sigma$ on the other Lyndon words, we find ourselves considering a shuffle of parameters $\tilde{\mu}$ and $\tilde{\mu'}$, one of them being of length strictly smaller than the corresponding initial parameter.  
 
 \bigskip
 
  {\bf First case: ${\bf i'}^{(1)} > {\bf i}^{(1)}$}. Then the word $\mu \odot \mu'$ begins with ${\bf i'}^{(1)}$. 
   
   The two words  ${\bf i}^{(1)}$ and  ${\bf i'}^{(1)}$  are dominant Lyndon words so we write them ${\bf i}^{(1)}=(k,k+1, \ldots ,k+d_1)$ and ${\bf i'}^{(1)}=(k',k'+1, \ldots ,d'_1)$  with either $k' > k$ or $k'=k$ and $d'_1 > d_1$.

  We first show that the assumption $\sigma(\mu . \mu') \geq \mu \odot \mu'$ implies $ \sigma(n+1)=1$. 
  Indeed, if  $\sigma(n+1) \geq  d_1 + 1$ then, as the restrictions of $\sigma$ to $[|1;n|]$ and $[| n+1;n+n'|]$ are increasing, we have $\sigma(1)=1, \ldots, \sigma(d_1)=d_1$. The word $\sigma(\mu . \mu')$ then begins with $(k, \ldots k+d_1, l, \ldots)$, where $l$ is equal to $k$ if $\sigma(n+1) >  d_1 + 1$ and to $k'$ if $\sigma(n+1) = d_1 + 1$. 
   If $k<k'$ then $ k, \ldots, k+d_1 < {\bf i'}^{(1)} $ and hence $\sigma(\mu . \mu') < \mu \odot \mu'$. 
   If $k'=k$ and $d'_1 > d_1$ then 
  \begin{align*}
   k, \ldots k+d_1, l \ldots  &=  k', \ldots , k' + d_1 , k' \\
                                            &<  k', \ldots , k' + d_1 , k' + d_1 + 1 \\
                                            &\leq k', \ldots k' + d'_1 \\
                                            &= {\bf i'}^{(1)} 
\end{align*}   
and  the conclusion is the same.

  If $\sigma(n+1) \in \{2,3,\ldots, d_1 \}$, then $\sigma(\mu . \mu')$ begins with $ (k,k+1,\ldots,k+p,k',\ldots)$ where $p$ is some integer such that $0 \leq p < d_1$. 
  If $k<k'$ then $k,k+1,\ldots,k+p,k' < {\bf i'}^{(1)} $ and hence $ \sigma(\mu . \mu') < \mu \odot \mu'$. 
  If $k'=k$ and $d'_1 > d_1$ then 
  \begin{align*}
   k,k+1,\ldots,k+p,k' &= k,k+1,\ldots,k+p,k \\
                                   &< k,k+1,\ldots,k+p,k+p+1 \\
                                   &\leq k,k+1,\ldots,k + d_1 \\
                                   &< k,k+1,\ldots,k+p,k + d'_1 \\
                                   &= {\bf i'}^{(1)} 
  \end{align*}
  and  the conclusion is the same. 
   Thus $\sigma(n+1)=1$. 
      
      As the restrictions of $\sigma$ to $[|1;n|]$ and $[| n+1;n+n'|]$ are increasing, $\sigma^{-1}(2)$ is either equal to $1$ or to $n+2$; but the first possibility gives a word beginning with $k'k$ which is obviously  strictly smaller than $\mu \odot \mu'$. Hence   $\sigma(n+2)=2$. Then by iterating this, we see that necessarily $\sigma(n+1)=1, \ldots , \sigma(n+d'_1)=d'_1$. In other words $\sigma$  sends the blocks ${\bf i'}^{(1)}$ on the left of the blocks ${\bf i}^{(1)}$, i.e. at the beginning of the word $\sigma(\mu . \mu')$.
 
 \bigskip
 
  {\bf Second case: ${\bf i'}^{(1)} < {\bf i}^{(1)}$.} Then $\mu \odot \mu'$ begins with $({\bf i}^{(1)})$ and with the previous notations, one has  either $k' < k$ or $k'=k$ and $d'_1 < d_1$. 
 
 We  show that the assumption $\sigma(\mu . \mu') \geq \mu \odot \nu$ implies  $ \sigma(n+1) > d_1$. 
 
 Indeed, if $ \sigma(n+1) \in \{2,\ldots, d_1 \}$, then $\sigma(\mu . \mu')$ begins with $(k,\ldots,k+p,k')$  where $p$ is some integer such that $0 \leq p < d_1$. But then 
 \begin{align*}
   k,k+1,\ldots,k+p,k' &\leq  k,k+1,\ldots,k+p,k \\
                                   &< k,k+1,\ldots,k+p,k+p+1 \\
                                   &\leq k,k+1,\ldots,k + d_1 \\                               
                                   &= {\bf i}^{(1)} 
  \end{align*}
 and hence $\sigma(\mu . \mu') < \mu \odot \mu'$. 
 
  If $\sigma(n+1)=1$, then $\sigma(\mu . \mu') \geq \mu \odot \mu'$ implies $k'=k$ (and hence $d'_1 < d_1$). But then it is easy to see that necessarily $\sigma(n+2)=2, \ldots , \sigma(n+d'_1)=d'_1$, i.e. $\sigma(\mu . \mu')$ begins with $(k,\ldots,k+d'_1,\ldots)$. The letter coming after $k+d'_1$ is either the first letter of ${\bf i}^{(1)}$ (if $\sigma(1)=d'_1 + 1$) or the first letter of ${\bf i'}^{(2)}$ (if $\sigma(n+d'_1 + 1)=d'_1 + 1$); in both cases it is smaller than $k$ and in particular smaller than $k+d'_1 + 1$ and hence $\sigma(\mu . \mu') < \mu \odot \mu'$.
  Thus $ \sigma(n+1) > d_1$. 
 
 In particular, $\sigma(1)=1,\ldots, \sigma(d_1)=d_1$ (in other words, $\sigma$ fixes the block ${\bf i}^{(1)}$, i.e. leaves it at the beginning of the resulting word.
 
 \bigskip
 
  {\bf Third case: ${\bf i'}^{(1)} = {\bf i}^{(1)}$}. Then the word $\mu \odot \mu'$ begins with $({\bf i}^{(1)})^{2}$. 
 
 We  show that under the assumption $\sigma(\mu . \mu') \geq \mu \odot \mu'$, one has either  $ \sigma(n+1) =1, \ldots ,\sigma(n+d_1) = d_1$ (i.e. $\sigma$ sends the block ${\bf i}^{(1)}$ coming from $\mu'$ to the left of the block ${\bf i}^{(1)}$ coming from $\mu$) or   $ \sigma(1) =1, \ldots ,\sigma(d_1) = d_1$ (i.e. $\sigma$ fixes the block ${\bf i}^{(1)}$ coming from $\mu$).
 
  Indeed, as the restrictions of $\sigma$ to $[|1;n|]$ and $[| n+1;n+n'|]$ are increasing, $\sigma^{-1}(1)$ is either equal to $1$ or to $n+1$. 
  
  If $\sigma(1)=1$, then $\sigma(n+1)$ is necessarily strictly greater than $d_1$, otherwise $\sigma(\mu . \mu')$ would begin with $(k,k+1,\ldots,k+p,k,\ldots)$ (where $p$ is some integer such that $0 \leq p < d_1$) and would be strictly smaller than $\mu \odot \mu'$. Hence in this case we get $ \sigma(1) =1, \ldots ,\sigma(d_1) = d_1$. 
  
  If $\sigma(n+1)=1$, then the same argument shows that $\sigma(1)$ is necessarily strictly greater than $d_1$, and hence we get $ \sigma(n+1) =1, \ldots ,\sigma(n+d_1) = d_1$.
 
 \bigskip

  In conclusion, we have shown that the permutations we are seeking fix the block ${\bf i}^{(1)}$ if ${\bf i}^{(1)} > {\bf i'}^{(1)}$,  send the block ${\bf i'}^{(1)}$ to the left of the block ${\bf i}^{(1)}$ if ${\bf i}^{(1)}<{\bf i'}^{(1)}$, and  send either ${\bf i}^{(1)}$ or ${\bf i'}^{(1)}$ to the beginning of the resulting word  if ${\bf i}^{(1)} = {\bf i'}^{(1)}$.
  The desired result follows by induction on $r+r'$. 
     
      \end{proof}

 \subsection{A mutation rule for dominant words}

 We now  use Theorem~\ref{thmonoid} (or equivalently Proposition~\ref{prod}) to obtain a mutation rule on the parameters of simple modules corresponding to cluster variables in the setting of \cite{KKKO}. We express it in a vector setting, i.e. in terms of the images of dominant words under the isomorphism~\eqref{isomonoids}. Recall that the image of any dominant word $\mu$ under the isomorphism~\eqref{isomonoids} is the vector $\overrightarrow{\mu}$ whose $i$th coordinate is equal to the multiplicity  of the Lyndon word corresponding to the $i$th positive root (in the decreasing order~\eqref{decrorder}) in the canonical factorization of $\mu$. Such vectors are of size  $r_n$, the number of positive roots  in type $A_n$ ($r_n = n(n+1)/2)$).

     \begin{ex}  
    In type $A_2$, there are three positive roots: $\alpha_2 > \alpha_1 + \alpha_2 > \alpha_1$.
     The word $21$ will be represented by the vector $ ~^t (1,0,1)$.
  
   In type $A_3$ there are six positive roots: $ \alpha_3 > \alpha_2 + \alpha_3 > \alpha_2 >  \alpha_1 + \alpha_2 + \alpha_3 > \alpha_1 +  \alpha_2 >  \alpha_1.$
The word $2312$ will be represented by the vector $~^t (0,1,0,0,1,0)$  and the word $321$ by the vector $ ~^t (1,0,1,0,0,1)$. 
 \end{ex}
 
  \smallskip
  \smallskip
   
  Let us consider a quantum monoidal seed $ \mathcal{S} := ( \{M_i\}_{i \in I} ,B, \Lambda,D)$ in the sense of \cite{KKKO}. Recall that $I$ splits into $I = J_{ex}  \cup J_{fr}$ with the $\{ [M_i] \}_{i \in J_{ex}}$ corresponding to unfrozen variables  and   the $\{ [M_i] \}_{i \in J_{fr}}$ corresponding to frozen variables. For every $i \in I$, let $\mu_i$ be the parameter of the simple module $M_i$ and $\overrightarrow{\mu_i}$ the corresponding vector.
  
\begin{rk} \label{vector}

 \begin{enumerate}
  
   \item The abelian monoid isomorphism~\eqref{isomonoids} naturally extends to an abelian group isomorphism between the respective Grothendieck groups of $(\M,\odot)$ and $(\mathbb{Z}_{\geq 0}^{r_n} , +)$, namely
  \begin{equation} \label{isogp}
  (\G , \odot) \simeq (\mathbb{Z}^{r_n} , +). 
  \end{equation}
 Under this isomorphism,  the inverse in $\G$ of a parameter $\mu \in \M$ corresponds to the opposite vector in $\mathbb{Z}^{r_n}$. For instance the vector corresponding to the generalized parameter $\mjh$ is 
 $$ \overrightarrow{\mjh} = \sum_{1 \leq i \leq n+m} b_{ij} \overrightarrow{\mu_i}. $$
   
    \item The lexicographic order $\leq$ on $\M$ and $\G$ also turns into a (total) ordering on $\mathbb{Z}^{r_n}$ through the above isomorphism: a vector $\overrightarrow{\mu_1}$ is strictly greater than a vector $\overrightarrow{\mu_2}$ if and only if the first non zero component of $ \overrightarrow{\mu_1} - \overrightarrow{\mu_2}$ is positive. 
  \end{enumerate}
  
   \end{rk}
   
   \smallskip
   
   Let $k$ be fixed in $J_{ex}$ and let us look at the mutation in direction $k$ of the seed $\mathcal{S}$. It leads to a new seed $ \mathcal{S}'$ with the same variables except $ M_k$ that has turned into $M'_k$ such that we have a short exact sequence of graded modules:
   
\begin{equation} \label{ses}
     0 \rightarrow q \bigodot_{b_{ik}>0} M_i^{\circ b_{ik}} \rightarrow q^{ \tilde{\Lambda}(M_k,M'_k)} M_k \circ M'_k \rightarrow \bigodot_{b_{ik}<0} M_i^{\circ (-b_{ik})} \rightarrow 0.
   \end{equation}
    
The next statement shows that one can deduce the parameter of the simple module $M'_k$ from the knowledge of the parameters $\mu_i$ and the exchange matrix $B$ of the seed $\mathcal{S}$.
   
     \begin{prop} \label{mutpara}
      
  Let $ \mu'_k$ be the parameter of the simple module $M'_k$ and $  \overrightarrow{\mu'_k}$ be the corresponding vector.
  Then  we have:
  $$ \overrightarrow{\mu'_k} = - \overrightarrow{\mu_k} + max \left( \sum_{b_{ik} > 0} b_{ik} \overrightarrow{\mu_i} , \sum_{b_{ik} < 0} (- b_{ik}) \overrightarrow{\mu_i} \right).$$
   
    \end{prop} 
    
  \begin{proof}
  
  As the real simple modules $M_i$ commute, the modules  $ \bigodot_{b_{ik}>0} {M_i}^{\circ b_{ik}}$  and $ \bigodot_{b_{ik}<0} {M_i}^{\circ (-b_{ik})}$  are simple. Thus they correspond to some dominant words $\mu_{+}$ and $\mu_{-}$. Using Theorem~\ref{KR} (ii), one can write
   \begin{equation*}
  \mu_{+} = max(L(\mu_{+})) 
                = max \left( \bigodot_{b_{ik}>0} {L(\mu_i)}^{\circ b_{ik}} \right) 
                = \bigodot_{b_{ik} > 0}  \mu_i^{\odot b_{ik}}.
   \end{equation*}         
 Under the isomorphism~\eqref{isomonoids} we get
$$ \overrightarrow{\mu_{+}} = \sum_{b_{ik}>0} b_{ik} \overrightarrow{\mu_i} . $$
Now the short exact sequence \eqref{ses}  gives the relation
$$  q^{ \tilde{\Lambda}(M_k,M'_k)} [M_k][M'_k] = q \prod_{b_{ik}>0}  [M_i]^{b_{ik}}  + \prod_{b_{ik}<0}  [M_i]^{-b_{ik}}  $$
 in the Grothendieck ring of the category $R-gmod$. 
 Taking the characters we get 
 $$  q^{ \tilde{\Lambda}(M_k,M'_k)} ch_q(M_k) \circ ch_q( M'_k) =  q ch_q(L(\mu_{+})) + ch_q(L(\mu_{-})).$$
 Looking at the highest weight on both sides of this equality we get
  \begin{align*}
 \mu_k \odot \mu'_k &= max \left( ch_q(M_k) \circ ch_q( M'_k) \right) \\
     &= max \left( max \left( ch_q (L(\mu_{+})) \right)  , max \left( ch_q (L(\mu_{-})) \right) \right).
 \end{align*}
 Applying isomorphism~\eqref{isomonoids}, we get
$$ \overrightarrow{\mu_k} + \overrightarrow{\mu'_k} = max(  \overrightarrow{\mu_{+}},  \overrightarrow{\mu_{-}} ) 
  = max \left( \sum_{b_{ik} > 0} b_{ik} \overrightarrow{\mu_i}  , \sum_{b_{ik} < 0}  (-b_{ik}) \overrightarrow{\mu_i} \right)$$   
which is the desired statement in the image of isomorphism~\eqref{isogp}.
 
  \end{proof}

 \subsection{An example in type $A_3$}

 In this subsection we apply Proposition~\ref{mutpara} to the example of the category $R-gmod$ for a Lie algebra of type $A_3$. This example provides an illustration of Theorem~\ref{mainthm} which will be proved in general type $A_n$ in Section~\ref{initseed}. 
 The category $R-gmod$ corresponds to $\Cw$ with $w = w_0$ the longest element of the Weyl group of $\mathfrak{g}$ (see Section~\ref{qmseedcw}). In type $A_3$ this element can be written as:
 $$ w_0 = s_1s_2s_3s_1s_2s_1.$$
 Theorem~\ref{initseedCw} provides an admissible pair (in the sense of Definition~\ref{KKKOdef6.1}), which gives rise to a quantum monoidal seed for this category. We denote this seed by ${\s}_0^3$. Firstly, one can see that $J_{ex} = \{1,2,3\}$ and $J_{fr}=\{4,5,6\}$ (see Section~\ref{qmseedcw}). The simple modules whose classes are the cluster variables of the seed ${\s}_0^3$ can be computed directly using \cite[Proposition 10.2.4]{KKKO}.  
 
 \begin{lem}
  The seed ${\s}_0^3$ for the category $R-gmod$ in type $A_3$ is given by three unfrozen variables $[L(1)], [L(12)], [L(21)]$ and three frozen variables $[L(123)], L(2312)], [L(321)]$ together with the following exchange matrix:
 $$ B_0 = \begin{pmatrix}  0 & 1 & -1  \\ -1 & 0 & 1  \\ 1 & -1 &  0 \\ 0 & -1 & 0  \\ 0 & 1 & -1  \\ 0 & 0 & 1 \end{pmatrix} .$$ 
 \end{lem}
 
 \begin{proof}
 By definition of the modules $M(k,0)$ defining the underlying admissible pair of the seed ${\s}_0^3$ (see Section~\ref{qmseedcw}), one has
 
$\begin{array}{lll}
 M(1,0)  &= &M(s_1 \omega_1 , \omega_1) ) \\
 M(2,0) &= &M(s_1 s_2 \omega_2 , \omega_2)  \\
 M(3,0) &= &M(s_1 s_2 s_1 \omega_1 , \omega_1) \\
 M(4,0) &= &M(s_3s_1s_2s_1 \omega_1 , \omega_1)  \\
 M(5,0) &= &M(s_2s_3s_1s_2s_1 \omega_1 , \omega_1)  \\
 M(6,0) &= &M(s_1s_2s_3s_1s_2s_1 \omega_1 , \omega_1).
\end{array}$
 
 Using \cite[Proposition 10.2.4]{KKKO}, one gets $M(1,0)=L(1)$, $M(2,0)=hd(L(1) \circ L(2))=L(12)$, $M(3,0) = hd(L(2) \circ L(1))=L(21)$. The computations are similar for $M(k,0) , k \in \{4,5,6\}$.
 
 \end{proof}
 
The (ungraded) short exact sequences corresponding to the mutations in each of the three exchange directions can be written as follows:
 $$ 0 \rightarrow L(21) \rightarrow L(1) \circ L \rightarrow L(12) \rightarrow 0. $$
 $$ 0 \rightarrow L(1) \circ L(2312) \rightarrow L(12) \circ M  \rightarrow L(21) \circ L(123) \rightarrow 0. $$ 
$$ 0 \rightarrow L(12) \circ L(321)  \rightarrow L(21) \circ N  \rightarrow L(1) \circ L(2312) \rightarrow 0. $$
Let $\lambda$ (resp. $\mu$, $\nu$) be the parameters of the simple module $L$ (resp. $M$, $N$). We can compute these parameters using  Proposition~\ref{mutpara}. For instance consider the second of the above exact sequences. Then with the notations of Remark~\ref{vector}, a straightforward computation gives the parameter of $M$ as $\overrightarrow{\mu} =  ~^t (0,1,0,0,0,1)$. Hence $M=L(231)$. 
 
 In the same way one can compute $L=L(2)$ and $N=L(312)$.

  \smallskip
  \smallskip
  
 Let ${\s}_1$ be the seed obtained from the seed ${\s}_0$ by mutation in the first direction. 
   One can now show that ${\s}_0$ is compatible in the sense of Definition~\ref{compat} and ${\s}_1$ is not. 
  
 First we write the exchange matrix for ${\s}_1$:
 $$ B_1  = \begin{pmatrix}  0 & -1 & 1  \\ 1 & 0 & 0  \\ -1 & 0 &  0 \\ 0 & -1 & 0  \\ 0 & 1 & -1  \\ 0 & 0 & 1  \end{pmatrix}.$$ 
 Then for ${\s_0}$, the images under isomorphism~\eqref{isogp}  of the generalized parameters $\mjh$ are :
$$ \overrightarrow{\hat{\mu}_1} = ~^t (0,0,1,0,1,1) \quad  \overrightarrow{\hat{\mu}_2} = ~^t (0,1,-1,1,1,0)\quad  \overrightarrow{\hat{\mu}_3} = ~^t (1,-1,1,0,0,0) $$
and for ${\s_1}$ we get:
$$  \overrightarrow{\hat{\mu}_1} = ~^t (0,0,-1,0,1,-1) \quad  \overrightarrow{\hat{\mu}_2} = ~^t (0,1,-1,-1,1,0) \quad  \overrightarrow{\hat{\mu}_3} = ~^t (1,-1,2,0,-1,1). $$
 Combining Remark~\ref{vector}(2) and Remark~\ref{rkpsi}, one can see  that ${\s}_0$ is compatible in the sense of Definition~\ref{compat}  but it is not the case for ${\s}_1$.

 \bigskip

 More generally, given an initial quantum monoidal seed $(\{M_i\}_i , B)$,  Proposition~\ref{mutpara} allows us to compute explicitly the parameters of the simple modules appearing when mutating the initial seed an arbitrary number of times in any directions.

 \section{A compatible seed for $R-gmod$ in type $A$}
   \label{initseed}

 In this section we compute in type $A_n$ the parameters of the simple modules of the monoidal seed ${\s}_0^n$  arising from the construction of \cite{KKKO} (see Subsection~\ref{qmseedcw}) for the category $R-gmod$.

 \subsection{Statements of the main theorems}
  \label{statements}
  
  In this subsection we state the two main results of this paper, Theorems~\ref{initpara}  and ~\ref{mainthm}. Recall that ${\A}_q( \mathfrak{n}) = {\A}_q( \mathfrak{n}(w_0))$ where $w_0$ is the longest element of the Weyl group of $\mathfrak{g}$. The category $R-gmod$ coincides with the category $\mathcal{C}_{w_0}$ (see Section~\ref{qmseedcw}). In type $A_n$ we have 
 $$w_0 = (1 \dots n)(1 \ldots (n-1)) \cdots (12)(1).$$
 Recall that $r_n := n(n+1)/2$ stands for the length of $w_0$. In the category $R-gmod =  \mathcal{C}_{w_0}$ in type $A_n$, Theorem~\ref{initseedCw} provides an admissible pair (in the sense of Definition~\ref{KKKOdef6.1}), which gives rise to a quantum monoidal seed for this category. We denote this seed by ${\s}_0^n$.
 
 \smallskip
 \smallskip
 
 Our first main result is the following:
  
   \begin{thm} \label{initpara}

  The cluster variables of the seed ${\s}_0^n$ can be explicitly described in terms of parameters as follows:

$ \begin{array}{ccccc}
  [L(1)]  & {} & {} & {} & {} \\\relax
  [L(12)] & [L((2)(1))] & {} & {} & {} \\\relax
  [L(123)] & [L((23)(12))] & [L((3)(2)(1))] & {} & {} \\\relax
  \vdots & \vdots & {} & {} & {} \\\relax
  [L(1 \ldots k)] & [L \left( (2 \ldots k)(1 \ldots k-1) \right)] & \cdots & [L((k)\cdots (1))] & {} \\\relax
  \vdots  & \vdots & {} & {} & {} \\\relax
  [L(1 \ldots n)] & [L \left( (2 \ldots n)(1 \ldots n-1) \right)]  & \cdots & \cdots &  [L((n) \cdots (1))].  
\end{array}$

 The set of frozen variables corresponds to the last line and the set of unfrozen variables consists in the union of lines $1 \ldots n-1$.

 \end{thm}
 
  The three following sections are devoted to some intermediate steps for the proof of Theorem~\ref{initpara}. Recall from Section~\ref{qmseedcw} that $J_{fr}$ denotes the index set of the frozen variables of the monoidal seed ${\s}_0^n$. In Section~\ref{initseedRgmod}, we prove that  $ |J_{fr}| = n$. We show that the knowledge of the dominant words of the $M_j , j \in J_{fr}$ is sufficient to recover the whole seed ${\s}_0^n$.  Section~\ref{weights} is devoted to the computation of the weights of the $M_j , j \in J_{fr}$. These weights are determined by the construction of \cite{KKKO} (see Section~\ref{qmseedcw}). In Section~\ref{frozenpara}, we use the fact that for any $j \in J_{fr}$, the module $M_j$ necessarily commutes with any other simple module in $R-gmod$. This strongly constrains the form of the corresponding dominant word. Together with the weights obtained in Section~\ref{weights}, we find at most $n$ possible dominant words, which is exactly the number of frozen variables computed in Section~\ref{initseedRgmod}. Hence we get a bijection between these parameters and the set of modules $ \{ M_j , j \in J_{fr} \}$. 
  
    We complete the proof in Section~\ref{proofs} by determining which parameter corresponds to every simple module, which is more precise than just a global bijection. The key argument is provided by Theorem~\ref{mutpara}.
    
    \smallskip
    \smallskip
    
    The following statement is our second main result. We deduce it from Theorem~\ref{initpara}.
   
   \begin{thm} \label{mainthm}

   The seed ${\s}_0^n$ is compatible in the sense of Definition~\ref{compat}.

 \end{thm} 

 In particular Conjecture~\ref{conjecture} holds for the category $R-gmod$ in type $A_n$.

 \subsection{Initial seed for $R-gmod$}
  \label{initseedRgmod}

  For $n \geq 1$, we consider a Lie algebra $\mathfrak{g}$ of type $A_n$. We let $\{\alpha_1, \ldots, \alpha_n \}$ denote the simple roots and $Q_{+}^{n} := \bigoplus_{i=1..n} \mathbb{N}\alpha_i$. We also denote $\Delta_{+}^{n}$ the set of positive roots, ${R-gmod}^n$ the category of (graded) finite dimensional representation of the quiver Hecke algebras associated with $\mathfrak{g}$, and  $\M^{n}$ the set of dominant words in bijection with the set of simple objects in 
  ${R-gmod}^n$ (up to isomorphism). There is a canonical embedding $\iota_{n}^{m}$ of $\M^{n}$ into $\M^{m}$ for any $m \geq n$. In particular the set of simple objects in ${R-gmod}^n$ is naturally included into the set of simple objects in ${R-gmod}^{m}$. We again denote $\iota_{n}^{m}$ this inclusion.

  Let $J_{ex}^n$ (resp. $J_{fr}^n$) denote the index set of the unfrozen variables (resp. frozen variables) of the seed ${\s}_0^n$ for every $n \geq 1$. We also set $J^n := J_{ex}^n \cup J_{fr}^n$.
  In order to prove Theorem~\ref{initpara}, we begin by determining the sets $J_{ex}^n$ and $J_{fr}^n$. We point out an inductive property of this seed: the set $\{ M_i , i \in J_{ex}^n \}$ coincides with the set $\{ \iota_{n-1}^{n} (M_i) , i \in J^{n-1} \}$.  
  
  \begin{prop} \label{Aseed}
  The cluster variables of the seed  ${\s}_0^n$ split into the following  exchange and frozen sets:
  
   \begin{enumerate}
     \item There are $n$ frozen variables in ${\s}_0^n$, which correspond to the classes of the last $n$ modules $M(r_{n-1} +1,0) , \ldots , M(r_n , 0)$.
     \item The set $\{ M_i , i \in J_{ex}^n \}$  coincides with the union of the sets $\{   \iota_{k}^{n} (M_i) , i \in J_{fr}^k \} , 1 \leq k \leq n-1$.
   \end{enumerate}
   
    \end{prop} 
   
 \begin{proof} 
 
  For the first statement, write $w_0 = (s_1 \ldots s_n)(s_1 \ldots s_{n-1}) \cdots (s_1s_2)(s_1) = s_{r_n} \cdots s_1$. Then for any $l \in \{ 1 , \ldots , r_{n-1} \}$, the letter $s_l$ is in $ \{ 1 , \ldots , n-1 \}$ and this letter obviously appears again in the word $w_0$ as $s_{l'}$ for some $l' > l$. In other words, $l_{+} \leq r_n$ and thus $l \in J_{ex}^n$. Conversely, if $ l \in \{ r_{n-1} +1, r_n \}$, then all the letters $s_{l'} , l' > l$ are distinct from $s_l$ and thus $l_{+} = r+1$.   Hence one has 
  $$ J_{ex}^n = \{ 1, \ldots , r_{n-1} \} \quad \text{and} \quad  J_{fr}^n = \{ r_{n-1} +1 , \ldots , r_n \}.$$
  In particular the modules corresponding to the  frozen variables of the seed ${\s}_0^n$ can be written as 
 
$  \begin{array}{rcl}
  M(r_{n-1}+1,0) &= &M \left( s_1(s_2s_1) \cdots (s_{n-1} \ldots s_1)s_n . \omega_n , \omega_n \right) \\
  \vdots &\vdots &\vdots \\
  M(r_n,0) &= &M \left( s_1(s_2s_1) \cdots (s_{n-1} \ldots s_1)(s_n \ldots s_1) . \omega_1 , \omega_1 \right).
  \end{array}$
 
 The second statement follows from the first one applied to the seeds ${\s}_0^k , 1 \leq k \leq n-1$. Indeed, one can write in the same way the modules $M_i , i \in J_{fr}^{n-1}$ as:

 $\begin{array}{c}
 M \left( s_1(s_2s_1) \cdots (s_{n-2} \ldots s_1)s_{n-1} . \omega_{n-1} , \omega_{n-1} \right) \\
  \vdots \\
M \left( s_1(s_2s_1) \cdots (s_{n-2} \ldots s_1)(s_{n-1} \ldots s_1) . \omega_1 , \omega_1 \right).
  \end{array}$
 
 The images via $\iota_{n-1}^{n}$ of these modules respectively coincide with the modules $M(r_{n-2}+1,0), \ldots , M(r_{n-1},0)$, whose classes are exactly the  last $n-1$ unfrozen variables of the seed ${\s}_0^n$. Iterating this, we conclude that the set $\{M_l, l \in J_{ex}^n \}$ is the union of the sets $\{\iota_{k}^{n} (M_l), l \in J_{fr}^k \}$ , $1 \leq k \leq n-1$. 
 
  \end{proof} 

As a direct consequence of the previous Proposition, it suffices to compute the parameters of the modules corresponding to the frozen variables of the seed ${\s}_0^n$. This is what we focus on in the next two subsections.

 \subsection{Weights of the simple modules $M(r_{n-1}+k,0) , 1 \leq k \leq n$}
 \label{weights}

  From now on, the integer $n$ is fixed. We write $J_{fr}$ for $J_{fr}^n$. By Proposition~\ref{Aseed}, the simple modules corresponding to the frozen variables of the seed ${\s}_0^n$ are the $M(r_{n-1}+k,0) , 1 \leq k \leq n$. For simplicity we set $M_k := M(r_{n-1}+k , 0)$ for any $1 \leq k \leq n $. This subsection is devoted to the computation of the weights of the simple modules $M_k$, i.e. the elements $\beta_k$ such that $M_k \in R(\beta_k)-gmod$ for every $1 \leq k \leq n$. Our main tool is the definition of the modules $M(l,0)$ from \cite{KKKO} (see Definition~\ref{determimodbis}).

   \begin{prop} \label{occurrences}
For each $ 1 \leq k \leq n/2 $, the two modules $M_k$ and $M_{n-k+1}$ both belong to the subcategory $R(\alpha_n + 2 \alpha_{n-1} +\cdots+ k\alpha_{n-k+1}+ \cdots + k \alpha_{k} + \cdots + 2 \alpha_{2} + \alpha_{1})-mod$.
 \end{prop}

 \begin{proof}

For $1 \leq l \leq n$, we have 
$$ M_l = M(r_{n-1} +l,0) = M \left( s_1(s_2s_1) \cdots (s_{n-1} \ldots s_1)(s_n \ldots s_k) \omega_k , \omega_k \right) \text{ where $k := n-l+1$}. $$
   One computes $ \zeta_k := s_1(s_2s_1) \cdots (s_{n-1} \ldots s_1)(s_n \ldots s_k) \omega_k$. The weight of $M_l$ is given by $\omega_k - \zeta_k$ (see Corollary~\ref{determimod}). 
\begin{align*}
 \zeta_k  &= s_1(s_2s_1) \cdots (s_{n-1} \ldots s_1) ( \omega_k - (\alpha_n + \cdots + \alpha_k) ) \\
          &= s_1(s_2s_1) \cdots (s_{n-2} \ldots s_1)( \omega_k - (\alpha_n + 2 \alpha_{n-1} + \cdots + 2 \alpha_k + \alpha_{k-1}) ) \\
         \begin{split}
          &= s_1(s_2s_1) \cdots (s_{n-3} \ldots s_1)( \omega_k - (\alpha_n + 2 \alpha_{n-1} + 3 \alpha_{n-2}+ \cdots + \\
            &  \qquad \qquad \qquad \qquad \qquad \qquad \qquad \qquad + \cdots + 3 \alpha_k + 2 \alpha_{k-1} + \alpha_{k-2}) ).
          \end{split}
\end{align*}
  If $2k \leq n$ then by iterating we get 
\begin{equation*} 
  \begin{split}
  \zeta_k = s_1(s_2s_1) \cdots (s_{n-k} \ldots s_1) &( \omega_k - (\alpha_n + 2 \alpha_{n-1}  + \cdots + k \alpha_{n-k+1}+ \cdots + \\
   & \qquad \qquad \qquad + \cdots +  k \alpha_k + \cdots + 2 \alpha_2 + \alpha_1) )
  \end{split}
\end{equation*}
but   $\omega_k - (\alpha_n + 2 \alpha_{n-1} + \cdots + k \alpha_{n-k+1}+ \cdots + k \alpha_k +  \cdots + 2 \alpha_2 + \alpha_1)$ is invariant under the action of $s_1, \ldots , s_{n-k}$. 
  Hence $$\zeta_k =  \omega_k - (\alpha_n + 2 \alpha_{n-1} + \cdots + k \alpha_{n-k+1}+ \cdots + k \alpha_k + \cdots + 2 \alpha_{2} + \alpha_{1}).$$ 
  If $2k > n$ then by iterating we get 
  \begin{align*}
         \begin{split}
        \zeta_k &= s_1(s_2s_1) \cdots (s_k \ldots s_1) ( \omega_k - ( \alpha_n + 2 \alpha_{n-1} + \cdots + (n-k) \alpha_{k+1} \\
        & \qquad  \qquad  \qquad  \qquad \qquad \qquad \qquad + (n-k) \alpha_k + \cdots + 2 \alpha_2 + \alpha_{2k-n+1} ) ) 
          \end{split} \\
            \begin{split}
      &= s_1(s_2s_1) \cdots (s_{k-1} \ldots s_1).  ( \omega_k - (\alpha_n + 2 \alpha_{n-1} + \cdots + (n-k) \alpha_{k+1} \\
       &  \qquad \qquad \qquad \qquad  + (n-k+1)\alpha_k + (n-k)\alpha_{k-1} \cdots + 2 \alpha_{2} + \alpha_{2k-n}) )
                  \end{split}  \\
           &= \ldots \\
           \begin{split}
  &= s_1(s_2s_1) \cdots (s_{n-k} \ldots s_1).  ( \omega_k - (\alpha_n + 2 \alpha_{n-1} + \cdots + (n-k) \alpha_{k+2}  \\
   & \qquad +  (n-k+1) \alpha_{k+1} + \cdots +  (n-k+1)\alpha_{n-k+1} + \cdots + 2 \alpha_{2} + \alpha_{1}) )
                  \end{split}
\end{align*} 
 and $\omega_k - (\alpha_n + 2 \alpha_{n-1} + \cdots + (n-k+1) \alpha_{k+1}+ \cdots (n-k+1)\alpha_{n-k+1} + \cdots + 2 \alpha_{2} + \alpha_{1})$ is invariant under the action of $s_1 , \ldots , s_{n-k}$. 
  Hence we get 
 \begin{align*}
     \zeta_k &=  \begin{cases} 
    \begin{split}
       & \omega_k - (\alpha_n + 2 \alpha_{n-1} + \cdots + k \alpha_{n-k+1}+ \cdots +  \\
        &  \qquad \qquad \qquad \qquad + \cdots + k \alpha_k + \cdots + 2 \alpha_{2} + \alpha_{1})     \qquad  \quad  \text{if $2k \leq n$,} 
         \end{split} \\
   \begin{split}
     & \omega_k - (\alpha_n + 2 \alpha_{n-1} + \cdots + (n-k+1)\alpha_{k}+ \cdots +  \\
       & \qquad   + \cdots + (n-k+1) \alpha_{n-k+1} + \cdots + 2 \alpha_{2} + \alpha_{1}) \qquad \text{if $2k > n$.} 
      \end{split}
 \end{cases}
      \\
             &= \begin{cases} 
     \begin{split}
       & \omega_k - (\alpha_n + 2 \alpha_{n-1} + \cdots + k \alpha_{l}+ \cdots +  \\
        & \qquad \qquad \qquad \qquad \quad + \cdots + k \alpha_k + \cdots + 2 \alpha_{2} + \alpha_{1}) \qquad    \text{if $k<l$,} 
        \end{split}\\
    \begin{split}
       & \omega_k - (\alpha_n + 2 \alpha_{n-1} + \cdots l\alpha_{k}+ \cdots + \\
 & \qquad  \qquad \qquad \quad \qquad + \cdots +   l \alpha_{l} + \cdots + 2 \alpha_{2} + \alpha_{1} ) \qquad \text{if $k \geq l$.} 
     \end{split}
      \end{cases}
\end{align*}
 Hence for each $ 1 \leq k \leq n/2 $, the two modules $M_k$ and $M_{n-k+1}$ both belong to the subcategory $R(\alpha_n + 2 \alpha_{n-1} + \cdots + k\alpha_{n-k+1}+ \cdots + k \alpha_{k} + \cdots + 2 \alpha_{2} + \alpha_{1})-mod$.

 \end{proof}

 Let us now fix $ 1 \leq k \leq n$ and consider the parameter $\mu_k$ of the simple module $ M_k$. Let $m_k$ be the length of $\mu_k$. 
 Since $k$ and $n-k+1$ play symmetric roles, we assume from now on that $k \leq n/2$.   
 
  The following statement is a direct consequence of the previous proposition. 
  
   \begin{cor}  \label{pairing}
    For any $1 \leq i \leq n $,
    $$  (\mu_k , i) = 
         \begin{cases}  1  &\text{if $i=k$ or $i=n-k+1$,} \\
                                  0  &\text{otherwise.}
        \end{cases}                          
 $$
 \end{cor}
 
  \begin{proof}
   For $1 \leq i \leq k-1$ or $n-k+2 \leq i \leq n$, by Proposition~\ref{occurrences} there are $i-1$ (resp. $i,i+1$) occurrences of the letter $i-1$ (resp. $i,i+1$) in the word $\mu_k$ and hence $(\mu_k , i) = 2i -(i-1) -(i+1) = 0$. 
   
   If $k+1 \leq i \leq n-k$ then by Proposition~\ref{occurrences} each of the letters $i-1,i,i+1$ appears $k$ times and hence $(\mu_k , i) = 2i -i -i =0$. 
   
  Finally  if $i=k$ then by Proposition~\ref{occurrences} there are $k$ occurrences of the letters $k,k+1$ and $k-1$ occurrences of the letter $k-1$ which gives $(\mu_k , i)=2k -k -(k-1)=1$. If $i=n-k+1$ then there are $k$ occurrences of the letters $i-1,i$ and $k-1$ occurrences of the letter $i+1$ and thus $(\mu_k , i)=2k -k -(k-1)=1$.
  \end{proof}
 
 In particular, one can compute  the quantities $\Lambda(M_k,L(i))$ for any $1 \leq i \leq n$:
 
  \begin{cor} \label{Lambda}
    For any $1 \leq i \leq n $, let $N_i$ be the number of occurrences of the letter $i$ in the word $\mu_k$. Let $s_i$ and $s'_i$ be the integers such that (see Remark~\ref{rklambda})
    $$ \Lambda(L(i),M_k) = -(\mu_k , i) + 2N_i - 2s_i \quad  \text{and} \quad  \Lambda(M_k,L(i)) = -(\mu_k , i) + 2N_i - 2s'_i .$$     
   Then one has $s_i=s'_i=N_i$ if $i \notin \{k,n-k+1 \}$ and either $s_i=N_i , s'_i = N_i - 1$ or $s_i = N_i - 1 , s'_i = N_i $ if $i \in \{k,n-k+1 \}$.
 \end{cor}
 
 \begin{proof}   
  By Corollary~\ref{pairing}, the quantity $\Lambda(L(i),M_k)$ can be written as 
    $$  \Lambda(L(i),M_k) = 
         \begin{cases}  2N_i -1 -2s_i  &\text{if $i=k$ or $i=n-k+1$,} \\
                                  2N_i -2s_i  &\text{otherwise.}
        \end{cases}                        
 $$
and similarly for $\Lambda(M_k,L(i))$ with $s'_i$. 
 As $M_k$ commutes with $L(i)$, one has by Lemma~\ref{comutlambda}:
$$   s_i + s'_i = 
          \begin{cases}  2N_i -1 &\text{if $i=k$ or $i=n-k+1$,} \\
                                   2N_i      &\text{otherwise.}
          \end{cases}
          $$
  As the integers $s_i$ and $s'_i$ are always smaller than $N_i$, one gets the desired result. 
  
   \end{proof}

  \begin{rk} \label{rklambda}
 In the following we will make several computations of some $\Lambda(M,N)$ for various simple real objects $M,N$ in $R-gmod$ in order to check commutation between these modules. For any $\beta \in Q_{+}$, any simple (left) $R(\beta)$-module $M$  is cyclic, i.e. is isomorphic to $R(\beta).u$ for some $u \in M$. We will refer to any such vector $u$ in $M$ as a generating vector in $M$. Now let $\beta, \gamma \in Q_{+}$, $M$ a simple $R(\beta)$-module, and $N$  a simple  $R(\gamma)$-module. As the morphism 
$$\begin{array}{ccccc}
   & M \otimes N & \longrightarrow & N \circ M \\
   & u \otimes v & \longmapsto & \varphi_{w[n,m]}(v \otimes u)
    \end{array} $$
is $R(\beta) \otimes R(\gamma)$-linear, computing the map $R_{M,N}$ is equivalent to computing the action of  $\varphi_{w[n,m]} \in R(\beta + \gamma)$ on the tensor product of generating vectors $u$ and $v$ for $M$ and $N$.

 Now let $u_z := 1 \otimes u \in M_z$ and  let $\tilde{s}$ be the valuation of the polynomial in $z$ given by $\varphi_{w[n,m]}.(v \otimes u_z)$. As the actions of the generators $x_i, \tau_k$ and $e(\nu)$ can only make the degree in $z$ increase, the image of the map $R_{M_z,N}$ is contained in $z^{\tilde{s}} N \circ M_z$. Moreover, by definition of $\tilde{s}$, $\varphi_{w[n,m]}.(v \otimes u_z)$ contains a nonzero term of degree $\tilde{s}$  hence it does not belong to $z^k N \circ M_z$ for any $k > \tilde{s}$. Hence $\tilde{s}$ coincides with $s$ in Definition~\ref{renormRmat}. Thus in what follows, for any simple  $R(\beta)$-module $M$ and any simple  $R(\gamma)$-module $N$, we will always fix some choices of generating vectors $u \in M$ and $v \in N$ and  write 
$$ \Lambda(M,N) = -(\beta, \gamma) + 2 (\beta, \gamma)_n -2s $$
with $s$ being the valuation  of the polynomial in $z$ given by $\varphi_{w[n,m]}.(v \otimes u_z)$. 
 \end{rk}
 
 \subsection{Dominant words associated to frozen variables in $R-gmod$}
  \label{frozenpara}
 
  In this subsection, we compute the dominant words associated to the frozen variables for the category $R-gmod$ in type $A_n$. As in the previous subsection, we fix $k$ such that $1 \leq k \leq n/2$  and we consider $M_k = M(r_{n-1}+k,0)$ the simple module whose isomorphism class is the $k$th frozen variable in the seed ${\s}_0^n$ constructed in \cite{KKKO}. We use the fact that $M_k$ commutes with all the simple modules in $R-gmod$. In particular, it commutes with all the cuspidal modules $L(i), 1 \leq i \leq n$. Together with the form of the weight of $M_k$ given by Proposition~\ref{occurrences} and Corollay~\ref{Lambda}, this leads to only $n$ possible dominant words. As there are exactly $n$ frozen variables in the seed ${\s}_0^n$ (see Proposition~\ref{Aseed}), we get a bijection between the possible parameters and the frozen variables for $R-gmod$.
  
     For every $1 \leq i \leq n$, the algebra $R(\alpha_i)$ is generated by one generator $x_i$ and one generator $e(i)$ commuting with each other (with the notation of Section~\ref{qha}, the set $Seq(\alpha_i)$ is a singleton consisting in the word reduced to a single letter $i$). Recall from Section~\ref{irredKLR} that for every $1 \leq i \leq n$, the cuspidal module $L(i)$ is a one dimensional vector space spanned by a generating vector $v_i$ with action of $R(\alpha_i)$ given by:
$$  x_i \cdot v_i = 0, \quad e(i) \cdot v_i = v_i. $$
  It is the only simple object in the category $R(\alpha_i)-mod$. 

   As above we let $\mu_k$ denote the parameter of the simple module $M_k$ and $m_k$ the length of $\mu_k$. In what follows we will write the word $\mu_k$ as 
  $$\mu_k = h_1 , \ldots , h_{m_k}.$$
  Note that in this setting the $h_j$ are the \textit{letters} of the word $\mu_k$, whereas we use bold letters ${\bf i}_l$ to refer to \textit{Lyndon words} in the canonical factorization of $\mu_k$ (see after Remark~\ref{littlerk} below).
  
   As the module $M_k$ is simple and real, Lemma~\ref{comutlambda} shows that checking  its commutation with any other simple module $L$ is equivalent to computing the quantities $\Lambda(L,M_k)$ and $\Lambda(M_k,L)$. When $L=L(i)$ for some $i \in \{1 , \ldots , n \}$, these quantities are given by Corollary~\ref{Lambda}. Thus as explained in Remark~\ref{rklambda} above, once fixed a generating vector $u$ for $M_k$, we will compute the valuations $s_i$ (resp.$s'_i$) of the polynomial functions $\varphi_{w[m_k,1]}(u \otimes (v_i)_z) =   \varphi_1 \cdots \varphi_{m_k} (u \otimes (v_i)_z)$ (resp. $\varphi_{w[1,m_k]}(v_i \otimes u_z) =  \varphi_{m_k} \cdots \varphi_1 (v_i \otimes u_z)$) for various choices of $i \in \{1, \ldots ,n \}$. 
 We fix once for all a generating vector $u$ (resp. $v_i, 1 \leq i \leq n$) for $M_k$ (resp. $L(i), 1 \leq i \leq n$).
 
 We begin by showing that there are only two possibilities for the first letter of $\mu_k$.

   \begin{lem} \label{lem4}

Let $p=h_1$ denote the first letter of $\mu_k$.
 The letter $p$ is equal either to $k$ or $n-k+1$.

  \end{lem}

  \begin{proof}

With the same notations as in Corollary~\ref{Lambda},  we show that $s_p \leq N_p-1$.  
   \begin{align*}
     \varphi_{1} \cdots \varphi_{m_k}.(u \otimes (v_p)_z)  &=  \varphi_{1} \cdots \varphi_{m_k}e(h_1, \ldots , h_{m_k},p).(u \otimes (v_p)_z) \\
     &= \varphi_{1} e(h_1,p,h_2, \ldots , h_{m_k}) \varphi_2 \cdots \varphi_{m_k}.(u \otimes (v_p)_z) \\
     &= (\tau_1 (x_1 - x_2) + 1) \varphi_2 \cdots \varphi_{m_k}.(u \otimes (v_p)_z).
   \end{align*}
    The operator $x_1$ commutes with $\varphi_2, \ldots, \varphi_{m_k}$ and acts trivially on $u$. Moreover, $x_2  \varphi_2 \cdots \varphi_{m_k} =  \varphi_2 \cdots \varphi_{m_k} x_{m_k +1}$ (see for example \cite[Lemma 1.3.1]{KKK}) and $x_{m_k +1} \cdot (u \otimes (v_p)_z) = z (u \otimes (v_p)_z)$. Hence we get 
  $$  \varphi_{1} \cdots \varphi_{m_k}.(u \otimes (v_p)_z) = -z. \tau_1 . \varphi_{2} \cdots \varphi_{m_k}.(u \otimes (v_p)_z) + \varphi_{2} \cdots \varphi_{m_k}.(u \otimes (v_p)_z) . $$ 
   The operator $\varphi_{2} \cdots \varphi_{m_k}$ acts non-trivially on $(u \otimes (v_p)_z)$ (as the renormalized R-matrix $r_{L(p),M_k}$ never vanishes).  Thus  $\varphi_{2} \cdots \varphi_{m_k}.u$ is a non zero polynomial function, and the above equality implies that $s_p$ is equal to its valuation. This polynomial function has degree less than $N_p-1$, as the only operators $\varphi_j$ that can make the degree rise are the ones corresponding to an occurrence of $p$ in $\mu_k$. 
   
This implies  $s_p \leq N_p-1$. Hence by Corollary~\ref{Lambda}, $p \in \{k,n-k+1 \}$. 

 \end{proof}

\begin{rk} 
 With the same proof, one can show that the last letter of the word $\mu_k$ is either $k$ or $n-k+1$ as well. 
 \end{rk}

   \begin{lem} \label{endLyndon}

      \begin{enumerate}[(i)]
        \item For each $1 \leq k' \leq k$, there is exactly one Lyndon word  ending with $n-k'+1$  in the canonical factorization of $\mu_k$. 

         \item Moreover, denoting by ${\bf j}_{n-k'+1}$ the unique Lyndon word ending with $n-k'+1$ (for each $1 \leq k' \leq k$), one has $ {\bf j}_n > \cdots > {\bf j}_{n-k+1}$.
       \end{enumerate}

\end{lem}

  \begin{proof}

We prove the first statement by induction on $k'$. We know that the letter $n$ appears exactly once; thus there is a unique Lyndon word ${\bf j}_n$ containing (and thus ending with) $n$, which proves the statement for $k'=1$. Suppose $k \geq 2$ and (i) holds for $1 \leq k' < k$. Denote by ${\bf j}_n , \ldots , {\bf j}_{n-k'+1}$ the Lyndon words respectively ending with $n , \ldots , n-k'+1$. Their first letters are all smaller than $p$ and in particular smaller than $n-k+1$ by Lemma~\ref{lem4}. Hence they all contain the letter $n-k'$, which makes $k'$ occurrences of this letter. By Proposition~\ref{occurrences}, $n-k'$ has to appear $k'+1$ times in the word $\mu_k$. Hence there is a unique Lyndon word ${\bf j}_{n-k'}$ containing $n-k'$ but none of the letters  $n , \ldots , n-k'+1$, which means that this Lyndon word ends with $n-k'$. Thus the first statement holds by induction. 

\bigskip

 For the second statement, let $m \in \{n-k+1 , \ldots , n \}$ such that ${\bf j}_m$ is the smallest of the ${\bf j}_l$;  this is equivalent to saying that it is  the last (among the ${\bf j}_l$) to appear in the canonical factorization of $\mu_k$. 

Note that the first statement implies that, for each $1 \leq k' \leq k$, the letter $n-k'+1$ appears $k'$ times among the Lyndon words  ${\bf j}_n , \ldots , {\bf j}_{n-k+1}$: once in each of ${\bf j}_n , \ldots , {\bf j}_{n-k'+1}$, and none in the others. Together with Proposition~\ref{occurrences}, this implies that the letters  $n-k'+1$ ($1 \leq k' \leq k$) do not appear in any other Lyndon word of the canonical factorization of $\mu_k$. Hence denoting by $i$ the position of the last letter of ${\bf j}_m$ in the word $\mu_k$, one has $h_i = m$ and $ h_{i+1} , \ldots , h_{m_k} < m$. 
Thus one has
 \begin{align} \label{rmat}
 \varphi_{m_k} \cdots \varphi_1.(v_m \otimes u_z) &=  \tau_{m_k} \cdots \tau_{i+1} (\tau_i (x_i - x_{i+1})+1)  \varphi_{i-1} \cdots \varphi_1.(v_m \otimes u_z)  \nonumber \\
               &= \tau_{m_k} \cdots \tau_{i+1} \varphi_{i-1} \cdots \varphi_1 . (v_m \otimes u_z) \pm z . \tau_{m_k} \cdots \tau_{i+1} \tau_i \varphi_{i-1} \cdots \varphi_1 . (v_m \otimes u_z).
 \end{align}
 We denote $Q(z)$ the first term 
 $$ \tau_{m_k} \cdots \tau_{i+1} \varphi_{i-1} \cdots \varphi_1 . (v_m \otimes u_z). $$
  We show that if  $m \geq n-k+2$, then $Q(z)$ is non zero.
 As the renormalized R-matrix $r_{{M_k},L(m)}$ does not vanish, the action of the operator $\varphi_{i-1} \cdots \varphi_1$  on $(v_m \otimes u_z)$ is a non zero polynomial function in $z$ of degree less than $N_m-1$. Consider now the action of $\tau_{m_k}$ on $Q(z)$:
\begin{align}
 \tau_{m_k} . Q(z) &=\tau_{m_k} . \tau_{m_k} \cdots \tau_{i+1} \varphi_{i-1} \cdots \varphi_1 . (v_m \otimes u_z)  \nonumber \\
      &= \tau_{m_k} . \tau_{m_k} \cdots \tau_{i+1} \varphi_{i-1} \cdots \varphi_1 . e(m \mu_k) . (v_m \otimes u_z)   \nonumber \\
      &= \tau_{m_k} . \tau_{m_k} e(s_{m_k -1} \cdots s_{i+1} s_{i-1} \cdots s_1 . m \mu_k) . \tau_{m_k -1} \cdots \tau_{i+1} \varphi_{i-1} \cdots \varphi_1 . (v_m \otimes u_z)  \nonumber \\
      &= \tau_{m_k}^2 . e(h_1 , \ldots , h_{m_k -1}, m ,h_{m_k}) . \tau_{m_k -1} \cdots \tau_{i+1} \varphi_{i-1} \cdots \varphi_1 . (v_m \otimes u_z)  \nonumber \\
      &= e(h_1 , \ldots , h_{m_k -1}, m ,h_{m_k}) . \tau_{m_k -1} \cdots \tau_{i+1} \varphi_{i-1} \cdots \varphi_1 . (v_m \otimes u_z)  \quad \text{as $h_{m_k} \leq m-2$}  \nonumber \\
      &= \tau_{m_k -1} \cdots \tau_{i+1} \varphi_{i-1} \cdots \varphi_1 . (v_m \otimes u_z) \nonumber
\end{align} 
 Similarly, all the letters in position $i+1, \ldots , m_k$ are less than $n-k$ and in particular they are less than $m-2$. The same argument can be applied to $\tau_{m_k -1} , \ldots , \tau_{i+1}$ and thus we get
$$ \tau_{i+1} \cdots  \tau_{m_k} . Q(z) = \tau_{i+1} \cdots  \tau_{m_k} . (\tau_{m_k} \cdots  \tau_{i+1} \varphi_{i-1} \cdots \varphi_{1}.(v_m \otimes u_z)) 
           = \varphi_{i-1} \cdots \varphi_{1}.(v_m \otimes u_z) $$
which is not zero.
A fortiori $Q(z)$ itself is non zero. It is thus a non zero polynomial function of degree less than  $N_m-1$ and the equality~\eqref{rmat} above shows that $s'_m$ is necessarily equal to its valuation. 

This implies  $s'_m \leq N_m-1$, and in particular $m \in \{k,n-k+1 \}$ by Corollary~\ref{Lambda}. This contradicts the hypothesis $m \geq n-k+2$. Thus we have shown $m=n-k+1$. 
  
  By iterating this we conclude that the Lyndon words ${\bf j}_n , \ldots , {\bf j}_{n-k+1}$ appear in this order in the canonical factorization of $\mu_k$, which is the desired statement. 
 
 \end{proof}

\smallskip

 From now on we write the canonical factorization of $\mu_k$ as $\mu_k = {\bf i}_0 \cdots {\bf i}_r $ with ${\bf i}_0 \geq \cdots \geq {\bf i}_r$.
For each $0 \leq j \leq r$ we denote by $p_j$ the first letter of the Lyndon word ${\bf i}_j$.The sequence $(p_j)_{0 \leq j \leq r}$ is decreasing, with $p_0 = p$ and $p_r = 1$ (the letter $1$ appears once, necessarily in the smallest of the ${\bf i}_j$). We also denote by $a_j$ the position of the letter $p_j$ in the word $\mu_k$.  

  \begin{rk} \label{littlerk}
 Note that as an immediate consequence of the previous lemma, one has  $r \geq k-1$. 
 \end{rk} 

With these notations we can make the following observation, as a straightforward consequence of Lemma~\ref{endLyndon}.

\begin{cor} \label{firstletter}
 The Lyndon word ${\bf i}_0$ ends with the letter $n$. In other words ${\bf i}_0 = {\bf j}_n$. 
 \end{cor}
 
 \begin{proof}
  As ${\bf i}_0$ is greater than any other Lyndon word appearing in the canonical factorization of $\mu_k$, in particular it is greater than ${\bf j}_n$. Hence by Lemma~\ref{endLyndon}(ii), one can write 
\begin{equation} \label{ineqLyndon}
{\bf i}_0 \geq {\bf j}_n >  \cdots > {\bf j}_{n-k+1}.
\end{equation} 
  Thus all of the Lyndon words ${\bf j}_n , \ldots , {\bf j}_{n-k+1}$ begin with a letter smaller than $p_0=p$. By definition they end with letters greater than or equal to $n-k+1$ which is greater than $p$ by Lemma~\ref{lem4} (recall that we assumed $k \leq n-k+1$ at the beginning of this section). Hence each of the Lyndon words ${\bf j}_n , \ldots , {\bf j}_{n-k+1}$ contains the letter $p$ which makes $k$ occurences of $p$. From Lemma~\ref{lem4} and Proposition~\ref{occurrences}, we conclude that no other Lyndon word contains $p$. Thus the Lyndon word ${\bf i}_0$ has to be one of the ${\bf j}_l$ and the inequalities~\eqref{ineqLyndon} impose ${\bf i}_0 = {\bf j}_n$. 
  \end{proof}

   \begin{lem} \label{lem3bis}

For any $1 \leq j \leq r$, if $p_j \neq k$, then $ p_{j-1} - p_j \leq 1$. 

 \end{lem}

 \begin{proof}

 Assume there exists $j \in \{1, \ldots , r \}$ such that $p_j \neq k$ and  $p_{j-1} - p_j \geq 2$. We set $q:= p_j$ and show that $s_q \leq N_q - 1$. 
 Let $i$ denote the position of $q$ in the word $\mu_k$. Then $h_i = q$; on the other hand all the letters in position $1, \ldots , i-1$ are greater than $q+2$ (as they are greater than $p_{j-1}$), hence 
 $$  \varphi_1 \cdots \varphi_{m_k}.(u \otimes (v_q)_z)  =  \tau_1 \cdots \tau_{i-1} (\tau_i (x_i - x_{i+1})+1)  \varphi_{i+1} \cdots \varphi_{m_k}.(u \otimes (v_q)_z)  . $$
  By similar arguments as in the proof of Lemma~\ref{endLyndon} (ii), this implies $s_q \leq N_q-1$. Hence by Corollary~\ref{Lambda}, $q \in \{k,n-k+1 \}$. Now by hypothesis $q \leq p_{j-1} -2 \leq p_0 -2 < p_0$ and $p_0 \leq n-k+1$ by Lemma~\ref{lem4}. In particular $q < n-k+1$. As by assumption $q=p_j \neq k$, we get the desired contradiction.

 \end{proof}

   \begin{prop} \label{lem5}
 With the previous notations, one has:
 
 \begin{enumerate}[(i)]
    \item $ p_1 < p_0$. 
    \item For all $j \geq 1$, if $p_j \neq k$ then  $p_{j+1} < p_j$.
\end{enumerate}

 \end{prop}

 \begin{proof}
Assume $p_1=p_0=p$ and let $i$  denote the position of the letter $p_1$ in the word $\mu_k$, i.e. $h_i=p_1$. First note that this implies $2 \leq p \leq n-1$ (as the letters $1$ and $n$ appear only once in the word $\mu_k$). In particular ${\bf i}_0$ is of length strictly greater than $2$, as ${\bf i}_0 = (p \ldots n)$ by Corollary~\ref{firstletter}. In other words $i>2$. 
One has:
 \begin{equation} \label{poly} 
  \varphi_1 \cdots \varphi_{m_k} . (u \otimes (v_p)_z) 
               = (\tau_1 (x_1 - x_2) + 1) \tau_2 \cdots \tau_{i-1} (\tau_i (x_i - x_{i+1}) + 1) \varphi_{i+1} \cdots \varphi_{m_k} . (u \otimes (v_p)_z).
   \end{equation}
   As the renormalized $R$-matrix $r_{L(p),M_k}$ does not vanish, the operator $\varphi_{i+1} \cdots \varphi_{m_k}$ acts as a non zero polynomial function on $(u \otimes (v_p)_z)$ . We set $Q(z) := \varphi_{i+1} \cdots \varphi_{m_k} . (u \otimes (v_p)_z)$. Note that $deg(Q) \leq N_p - 2$. Let $P(z)$ denote the polynomial function given by the term 
   $$ \tau_2 \cdots \tau_{i-1}  \varphi_{i+1} \cdots \varphi_{m_k} . (u \otimes (v_p)_z)   = \tau_2 \cdots \tau_{i-1} . Q(z) $$
   from the above equality and let us consider the action of the operator $\tau_{i-1} \cdots \tau_{1}$ on $P(z)$. Recall that $i>2$. We first write
 $$  \tau_2 \tau_1 . P(z) = \tau_2 \tau_1 . (\tau_2 \cdots \tau_{i-1}) Q(z) = (\tau_1 \tau_2 \tau_1 +1) . \tau_3 \cdots \tau_{i-1} \varphi_{i+1} \cdots \varphi_{m_k} . (u \otimes (v_p)_z) $$
using the braid relation. The operator $\tau_1$ commutes with $\tau_3 \cdots \tau_{i-1}$ as well as with $\varphi_{i+1} , \ldots , \varphi_{m_k}$. As ${\bf i}_0$ is of length greater than $2$, the action of $\tau_1$ on $(u \otimes (v_p)_z)$ is the same as its action on the cuspidal module $L({\bf i}_0)$ which is trivial. Hence we get
$$  \tau_2 \tau_1 . P(z) =  \tau_3 \cdots \tau_{i-1} \varphi_{i+1} \cdots \varphi_{m_k} . (u \otimes (v_p)_z). $$
 The letters $h_3 , \ldots  , h_{i-1}$ are greater than $p+2$ hence by arguments similar to the proof of Lemma~\ref{endLyndon} (ii), we get 
  $$\tau_{i-1} \cdots \tau_{3} . (\tau_3 \cdots \tau_{i-1}) . Q(z) = Q(z). $$  
Finally we get 
$$ \tau_{i-1} \cdots \tau_{1} . P(z) = \varphi_{i+1} \cdots \varphi_{m_k} . (u \otimes (v_p)_z)  = Q(z)$$
which is not zero. A fortiori $P(z)$ itself is a non zero polynomial function. All the other terms in equation~\eqref{poly} are either zero, either of valuation strictly greater than the valuation of $Q(z)$. This implies that $s_p$ is equal to the valuation of $Q(z)$. In particular $s_p \leq N_p - 2$. This contradicts Corollary~\ref{Lambda}. Hence $p_1<p_0$. 
 
 \bigskip 
 
For the second statement assume we have $j \geq 1$ such that  $ q:=p_{j+1}=p_j \neq k$ and  consider $j$ minimal for this property.  For the sake of simplicity, we only deal with the case where only the  two Lyndon words ${\bf i}_j$ and ${\bf i}_{j+1}$ begin with the letter $q$ i.e. $p_{j-1}>q, p_j = p_{j+1} =q $ and $p_{j+2} \leq q-1$. The proof is analogous if there are several words ${ \bf i}_{j'}$ beginning with $q$.

As by hypothesis $q \neq k$, Lemma~\ref{lem3bis} implies $p_{j-1} \in \{q,q+1 \}$ and hence $p_{j-1}=q+1$ as $p_{j-1}>q$.  Moreover by minimality of $j$, $q+1$ appears exactly once in the subsequence $(p_{i})_{i<j}$.

Set $a:= a_{j-1}$, $b:=a_j$ and $c:= a_{j+1}$. As above $N_q$  denotes the number of occurrences of $q$ in the word $w_k$.
 We write 
\begin{equation} \label{calcul}
  \varphi_1 \cdots \varphi_{m_k}.(u \otimes (v_q)_z)  = \tau_1 \cdots \tau_{b-1} (\tau_b (x_b - x_{b+1}) +1) \tau_{b+1} \cdots \tau_{c-1} (\tau_c (x_c - x_{c+1}) +1)  Q(z)
\end{equation}
where $Q(z) := \varphi_{c+1} \cdots \varphi_{m_k} . (u \otimes (v_q)_z) $. As the renormalized $R$-matrix $r_{L(q),M_k}$ does not vanish, $Q$ is a nonzero polynomial function. Its degree is equal to $N_q -2$. 

Now we prove that $Q(z)$ is in fact a monomial in $z$. Indeed, any occurrence of $q$ in a  position $i \in \{c+1 , \ldots , m_k \}$ appears inside a Lyndon word ${\bf i}_{j'}$ (with $j' > j+1$) beginning with a letter strictly smaller than $q$. Then the operator $\tau_{i-1}$ commutes with any $\varphi_h$ , $h > i$ and  acts by zero on $(u \otimes (v_q)_z)$ . Thus 
\begin{align*}
 \tau_{i-1} \varphi_{i} \varphi_{i+1} \cdots \varphi_{m_k}.(u \otimes (v_q)_z)  &= \tau_{i-1} \left( \tau_i (x_i - x_{i+1}) +1 \right) \varphi_{i+1} \cdots \varphi_{m_k}.(u \otimes (v_q)_z)  \\
           &= \tau_{i-1} \tau_i (x_i - x_{i+1}) \varphi_{i+1} \cdots \varphi_{m_k}.(u \otimes (v_q)_z)  \\
          &= z . \tau_{i-1} \tau_i  \varphi_{i+1} \cdots \varphi_{m_k}.(u \otimes (v_q)_z) 
\end{align*}
 up to some sign. 
This is valid for any occurrence of $q$ between the positions $c+1$ and $m_k$ and hence 
$$ Q(z) = \varphi_{c+1} \cdots \varphi_{m_k}.(u \otimes (v_p)_z) = z^{N_q -2} \tau_{c+1} \cdots \tau_{m_k}.(u \otimes (v_p)_z)  $$
is a monomial in $z$ of degree  $N_q -2$. 

The term $\tau_1 \cdots \tau_{b-1} \tau_{b+1} \cdots \tau_{c-1}  Q(z)$ coming from equation~\eqref{calcul} is necessarily zero: if it was not, then equation~\eqref{calcul} implies that  $s_q$ would be equal to $N_q - 2$, which contradicts Corollary~\ref{Lambda}.
There are two terms of degree $N_q -1$ in equation~\eqref{calcul}: $\tau_1 \cdots \tau_{c-1}  Q(z)$ and $\tau_1 \cdots \tau_{b-1} \tau_{b+1} \cdots  \tau_c Q(z)$. Denote them respectively by $A(z)$ and $B(z)$. 

We show that the operator $\tau_{c-1} \cdots \tau_1$ acts nontrivially on $A(z)$ and trivially on $B(z)$. This implies that $A(z) + B(z)$ cannot be zero and therefore there is a nonzero term of degree $N_q -1$ in equation~\eqref{calcul}.

\bigskip

 \textbf{Action of $\tau_{c-1} \cdots \tau_1$ on $A(z)$.}
Let us first look at the action of the operators $\tau_1 , \cdots , \tau_{a-1}$ on $A(z)$; if $j=1$ this is of course not necessary as $a=a_0=1$. Otherwise one has $j \geq 2$ and this action is easy to compute: for instance for $\tau_1$ one has: 
\begin{align*}
 \tau_1 . A(z) &= {\tau_1}^2 \tau_2 \cdots \tau_{c-1} Q(z)\\
    &= {\tau_1}^2 e(h_1,h_c,h_2,\ldots,h_{c-1},q,h_{c+1},\ldots,h_{m_k}) \tau_2 \cdots \tau_{c-1} Q(z) \\
    &= \tau_2 \cdots \tau_{c-1} Q(z) \text{ as $h_c=q$ and $h_1=p_0>p_1>q$ by minimiality of $j \geq 2$.}
\end{align*}
Simmilarly $h_2, \ldots,h_{a-1} \geq q+2$ and thus one gets 
$$ \tau_{a-1} \cdots \tau_1 . A(z) =  \tau_a \cdots \tau_{c-1}. Q(z)  .$$
Let us now look at the action of $\tau_a$:
\begin{align*}
 \tau_a . (\tau_{a} \cdots \tau_{c-1} . Q(z)) &= {\tau_a}^2 e(h_1,\ldots,h_a,h_c,h_{a+1},\ldots,h_{c-1},q,h_{c+1},\ldots,h_{m_k}) \tau_{a+1} \cdots \tau_{c-1} Q(z) \\
   &= (x_a - x_{a+1})  \tau_{a+1} \cdots \tau_{c-1} Q(z) \text{as $h_a=p_{j-1}=q+1$ and $h_c=p_{j+1}=q$} 
   \end{align*}
 The operator $x_a$ commutes  with  $\tau_{a+1} , \cdots \tau_{c-1}$ and acts trivially on the generating vector hence one gets
 \begin{align*}
   \tau_a . (\tau_{a} \cdots \tau_{c-1} . Q(z))  &=  -x_{a+1} \tau_{a+1} \cdots \tau_{c-1} . Q(z)   \\
    &= - x_{a+1} \tau_{a+1} \cdots \tau_b e(h_1,\ldots,h_b,h_c,h_{b+1},\ldots,h_{c-1},\ldots) \tau_{b+1} \cdots \tau_{c-1} Q(z) \\
    &=  -\tau_{a+1} \cdots \tau_{b-1} x_{b} \tau_{b} e(h_1,\ldots,h_b,h_c,h_{b+1},\ldots) \tau_{b+1} \cdots \tau_{c-1} . Q(z) \\
    & \qquad \qquad \qquad \qquad \qquad \qquad \qquad \qquad \text{($h_c=q$, $h_{a+1}, \ldots,h_{b-1} \geq q+2$)} \\
    &=  -\tau_{a+1} \cdots \tau_{b-1} (\tau_{b}x_{b+1} +1)  \tau_{b+1} \cdots \tau_{c-1} . Q(z) \text{($h_b=h_c=q$)}    \\
    &= - \tau_{a+1} \cdots \tau_{b-1} \tau_{b}   \tau_{b+1} \cdots \tau_{c-1} x_{c} . Q(z)
                    -  \tau_{a+1} \cdots \tau_{b-1}  \tau_{b+1} \cdots \tau_{c-1} . Q(z). 
\end{align*}

The operator $x_c$ acts trivially on $(u \otimes (v_q)_z)$  hence the first term of the right hand side in the last equality is zero. Now as  $h_{a+1} , \ldots , h_{b-1} \geq q+2$ and $h_b=q$, the action of the operator  $\tau_{b-1} \cdots \tau_{a+1}$ on the surviving term is similar to the action of $\tau_{a-1} \cdots \tau_1$ computed above. Hence we get 
$$ \tau_{b-1} \cdots \tau_{a} . (\tau_{a} \cdots \tau_{c-1} . Q(z)) = - \tau_{b+1} \cdots \tau_{c-1} . Q(z).$$
The situation is now similar to (i): using the braid relation, one can see that the action of $\tau_{b+1} \tau_{b}$ on $\tau_{b+1} \cdots \tau_{c-1}.Q(z)$ will give two terms, the only non-trivial one being $\tau_{b+2} \cdots \tau_{c-1}.Q(z)$. The letters $h_{b+2}, \ldots , h_{c-1}$ are all greater than $q+2$ and hence one concludes as before that $$ \tau_{c-1} \cdots \tau_{b+2} . (\tau_{b+2} \cdots \tau_{c-1}.Q(z)) = Q(z).$$ 

 Finally we have shown that $\tau_{c-1} \cdots \tau_1$ acts by identity on $A(z)$ (up to some sign).

 \textbf{Action of $\tau_{c-1} \cdots \tau_1$ on $B(z)$.}
  One again has:
$$ \tau_{a-1} \cdots \tau_1 . B(z) = \tau_{a-1} \cdots \tau_1 . (\tau_1 \cdots \tau_{b-1} \tau_{b+1} \cdots  \tau_{c} Q(z)) = \tau_{a} \cdots \tau_{b-1} \tau_{b+1} \cdots  \tau_{c} Q(z) .$$
But then 
 \begin{align*}
   \tau_{a}. (\tau_{a} \cdots \tau_{b-1} \tau_{b+1} \cdots  \tau_{c} Q(z) ) &= (x_a - x_{a+1}) \tau_{a+1} \cdots \tau_{b-1} \tau_{b+1} \cdots  \tau_{c} Q(z) \\
             &= \tau_{a+1} \cdots \tau_{b-1} x_{b} \tau_{b+1} \cdots  \tau_{c} Q(z) \\
\end{align*}
up to some sign, then the operator $x_{b}$ commutes with $\tau_{b+1} , \ldots  , \tau_{c} , \varphi_{c+1} , \ldots, \varphi_{m_k}$ and acts by zero on $(u \otimes (v_q)_z) $. 

Finally we have shown that the operator $\tau_{1} \cdots  \tau_{c-1}$ acts nontrivially on $A(z) + B(z)$ and in particular $A(z) + B(z) \neq 0$. Therefore $s_q = N_q -1$. Now, $q<p_0 \leq n-k+1$ by Lemma~\ref{lem4} and by assumption $q \neq k$ hence $q \notin \{k, n-k+1 \}$. By Corollary~\ref{Lambda}, this contradicts the inequality $s_q \leq N_q - 1$.

 In conclusion (ii) holds.   

\end{proof}

  \begin{cor} \label{cor1}

 The sequence $(p_j)$ takes exactly once every value $1 , \ldots , k-1$ and at least once the value $k$.

 \end{cor}

 \begin{proof}

The last term of the sequence $(p_j)$ is  $p_r=1$. Recall that $r \geq k-1$ (Remark~\ref{littlerk}). By (finite) induction on $t \in \{0, \ldots ,k-1 \}$ one shows that $p_{r-t}=t+1$. Indeed, if $k=1$ there is nothing to prove. If $k \geq 2$, assume $p_r = 1, \ldots ,p_{r-t}=t+1$ with $t<k-1$; then $p_{r-t} \leq k-1$ and Lemma~\ref{lem3bis} implies $p_{r-t-1} \leq p_{r-t} +1 $. If $p_{r-t-1} \neq k$ then Proposition~\ref{lem5} (ii) implies $p_{r-t-1} > p_{r-t}$ and thus $p_{r-t-1}=p_{r-t} +1$ which gives $p_{r-(t+1)} = t+1$. If $p_{r-t-1}=k$ then as $p_{r-t} \leq k-1$, necessarily one has $p_{r-t} = t = k-1$ and $p_{r-t-1} = k$ which again gives $p_{r-(t+1)} = t+1$. This implies that the sequence $(p_j)$ takes exactly once each value $1, \ldots , k-1$ (and at least once the value $k$).  

  \end{proof}

  \begin{cor} \label{cor2}

 In the case $p=k$, the parameter $\mu_k$ of $M_k$ is given by 
 $$ \mu_k = (k \ldots n)(k-1 \ldots n-1) \cdots (1 \ldots n-k+1).$$

 \end{cor}

 \begin{proof} 

By Proposition~\ref{lem5} (i), the sequence $(p_j)$ takes exactly once the value $k$. Together with Corollary~\ref{cor1}, we deduce that  the word $\mu_k$ has the form 
$$ \mu_k = (k \ldots)(k-1 \ldots) \cdots (1 \ldots).$$
Combining this with  Lemma~\ref{endLyndon} (i) and (ii),  we get the desired statement.

 \end{proof}

 \bigskip

One can now focus on the case $p_0 = n-k+1$.

   \begin{prop} \label{lem6}
  The sequence $(p_j)_{0 \leq j \leq r}$ takes exactly once every value between $n-k+1$ and $1$. In other words, $r=n-k+1$ and $p_j = n-k+1-j$ for all $1 \leq j \leq n-k+1$. 
 \end{prop}

 \begin{proof} 

By Corollary~\ref{cor1}, we already know that that the values $1 , \ldots , k-1$ are taken exactly once and the value $k$ at least once.

 { \bf Values $k+1, \ldots , n-k+1$.}
 Let  $ i := \max \{j , p_j > k \}$ (it exists as $p_0 = n-k+1 > k$) and $m := p_i$.

 Assume $m \geq k+2$. Then the commuting of $M_k$ with $L(k)$ implies that ${\bf i}_{i+1}$ is the only Lyndon word beginning with $k$ (equivalently $p_{i+1}=k , p_{i+2} = k-1 , \ldots , p_r = 1$). Indeed, all the letters in position strictly smaller than $a_{i+1}$ are greater than $k+2$ hence $\varphi_{i'} = \tau_{i'}$ for all $i' < a_{i+1}$ and 
$$ \tau_{a_{i+1}-1} \cdots \tau_1 . (\varphi_{1} \cdots \varphi_{m_k}.(u \otimes (v_k)_z)  ) = \varphi_{a_{i+1}} \cdots \varphi_{m_k}.(u \otimes (v_k)_z). $$
 Then the same proof as for Proposition~\ref{lem5}(i) shows that necessarily $ k=p_{i+1} > p_{i+2}$.

 Thus there is exactly one Lyndon word ${\bf i}_j$ beginning with every letter $1, \ldots ,k$. The letter $m-1$ does not appear in any of the words ${\bf i}_j$ for $j \leq i$ (all these words begin with letters greater than $m$) and appears exactly $k$ times in the word $\mu_k$ (as $n-k \geq m-1 \geq k+1$) hence it appears exactly once in each of the words ${\bf i}_{i+1}, \ldots,{\bf i}_r$. This implies that the last letters of all of these words are greater than $m-1$ and in particular so is $k$ (last letter of ${\bf i}_r$), i.e. $m \leq k+1$ which contradicts the hypothesis.  

   Hence $m=k+1$ i.e. the sequence $(p_j)$ takes all the values $n-k+1, \ldots ,1$. By Proposition~\ref{lem5}(ii) the values $n-k+1, \ldots ,k+1$ appear exactly once in the sequence $(p_j)$.

\bigskip

{\bf Value $k$.} 
 If there are more than two Lyndon words ${\bf i}_j$ beginning with the letter $k$ then the same proof as for Proposition~\ref{lem5}(ii) (it can be applied as $m=k+1$) implies $s_k \leq N_k-1$. But as the last letter of the word $\mu_k$ is $k$, the same proof as for Lemma~\ref{lem4} shows that $s'_k$ is also smaller than $N_k-1$. Hence both $s_k$ and $s'_k$ are less than $N_k - 1$  which contradicts Corollary~\ref{Lambda}.

 Therefore the sequence $(p_j)$ takes exactly once every value $n-k+1, \ldots ,1$.  

 \end{proof}

     \begin{cor} \label{props}
 For any $0 \leq j \leq n-k$, the Lyndon word ${\bf i}_j$ is $(n-k+1-j \ldots n-j)$.
   \end{cor}

 \begin{proof}
 We show it by induction on $j$. In fact we prove the following properties:
 \begin{enumerate}[(i)]
    \item For every $0 \leq j \leq n-k$ there is exactly one Lyndon word ending with each of the letters $n-j, \ldots ,n-k+1-j$.
    \item The Lyndon word ending with $n-j$ begins with the letter $n-k+1-j$.
\end{enumerate}

 For $j=0$ it follows from Corollary~\ref{firstletter}.
  
 Assume  (i) and (ii) hold until the rank $j$. By hypothesis the Lyndon words ending with the letters $n,n-1 \ldots ,n-j$ respectively begin with the letters $n-k+1, \ldots ,n-k+1-j$ and in particular do not contain $n-k+j$. As by Proposition~\ref{lem6} there is exactly one Lyndon word beginning with each of the letters $n-k+1, \ldots ,1$, the Lyndon words ending with letters $n-1-j, \ldots ,n-k+1-j$ begin with letters less than $n-k-j$ and hence contain $n-k-j$. This gives $k-1$ Lyndon words containing the letter $n-k-j$. As this letter appears exactly $k$ times in the word $\mu_k$, there exists a Lyndon word that contains $n-k-j$ but is not one of the previous words, i.e. does not end with any of the letters $n, \ldots ,n-j$. Hence it does not contain $n-k-j+1$ (as $n-k-j+1$ appears $k$ times, once  in each of the $k$ words ending with $n-j, \ldots ,n-k-j+1$). This means there is a unique Lyndon word ending with the letter $n-k-j$, which proves (i) at the rank $j+1$. 
 
 Now (ii) at rank $j$ and (i) at rank $j+1$ together with Proposition~\ref{lem6} easily imply (ii) at rank $j+1$. 

\end{proof} 

From Corollary~\ref{cor2} and Corollary~\ref{props}, one concludes that among the two simple modules $M_k$ and $M_{n-k+1}$, one of them has a parameter whose first letter is $k$, namely $(k \ldots n)(k-1 \ldots n-1) \cdots (1 \ldots n-k+1)$, and the other has a parameter whose first letter is $n-k+1$, namely $(n-k+1 \ldots n) \cdots (1 \ldots k)$.
 
  \subsection{Proofs of main theorems}
  \label{proofs}
 
 At this stage, one only has bijections between pairs of modules and pairs of dominant words: for each $1 \leq k \leq n/2$, the set of modules $\{M_k,M_{n-k+1} \}$ is in one-to-one correspondence with the set $\{(k \ldots n)(k-1 \ldots n-1) \cdots (1 \ldots n-k+1),(n-k+1 \ldots n) \cdots (1 \ldots k)\}$. A priori this yields two possibilities for each $k$. To complete the proof of Theorem~\ref{initpara}, we need to show that for every $1 \leq k \leq n/2$, one has
 $$ M_k = L \left( (k \ldots n)(k-1 \ldots n-1) \cdots (1 \ldots n-k+1) \right)$$ 
 and
 $$ M_{n-k+1} = L \left( (n-k+1 \ldots n)(n-k \ldots n-1) \cdots (1 \ldots k) \right). $$
 The key argument is the mutation rule for dominant words given by Proposition~\ref{mutpara}.

 \begin{proof}[Proof of Theorem~\ref{initpara}.]
 
  We prove by induction on $k \in \{1 , \ldots , n\}$ that
 
$   \begin{array}{ccl}
    M_{r_{k-1}+1}  &= &L(1 \ldots k) \\
   M_{r_{k-1}+2}  &= &L \left( (2 \ldots k)(1 \ldots k-1) \right) \\
    \ldots  & \ldots   &  \qquad \ldots \\
    M_{r_k} &= &L(k \cdots 1).   
\end{array}    $
 
   The result already holds for $k=1$ and $k=2$. 
 Consider $1 \leq k \leq n$ and assume the result holds at the rank $k$. 
 
  Let $ j \in \{ r_{k-1}+2, \ldots, r_{k}-1 \}$ and let us write the (ungraded) short exact sequence corresponding to the mutation in direction $j$:
  $$ 0 \rightarrow  M_{j_{+}} \circ M_{j-1} \circ M_{j_{-}+1}\rightarrow  M_j \circ {M_j}' \rightarrow  M_{j_{-}} \circ M_{j+1} \circ M_{j_{+}-1}  \rightarrow 0 . $$
 Let $p := j- r_{k-1}$.
 
By the induction hypothesis, one has
 
$ \begin{array}{ccl}
    M_j  &= &L \left( (p \ldots k) \cdots (1 \ldots k-p+1) \right) \\
   M_{j-1}  &= &L \left( (p-1 \ldots k) \cdots (1 \ldots k-p+2) \right) \\
    M_{j+1}  &=  &L \left( (p+1 \ldots k) \cdots (1 \ldots k-p) \right) \\
\end{array}    $

The Lyndon word $(p \ldots k)$  appears in the parameter of $M_j$ hence in the parameter of $ M_j \circ M'_j$. Hence by Proposition~\ref{mutpara}, it necessarily appears either in $\mu_{j_{+}} \odot \mu_{j-1} \odot \mu_{j_{-}+1} $ or in $\mu_{j_{-}} \odot \mu_{j+1} \odot \mu_{j_{+}-1}$. 
 Obviously, it does not appear in  $\mu_{j-1}$ nor in $\mu_{j+1}$. Moreover, $\mu_{j_{-}}$ and $ \mu_{j_{-}+1}$ do not contain the letter $k$ hence $(p \ldots k)$ does not appear in the canonical factorizations of these  parameters either.

 Now by Proposition~\ref{occurrences},  $\mu_{j_{+}}$ is either $ (p+1 \ldots k+1) \allowbreak  \cdots \allowbreak  (1 \ldots k-p) $ or $ (k-p+1 \ldots k+1) \allowbreak \cdots \allowbreak  (1 \ldots p+1) $ and  $\mu_{j_{+}-1}$ is either $(p \ldots k+1) \allowbreak \cdots \allowbreak (1 \ldots k-p+2) $ or $(k-p+2 \ldots k+1) \allowbreak \cdots \allowbreak (1 \ldots p) $. The only of these words in which the Lyndon $(p \ldots k)$ appears is $ (p+1 \ldots k+1) \allowbreak \cdots \allowbreak  (1 \ldots k-p) $ and thus $\mu_{j_{+}} = \mu_{r_k + p+1} =  (p+1 \ldots k+1)  \allowbreak  \cdots  \allowbreak  (1 \ldots k-p) $.

 One can do this for any  $j \in \{ r_{k-1}+2, \ldots ,  r_k -1 \}$, and the same arguments hold for $j=r_{k-1}+1$ and $j=r_k$. 
 Thus the desired result  holds at rank $k+1$. 

 \end{proof}

One can now prove Theorem~\ref{mainthm}.

 \begin{proof}[Proof of Theorem~\ref{mainthm}.]

 We begin by describing the exchange matrix corresponding to the quiver given in \cite[Definition 11.1.1]{KKKO}.
For any $1 \leq k \leq n-1$ define the following matrices:
$$
 A_k :=  \left( \begin{array}{ccccc}

 0 & 1 & \cdots & \cdots  & 0 \\
 -1 & \ddots & \ddots  & {} & \vdots \\
 0 & \ddots & \ddots & \ddots  & \vdots \\
 \vdots & \ddots &  \ddots  & \ddots & 1 \\
 0  & \cdots &  0 & -1 & 0 \\

\end{array}
 \right) ,   \quad     B_k :=   \left( \begin{array}{ccccc}

 -1 & 0 & \cdots & \cdots  & 0 \\
 1 & -1 & \ddots  & {} & \vdots \\
 0& \ddots & \ddots & \ddots  & \vdots \\
 \vdots & \ddots &  \ddots  & \ddots & 0 \\
 0  & \cdots & 0 & 1 & -1 \\
 0 & \cdots & \cdots &  \cdots  0 & 1

\end{array}
 \right)     $$
 and $$ C_k := - ~^t B_{k-1} $$
 of respective sizes $k \times k$,  $k+1 \times k$, and $k-1 \times k$.
  
  Now the whole exchange matrix can be written by blocks as follows: 
 $$   \left( \begin{array}{cccccc}
  
  A_1  &   C_2  &    0   & \cdots & \cdots &  0 \\
  B_1  &   A_2  &   C_3  & \ddots &   {}   & \vdots \\
   0   &   B_2  &   A_3  &   C_4  & \ddots & \vdots \\
\vdots & \ddots & \ddots & \ddots & \ddots & \vdots  \\
\vdots &   {}   &    0   & \ddots & \ddots &  C_{n-1} \\ 
   0   & \cdots & \cdots &   0    & B_{n-2} & A_{n-1} \\
   0   & \cdots & \cdots & \cdots &    0   &  B_{n-1}

  \end{array}
  \right). $$ 
Recall that for any parameter $\mu \in \M$, $\mu^{\odot -1}$ denotes the inverse of $\mu$ in the Grothendieck group $\G$ of $\M$.  We can now compute the parameters $\hat{\mu_j}$ associated to the $\hat{y_j}$ as in Definition~\ref{compat}. 
For instance, for any $2 \leq j \leq n-2$, 
\begin{multline}
  \hat{\mu}_{r_{n-2}+j} = \left( (j-1 \ldots n-2) \cdots (1 \ldots n-j) \right)^{\odot -1} \odot  \left( (j \ldots n-2) \cdots (1 \ldots n-j-1) \right)  \nonumber\\
                    \odot  \left( (j-1 \ldots n-1) \cdots (1 \ldots n-j+1) \right) \odot  \left( (j+1 \ldots n-1) \cdots (1 \ldots n-j-1) \right)^{\odot -1}  \nonumber \\
                    \odot \left( (j \ldots n) \cdots (1 \ldots n-j+1) \right)^{\odot -1} \odot  \left( (j+1 \ldots n) \cdots (1 \ldots n-j) \right)
\end{multline}
which simplifies as 
  $$ \hat{\mu}_{r_{n-2}+j} = \left( (j+1 \ldots n)(j \ldots n-1) \right) \odot \left( (j+1 \ldots n-1)(j \ldots n) \right)^{\odot -1} . $$
Hence for any parameter $\mu \in \M$, one has
 \begin{align*}
   \left( (j+1 \ldots n-1)(j \ldots n) \right) \odot \hat{\mu}_{r_{n-2}+j} \odot \mu &= \left( (j+1 \ldots n)(j  \ldots n-1) \right) \odot \mu \\
                                                        &> \left( (j+1 \ldots n-1)(j \ldots n) \right) \odot \mu.
\end{align*}
This exactly means $ \hat{\mu}_{r_{n-2}+j} \odot \mu > \mu $ in $\G$ for any $\mu \in \M$. The computations for any other index $s \in \{1, \ldots , r_{n-1} \} $ are similar. 

 Using Remark~\ref{rkpsi} we conclude that the seed ${\s}_0^n$ is compatible.
 
 \end{proof}

 \section{Possible further developments}
 
  In this section we mention a couple of situations where interesting consequences may arise from the study of compatible seeds in various contexts of monoidal categorifications of cluster algebras.
 
  \subsection{Dominant words and $g$-vectors}
  
    By Theorem~\ref{mainthm}, the seed ${\s}_0^n$ for the category $R-gmod$ in type $A_n$ is compatible in the sense of Definition~\ref{compat}. As explained in Subsection~\ref{compatseed}, this yields some interesting combinatorial relationships between dominant words and $g$-vectos. 
    
   More precisely, consider as in Subsection~\ref{compatseed} $x_l^t$ any cluster variable in $\A$, $M_l^t$ the simple module in $R-gmod$ such that $[M_l^t]=x_l^t$ and $\mu_l^t$ the dominant word associated to $M_l^t$. For simplicity, we will write $x$ (resp. $M,\mu$) for $x_l^t$ (resp. $M_l^t,\mu_l^t$) without ambiguity as we will focus here on this module.  For any dominant Lyndon word (i.e. any positive root in type $A_n$) $(k \ldots l)$, we let $m_{(k \ldots l)}$ denote the multiplicity of the Lyndon word $(k \ldots l)$ in the canonical factorization of $\mu$. 
   As in Section~\ref{algclust}, we consider $F$ and ${\bf g}= (g_1 , \ldots , g_{r_{n-1}})$  the $F$-polynomial and the $g$-vector  associated to $x$. We also  let  $a_1, \ldots , a_{r_{n-1}}$ denote the exponents of the unique monomial of maximal degree of $F$ (see Theorem~\ref{FpolyDWZ}(i)), and $c_1 , \ldots , c_n$ the (negative) integers such that  $F_{|{\p}}(y_1, \ldots, y_n) =  x_{r_{n-1}+1}^{c_1} \cdots x_{r_n}^{c_n}$ (see Subsection~\ref{compatseed}).
 
  \bigskip
 
First consider the positive roots ending with the letter $n$. It follows from Theorem~\ref{initpara} that these Lyndon words do not appear in the dominant words associated to the unfrozen variables of the seed ${\s}_0^n$. Hence the $g$-vectors will not be involved. Moreover, for any $ 1 \leq j \leq n$, the Lyndon word $(j \ldots n)$ appears in exactly one of the dominant words corresponding to the frozen variables of ${\s}_0^n$ namely $(j \ldots n) \cdots (1 \ldots n-j+1)$. If $2 \leq j \leq n-1$, then the Lyndon word $(j \ldots n)$ appears in exactly two of the $\mjh$ namely $\hat{\mu}_{r_{n-1}+j-1}$ and $\hat{\mu}_{r_{n-1}+j}$. The relations are the following:
$$ m_{(j \ldots n)} = a_{r_{n-2}+j-1} - a_{r_{n-2}+j}  - c_j. $$
For positive roots of the form $(j \ldots n-1)$ with $2 \leq j \leq n-2$, similar arguments show that 
$$ m_{(j \ldots n-1)} =  a_{r_{n-2}+j} - a_{r_{n-2}+j+1}  - c_{j+1}    -a_{r_{n-2}+j-1} + a_{r_{n-2}+j+1} + a_{r_{n-3}+j-1} - a_{r_{n-3}+j} + g_{r_{n-2}+j}  $$
which simplifies as
$$ m_{(j \ldots n-1)} =  a_{r_{n-2}+j} - c_{j+1} -a_{r_{n-2}+j-1} + a_{r_{n-3}+j-1} - a_{r_{n-3}+j} + g_{r_{n-2}+j}. $$

\bigskip
\bigskip

 \subsection{The coherent Satake category}
 
 Recently, Cautis-Williams \cite{CW} exhibited  a new example of monoidal categorification of cluster algebras, using the coherent Satake category. In this subsection we focus on the case of the general linear group $GL_n$. We begin by checking that Assumptions~\ref{decomp} and ~\ref{assump} hold in the framework of \cite{CW}. 

The simple objects in the coherent Satake category are parametrized (up to $\mathbb{G}_m$-equivariant shift)  by couples of a coweight and a weight, modulo action of the Weyl group. Equivalently they can be parametrized by \textit{dominant pairs}, i.e. couples of a dominant coweight $\lambda^{\vee}$ together with a weight $\mu$ dominant for the Levi factor of $P_{\lambda^{\vee}}$. Denote by $\mathcal{P}_{\lambda^{\vee},\mu}$  the simple perverse coherent sheaf corresponding to a dominant pair  $(\lambda^{\vee},\mu)  \in P^{\vee} \times P$. Then the following statement shows that Assumption~\ref{decomp} holds:
  
   \begin{prop}[{{\cite[Proposition 2.6]{CW}}}]
    Let $\mathcal{P}_{\lambda_1^{\vee},\mu_1}$ and $\mathcal{P}_{\lambda_2^{\vee},\mu_2}$ be two simple objects in the coherent Satake category. Then in its Grothendieck ring $K^{G(\mathcal{O}) \rtimes \mathbb{G}_m}(Gr_G)$ one has:
     $$ [\mathcal{P}_{\lambda_1^{\vee},\mu_1} \ast \mathcal{P}_{\lambda_2^{\vee},\mu_2}] 
               =  q^s  [\mathcal{P}_{\lambda_1^{\vee} + \lambda_2^{\vee}, \mu_1 + \mu_2}] 
                        + \sum_{(\lambda^{\vee},\mu) \in S} p_{\lambda^{\vee},\mu}  [\mathcal{P}_{\lambda^{\vee},\mu}]   $$
    where $s$ is some integer depending on $\lambda_1, \mu_1, \lambda_2, \mu_2$,   $p_{\lambda^{\vee},\mu}  \in \mathbb{Z}[q^{\pm 1/2}]$, and   $S$ is a finite collection of dominant pairs such that for every $ (\lambda^{\vee},\mu) \in S$, one has either $\lambda^{\vee} < \lambda_1^{\vee} + \lambda_2^{\vee}$, or $\lambda^{\vee} = \lambda_1^{\vee} + \lambda_2^{\vee}$ and ${\mid \mid \mu \mid \mid}^2 \leq {\mid \mid  \mu_1 \mid \mid}^2 + {\mid \mid  \mu_2 \mid \mid}^2$ for any $W$-invariant quadratic form ${\mid \mid \cdot \mid \mid}^2$.   
   \end{prop}
   
 Taking the lexicographic order on (dominant) pairs $(\lambda^{\vee},\mu)  \in P^{\vee} \times P$, the monoid structure on the set of dominant pairs can be simply taken as 
$$ (\lambda_1^{\vee},\mu_1) \odot (\lambda_2^{\vee},\mu_2) = (\lambda_1^{\vee} + \lambda_2^{\vee} , \mu_1 + \mu_2). $$
It is then clear that Assumption~\ref{assump} also holds. 
 
In the case of the general linear group $GL_n$, Cautis-Williams explicitly describe a monoidal seed in the coherent Satake category. However, this seed is not compatible in the sense of Definition~\ref{compat} above. For instance, for $GL_2$, this seed can be written as 
$$ (( [\mathcal{P}_{1,0}] , [\mathcal{P}_{1,1}] , [\mathcal{P}_{2,0}] , [\mathcal{P}_{2,1}]), B)$$
where the first two classes are the unfrozen variables and the last two are the frozen variables, and the exchange matrix  $B$ is given by:
$$ B = 
    \begin{pmatrix}
      0 & -2\\
      2 & 0\\
      0 & 1\\
      -1 & 0
    \end{pmatrix}. $$
Recall from \cite[Section 2.2]{CW} that $\mathcal{P}_{k,l}$ stands for $\mathcal{P}_{\omega_k^{\vee}, l \omega_k}$ for any $1 \leq k \leq 2$ and any $l \in \{0,1\}$. 

 One can now compute the generalized parameters $\hat{\mu_1}$ and $\hat{\mu_2}$ for this seed. A straightforward computation gives $\hat{\mu_1} = (2 \omega_1^{\vee} - \omega_2^{\vee} ,  2 \omega_1 - \omega_2)$ and $\hat{\mu_2}=(\omega_2^{\vee} - 2\omega_1^{\vee},0)$. The coweight $2 \omega_1^{\vee} - \omega_2^{\vee}$ is exactly the coroot $\alpha_1^{\vee}$, and hence for any dominant pair $(\lambda^{\vee},\mu)$ one has $\hat{\mu_1} \odot (\lambda^{\vee},\mu) \geq (\lambda^{\vee},\mu)$. However, the coweight part of $\hat{\mu_2}$ is obviously the opposite of  $\alpha_1^{\vee}$ and thus $\hat{\mu_2} \odot (\lambda^{\vee},\mu) \leq (\lambda^{\vee},\mu)$ for any dominant pair $(\lambda^{\vee},\mu)$. 
  We conclude that this seed is not compatible. 
 
  \smallskip
  
  It would be interesting to see if Conjecture~\ref{conjecture} holds in the coherent Satake category of the general linear group. Note that as the ordering on dominant pairs in partial, it is not clear that one can formulate mutation rules for parameters as in Section~\ref{mutrule}. Indeed, we crucially used the fact that the ordering on dominant words parametrizing simple modules over quiver Hecke algebras is total. This mutation rule allows to compute explicitly as many seeds as we want from the data of an initial seed. In the case of a partial ordering, we cannot do so a priori.

\addcontentsline{toc}{section}{References}

\Address

\end{document}